\documentclass{amsart}
\usepackage{amsmath}
  \usepackage{paralist}
  \usepackage{amssymb}
 \usepackage{amsthm}
\usepackage[colorlinks=true]{hyperref}
\hypersetup{urlcolor=blue, citecolor=red}

\newtheorem{theorem}{Theorem}[section]

\newtheorem{lemma}[theorem]{Lemma}
\newtheorem{proposition}[theorem]{Proposition}
\theoremstyle{definition}
\newtheorem{definition}[theorem]{Definition}
\newtheorem{remark}[theorem]{Remark}

\numberwithin{equation}{section}

\title[MFG model with mixed interactions]
      {Existence and multiplicity of solutions to the mean-field games model with mixed interactions}

\author[Xinfu Li, Xiangqing Liu,  Juncheng Wei \& Yuanze Wu]{}


\thanks{*Corresponding author.}

\thanks{Email Addresses:  lxylxf@tjcu.edu.cn; liuxiangqing@ynnu.edu.cn; wei@math.cuhk.edu.hk; yuanze.wu@ynnu.edu.cn.}

\begin{document}
\maketitle

\centerline{\scshape   Xinfu Li$^{a}$, Xiangqing Liu$^{b,c}$, Juncheng Wei$^{d}$ \& Yuanze Wu$^{b,c*}$}

\medskip
{\footnotesize
\centerline{$^a$School of Science, Tianjin University of Commerce, Tianjin, 300134,
China}
\centerline{$^b$School of Mathematics, Yunnan Normal University, Kunming, 650500, China}

\centerline{$^c$Yunnan Key Laboratory of Modern Analytical Mathematics and Applications, Kunming, 650500, China}

\centerline{$^d$Department of Mathematics, Chinese University of Hong Kong,
Shatin, NT, Hong Kong}

\bigskip
 %\centerline{(Communicated by the associate editor name)}

\begin{abstract}
In this paper, we consider the stationary version of the Mean-Field Games (MFG) models.  Inspired by \cite{Albuquerque-Silva2020, Bieganowski-Mederski2021, Lin-Wei05, Mederski-Schino2021}, we develop the minimization method on the Pohozaev manifold introduced in \cite{Soave20JDE, Soave20JFA} for the existence theory of the stationary version of the Mean-Field Games (MFG) models with $2$-homogeneous hamiltonians and mixed interactions.  As applications, we prove the existence and multiplicity of radial solutions of the Mean-Field Games (MFG) models with general $p$-homogeneous hamiltonians and mixed interactions under more general conditions, some of which are even new for $2$-homogeneous hamiltonians.  We hope that our techniques and ideas introduced in this paper would be helpful in understanding the optimal value of the total mass in the existence theory of radial solutions to the Mean-Field Games (MFG) models with general $p$-homogeneous hamiltonians and mixed interactions, as well as that of other models.

\vskip0.12in

\medskip

\end{abstract}

\section{Introduction}

\setcounter{section}{1}
\setcounter{equation}{0}

\subsection{Background.}\quad
In this paper, we consider the Mean-Field Games (MFG) models, which are introduced by Lasry and Lions in \cite{Lasry-Lions} and Huang et al. in \cite{Huang-Malhame-Caines} independently around twenty years ago.  These models read as
\begin{eqnarray}\label{MFG}
\left\{\aligned
&v_t=-\Delta v+H\left(\nabla v\right)-V(x)-f(m), \quad x\in\mathbb{R}^N,\ \ t>0,\\
&m_t=\Delta m+\nabla\cdot\left(\nabla H\left(\nabla v\right) m\right), \quad x\in\mathbb{R}^N,\ \ t>0,\\
&v_{t=T}=v_T,\quad m_{t=0}=m_0,
\endaligned\right.
\end{eqnarray}
where $m$ represents the population density, $v$ denotes the value function of a typical player, $m_0$ is the initial data of density, $v_{T}$ is the terminal data of value function $v$, $H: \mathbb{R}^N\to \mathbb{R}$ is the Hamiltonian, $V(x)$ denotes a potential function and $f$ is the interaction between agents.  The stationary version of \eqref{MFG}, which models the long time (collective) behavior of the
population of players in \eqref{MFG}, is given by
\begin{eqnarray}\label{sMFG}
\left\{\aligned
&-\Delta v+H\left(\nabla v\right)+\lambda=V(x)+f(m), \quad x\in\mathbb{R}^N,\\
&-\Delta m=\nabla\cdot\left(\nabla H\left(\nabla v\right) m\right), \quad x\in\mathbb{R}^N,\\
&\int_{\mathbb{R}^N}m(x)dx=M>0,
\endaligned\right.
\end{eqnarray}
where $\lambda>0$ is a part of unknowns as a Lagrange multiplier and $M>0$ is the total mass of population.

\vskip0.12in

We are interested in the radial solution of \eqref{sMFG}.  In this setting, it is pointed out in \cite{Cirant-Kong-Wei-Zeng} that \eqref{sMFG} can be reduced to a single equation under some suitable choice of the Hamiltonian $H$.  Let
\begin{eqnarray}\label{Hamiltonian}
H(t)=C_H|t|^{p'},\quad p'>1\text{ and }C_H>0,
\end{eqnarray}
then the second equation of \eqref{sMFG}, i.e., the Fokker-Planck equation, reads as
\begin{eqnarray*}
\nabla m+p'm C_H|\nabla v|^{p'-2}\nabla v=0,\quad\text{a.e. in }\mathbb{R}^N
\end{eqnarray*}
either in the radial setting or $p'=2$.
If we further write $u:=m^{\frac{1}{p}}$, where $p$ is the conjugate number of $p'$, then the first equation of \eqref{sMFG}, i.e., the Hamilton-Jacobi equation, reads as
\begin{equation}\label{pSchrodinger}
\begin{cases}
-\varrho\Delta_pu+V(x)u^{p-1}=\lambda u^{p-1}+f\left(u^{p}\right)u^{p-1},\ \text{in}\  \mathbb{R}^N,\\
u(x)>0,\ \text{in}\  \mathbb{R}^N,\\
\int_{\mathbb{R}^N}|u(x)|^pdx=M,
\end{cases}
\end{equation}
where $\Delta_pu=\text{div}(|\nabla u|^{p-2}\nabla u)$ is the $p$-Laplacian operator with $1<p<N$ and $\varrho=\left(\frac{p-1}{C_H}\right)^{p-1}$.

\vskip0.12in

It is known that if the cost of \eqref{sMFG} is monotone decreasing then it may admit many solutions.  Moreover, the analysis of existence becomes more challenging if the cost is further unbounded.  In particular, it is proved in the recent papers \cite{Soave20JDE, Soave20JFA} by the variational methods that if $N\geq3$, $p'=2$, $V(x)=0$ and $f(t)=\mu t^{\frac{q_1}{2}-1}+t^{\frac{q_2}{2}-1}$ with $2<q_1<2+\frac{4}{N}<q_2\leq 2^*=\frac{2N}{N-2}$ and $\mu>0$, then \eqref{pSchrodinger} has two radial solutions for $q_2<2^*$ and one radial solution for $q_2=2^*$ under some suitable assumptions on $\mu$ and $a$, where $2^*$ is the Sobolev critical exponent.  To precisely state them, we need to introduce some necessary notations and definitions.  Without loss of generality, we first rewrite \eqref{pSchrodinger} into the following $p$-Laplacian equation with the $L^p$-norm constraint
\begin{equation}\label{e1.1}
\begin{cases}
-\Delta_pu=\lambda |u|^{p-2}u+\mu |u|^{q_1-2}u+|u|^{q_2-2}u,\ \text{in}\  \mathbb{R}^N,\\
u(x)>0,\ \text{in}\  \mathbb{R}^N,\\
\int_{\mathbb{R}^N}|u(x)|^pdx=a^p,\ u\in W^{1,p}\left(\mathbb{R}^N\right)
\end{cases}
\end{equation}
by a suitable scaling,
where $a, \mu>0$ are parameters, $p<q_1<p+\frac{p^2}{N}<q_2\leq p^*=\frac{pN}{N-p}$ with $p^*$ the Sobolev critical exponent, $\lambda\in \mathbb{R}$ is a part of unknowns which serves as a Lagrange multiplier and $W^{1,p}\left(\mathbb{R}^N\right)$ is the Sobolev space given by
\begin{eqnarray*}
W^{1,p}\left(\mathbb{R}^N\right)=\left\{u\in L^{p}\left(\mathbb{R}^N\right)\mid \int_{\mathbb{R}^N}|\nabla u|^pdx<+\infty\right\}.
\end{eqnarray*}
We next recall the
Sobolev and Gagliardo-Nirenberg inequalities.  For any $N\geq 2$ and $1<p<N$, there exists an
optimal constant $S > 0$ depending only on $N$ and $p$ such that
\begin{equation*}\label{e1.5}
S\|u\|_{p^*}^p\leq \|\nabla u\|_p^p,\ \forall u\in D^{1,p}\left(\mathbb{R}^N\right),\ \  \text{(Sobolev\  inequality)}
\end{equation*}
where $D^{1,p}\left(\mathbb{R}^N\right)$ denotes the completion of $C_0^{\infty}\left(\mathbb{R}^N\right)$
with respect to the norm $\|u\|_{D^{1,p}}:=\|\nabla u\|_p$ and $\|\cdot\|_p$ is the usual norm in the Lebesgue space $L^p\left(\mathbb{R}^N\right)$ for $p>1$.
It is well known (cf. \cite{Talenti 1976}) that the optical constant $S$ is attained by
\begin{equation}\label{e2.1}
U_{\epsilon,y}=d_{N,p}\epsilon^{\frac{N-p}{p(p-1)}}\left(\epsilon^{\frac{p}{p-1}}+|x-y|^{\frac{p}{p-1}}\right)^{\frac{p-N}{p}},
\end{equation}
where $\epsilon>0$, $y\in \mathbb{R}^N$, $d_{N,p}>0$ depends on $N$ and $p$ such that $U_{\epsilon,y}$ satisfies the equation
\begin{equation*}
\begin{cases}
-\Delta_pu=u^{p^*-1},\  \text{in}\ \mathbb{R}^N,\\
u(x)>0,\  \text{in}\ \mathbb{R}^N
\end{cases}
\end{equation*}
and
\begin{equation*}
\|\nabla U_{\epsilon,y}\|_p^p=\|U_{\epsilon,y}\|_{p^*}^{p^*}=S^{\frac{N}{p}}.
\end{equation*}
If $q\in (p,p^*)$, we recall that there exists an optimal constant
$C_{N,p,q}$ depending on $N,p$ and $q$ such that
\begin{equation*}\label{e1.7}
\|u\|_{q}\leq C_{N,p,q}\|\nabla u\|_p^{\gamma_q}\|u\|_p^{1-\gamma_q},\ \forall u\in W^{1,p}\left(\mathbb{R}^N\right),\ \  \text{(Gagliardo-Nirenberg \ inequality)}
\end{equation*}
where
\begin{equation*}\label{e1.6}
\gamma_{q}:=\frac{N(q-p)}{pq}=\frac{N}{p}-\frac{N}{q}
\end{equation*}
with
\begin{equation}\label{e1.8}
\gamma_{q}
\left\{\begin{array}{ll}
<1,&\ \text{if}\ p<q<p^*,\\
=1,&\ \text{if}\ p=p^*,\\
\end{array}
\right.
\quad\text{and}\quad
q\gamma_{q}
\left\{\begin{array}{ll}
<p,&\ \text{if}\ p<q<p+\frac{p^2}{N},\\
=p,&\ \text{if}\ q=p+\frac{p^2}{N},\\
>p,&\ \text{if}\ p+\frac{p^2}{N}<q\leq p^*.\\
\end{array}
\right.
\end{equation}
It is also well known that solutions of \eqref{e1.1} corresponds to critical points of the following energy functional
\begin{equation}\label{e1.2}
\Psi_{\mu}(u):=\frac{1}{p}\|\nabla u\|_p^p-\frac{\mu}{q_1}\|u\|_{q_1}^{q_1}-\frac{1}{q_2}\|u\|_{q_2}^{q_2}
\end{equation}
constrained on the $L^p$-sphere
\begin{equation}\label{e1.4}
\mathcal{S}_a:=\left\{u\in W^{1,p}\left(\mathbb{R}^N\right)\mid\|u\|_p=a\right\}.
\end{equation}
Clearly, $\Psi_{\mu}(u)$ is of class $C^1$.  Since $q_2>p+\frac{p^2}{N}$, it is well known that $\Psi_{\mu}(u)$ is unbounded on $\mathcal{S}_a$.  Thus, a natural constraint to find critical points of $\Psi_{\mu}(u)$ is the following $L^p$-Pohozaev
manifold, which is introduced in \cite{Bartsch-Soave2017, Jeanjean97, Soave20JDE, Soave20JFA}:
\begin{equation}\label{e1.3}
\mathcal{P}_{a,\mu}:=\left\{u\in \mathcal{S}_a \mid \|\nabla u\|_p^p=\mu\gamma_{q_1}\|u\|_{q_1}^{q_1}+\gamma_{q_2}\|u\|_{q_2}^{q_2}\right\}.
\end{equation}
By using the fibering maps
\begin{equation}\label{e1.9}
\Phi_{\mu,u}(s):=\Psi_{\mu}\left((u)_s\right)=\frac{1}{p}s^p\|\nabla u\|_p^p-\frac{\mu}{q_1}s^{q_1\gamma_{q_1}}\|u\|_{q_1}^{q_1}-\frac{1}{q_2}s^{q_2\gamma_{q_2}}\|u\|_{q_2}^{q_2},
\end{equation}
which is introduced by Jeanjean in \cite{Jeanjean97}, $\mathcal{P}_{a,\mu} $ can be naturally divided into three parts:
\begin{equation*}\label{e1.10}
\begin{split}
\mathcal{P}_{a,\mu}^+:=\left\{u\in \mathcal{P}_{a,\mu} \mid p\|\nabla u\|_p^p>\mu q_1\gamma_{q_1}^2\|u\|_{q_1}^{q_1}+q_2\gamma_{q_2}^2\|u\|_{q_2}^{q_2}\right\},\\
\mathcal{P}_{a,\mu}^0:=\left\{u\in \mathcal{P}_{a,\mu} \mid p\|\nabla u\|_p^p=\mu q_1\gamma_{q_1}^2\|u\|_{q_1}^{q_1}+q_2\gamma_{q_2}^2\|u\|_{q_2}^{q_2}\right\},\\
\mathcal{P}_{a,\mu}^-:=\left\{u\in \mathcal{P}_{a,\mu} \mid p\|\nabla u\|_p^p<\mu q_1\gamma_{q_1}^2\|u\|_{q_1}^{q_1}+q_2\gamma_{q_2}^2\|u\|_{q_2}^{q_2}\right\},
\end{split}
\end{equation*}
where $(u)_s:=s^{\frac{N}{p}}u(sx)$ for $x\in \mathbb{R}^N$ and $s>0$ is the trajectory on $\mathcal{S}_a$ generated by $u$.
Let
\begin{equation}\label{e1.11}
m^{\pm}(a,\mu):=\inf_{u\in \mathcal{P}_{a,\mu}^{\pm}}\Psi_{\mu}(u).
\end{equation}
Then following \cite{Soave20JDE, Soave20JFA}, we say that $\hat{u}$ is a ground-state solution of the equation (\ref{e1.1}), if $\hat{u}$ solves (\ref{e1.1}) for
a suitable $\lambda\in  \mathbb{R}$ and $\Psi_{\mu}(\hat{u})=\inf_{u\in \mathcal{P}_{a,\mu}}\Psi_{\mu}(u)$.  We say that $\tilde{u}$ is a mountain-pass solution of the equation (\ref{e1.1}), if $\tilde{u}$ solves
(\ref{e1.1}) for a suitable $\lambda \in \mathbb{R}$ and $\Psi_{\mu}(\tilde{u})=m^{-}(a,\mu)$.
Now, Soave' theorems in \cite{Soave20JDE, Soave20JFA} can be stated as follows.
\begin{theorem}\label{thmSoave20JDE}
(\cite[Theorem~1.3]{Soave20JDE})\quad Let $N\ge 3$, $2<q_1<2+\frac{4}{N}<q_2<2^*$ and $a,\mu>0$.  If
\begin{equation*}
\frac{\left(\mu a^{q_1(1-\gamma_{q_1})}\right)^{q_2\gamma_{q_2}-2}\left(a^{q_2(1-\gamma_{q_2})}\right)^{2-q_1\gamma_{q_1}}}{\left(\frac{q_2(2-q_1\gamma_{q_1})}{2C_{N,q_2}^{q_2}(q_2\gamma_{q_2}-q_1\gamma_{q_1})}\right)^{2-q_1\gamma_{q_1}}\left(\frac{q_1(q_2\gamma_{q_2}-2)}{2C_{N,q_1}^{q_1}(q_2\gamma_{q_2}-q_1\gamma_{q_1})}\right)^{q_2\gamma_{q_2}-2}}<1,
\end{equation*}
where $C_{N,q}:=C_{N,2,q}$ is the optimal constant of the Gagliardo-Nirenberg inequality, then the equation~\eqref{e1.1} in the semilinear case $p=2$ has a ground-state solution $\hat{u}_{a,\mu}$ and a mountain-pass solution $\tilde{u}_{a,\mu}$.
\end{theorem}

\begin{theorem}\label{thmSoave20JFA}
(\cite[(1) of Theorem~1.1]{Soave20JFA})\quad Let $N\ge 3$, $2<q_1<2+\frac{4}{N}$, $q_2=2^*$ and $a,\mu>0$.  Then there exists $\alpha_{N,q_1}=\min\{C_1,C_2\}$, where
\begin{eqnarray*}
C_1=\left(\frac{2^*S^{\frac{2^*}{2}}(2-q_1\gamma_{q_1})}{2(2^*-q_1\gamma_{q_1})}\right)^{\frac{2-q_1\gamma_{q_1}}{2^*-2}}\frac{q_1(2^*-2)}{2C_{N,q_1}^{q_1}(2^*-q_1\gamma_{q_1})}
\end{eqnarray*}
and
\begin{eqnarray*}
C_2=\frac{22^*}{N\gamma_{q_1}C_{N,q_1}^{q_1}(2^*-q_1\gamma_{q_1})}\left(\frac{q_1\gamma_{q_1}S^{\frac{N}{2}}}{2-q_1\gamma_{q_1}}\right)^{\frac{2-q_1\gamma_{q_1}}{2}},
\end{eqnarray*}
with $C_{N,q_1}:=C_{N,2,q_1}$ and $S$ the optimal constants of the Gagliardo-Nirenberg and Sobolev inequalities, respectively, such that if $0<\mu a^{\frac{q_1-q_1\gamma_{q_1}}{2}}<\alpha_{N,q_1}$ then the equation~\eqref{e1.1} in the semilinear case $p=2$ has a ground-state solution $\hat{u}_{a,\mu}$.
\end{theorem}

\vskip0.12in

Theorems~\ref{thmSoave20JDE} and \ref{thmSoave20JFA} have been generalized to many other models as well as improved by either developing the minimization method on the Pohozaev manifold or introducing other methods in many recent papers, see, for example, \cite{Jeanjean-Jendrej-Le-Visciglia, Jeanjean-Le2021,  Li2021, Liu-Wu25, Wei-Wu2022} and the references therein for the improvements, \cite{Borthwick-Chang-Jeanjean-Soave2023, Borthwick-Chang-Jeanjean-Soave2024, Chang-Jeanjean-Soave2024, Pierotti-Soave2022} and the references therein for the generalizations and improvements to problems on metric graphs, \cite{Alves-Ji-Miyagaki2021, Bieganowski-Mederski2021, Jeanjean-Lu2022-1, Jeanjean-Lu2022-2, Jeanjean-Zhang-Zhong2024, Mederski-Schino2021} and the references therein for the generalizations and improvements to problems with general nonlinearities, \cite{Bartsch-Molle-Rizzi-Verzini2021, Bartsch-Qi-Zou2024, Pierotti-Verzini-Yu2025} and the references therein for the generalizations and improvements to problems on bounded domains or $V(x)\not=0$, and so on.  In particular, the $p$-Laplacian equation~\eqref{e1.1} are also been considered in these generalizations.
To our best knowledge, \cite{Zhang-Zhang} is the first paper to study the $p$-Laplacian equation~\eqref{e1.1} with the $L^p$-constraint, where the authors proved the following theorem.
\begin{theorem}\label{thmZhang-Zhang}
(\cite[(b) of Theorem~1.1]{Zhang-Zhang})\quad Let $N\ge 2$, $1<p<N$, $p<q_1<p+\frac{p^2}{N}<q_2<p^*$ and $a,\mu>0$.  If the parameters $a$ and $\mu$ satisfy
$\max_{t>0}h_1(t)\geq 0$
and
\begin{equation*}
\mu a^{q_1(1-\gamma_{q_1})}<(t_{a,q_1,q_2})^{p-q_1\gamma_{q_1}}\frac{q_1(q_2\gamma_{q_2}-p)}{C_{N,p,q_1}^{q_1}p(q_2\gamma_{q_2}-q_1\gamma_{q_1})},
\end{equation*}
where
\begin{equation*}
h_1(t)=\frac{1}{p}t^p-\frac{\mu}{q_1}C_{N,p,q_1}^{q_1}a^{q_1(1-\gamma_{q_1})}t^{q_1\gamma_{q_1}}-\frac{1}{q_2}C_{N,p,q_2}^{q_2}a^{q_2(1-\gamma_{q_2})}t^{q_2\gamma_{q_2}}
\end{equation*}
and $t_{a,q_1,q_2}$ is the unique maximal point of $h_1(t)$,
then \eqref{e1.1} has a positive ground-state solution which minimizes $\Psi_{\mu}(u)$ on $\mathcal{P}_{a,\mu}$.
\end{theorem}
The Sobolev critical case $q_2=p^*$ of the $p$-Laplacian equation~\eqref{e1.1} was first considered in \cite{Lou-Zhang-Zhang24}, where the authors proved the following theorem.
\begin{theorem}\label{thmZhang-Zhang24}
(\cite[Theorem~1.1]{Lou-Zhang-Zhang24})\quad Let $N\ge 2$, $1<p<N$, $p<q_1<p+\frac{p^2}{N}$, $q_2=p^*$  and $a,\mu>0$.  If the parameters $a$ and $\mu$ satisfy
$\max_{t>0}h_2(t)\geq 0$
and
\begin{equation*}
\mu a^{q_1(1-\gamma_{q_1})}<\left(\frac{S^{\frac{p^*}{p}}p^*(p-q_1\gamma_{q_1})}{p(p^*-q_1\gamma_{q_1})}\right)^{\frac{p-q_1\gamma_{q_1}}{p^*-p}}\frac{q_1(p^*-p)}
{C_{N,p,q_1}^{q_1}p(p^*-q_1\gamma_{q_1})},
\end{equation*}
where
\begin{equation*}
h_2(t)=\frac{1}{p}t^p-\frac{\mu}{q_1}C_{N,p,q_1}^{q_1}a^{q_1(1-\gamma_{q_1})}t^{q_1\gamma_{q_1}}-\frac{1}{p^*}S^{-\frac{p^*}{p}}t^{p^*},
\end{equation*}
then \eqref{e1.1} has a positive ground-state solution which minimizes $\Psi_{\mu}(u)$ on $\mathcal{P}_{a,\mu}$.
\end{theorem}

\vskip0.12in

Theorem~\ref{thmZhang-Zhang24} is improved in \cite{Deng-Wu,Feng-Li} in the sense that under the same conditions on $\mu$ and $a$,  Feng and Li in \cite{Feng-Li} and Deng and Wu in \cite{Deng-Wu} constructed a positive mountain-pass solution of \eqref{e1.1} by further assuming $3^{2/3}\leq N^{2/3}<p<3$ in \cite{Feng-Li} and $p\leq N^{1/2}$ or $N^{1/2}<p<3$ in \cite{Deng-Wu}, respectively.  We remark that the additional assumptions
on $N$ and $p$ in \cite{Deng-Wu,Feng-Li} look like technical conditions.  Thus, the first purpose of this paper is to remove these conditions in constructing a positive mountain-pass solution of \eqref{e1.1}.

\vskip0.12in

It is also worth pointing out that an interesting question in the existence theory of \eqref{e1.1} is the optimal value of the total mass $a$.  We remark that the same question is pointed out in \cite[Remark~1.3]{Jeanjean-Le2021} for the $2d$ and $3d$ Schrodinger-Poisson-Slater equation.  Going back to \eqref{e1.1}, starting from
\cite{Soave20JDE, Soave20JFA} (see also \cite{Deng-Wu,Feng-Li, Lou-Zhang-Zhang24, Zhang-Zhang} for the $p$-Laplacian case), the conditions imposed on the mass $a$ in the existence theory of \eqref{e1.1} always ensure that the degenerate submanifold $\mathcal{P}_{a,\mu}^0=\emptyset$ and the fibering maps $\Phi_{\mu,u}(s)$ have a global maximal point for every $u\in\mathcal{S}_a$.  These two properties play crucial roles in applying the minimization method on the Pohozaev manifold in these references.  In \cite{Cingolani-Jeanjean2019, Jeanjean-Le2021}, a first effect is made for the optimal value of the mass $a$ in the existence theory of solutions of the $2d$ and $3d$ Schr\"{o}dinger-Poisson-Slater equation (see \cite[Remark~1.3]{Jeanjean-Le2021} again), where
Cingolani and Jeanjean in \cite{Cingolani-Jeanjean2019} and Jeanjean and Le in \cite{Jeanjean-Le2021} do not impose the condition on the mass $a$ such that the fibering maps $\Phi_{\mu,u}(s)$ have a global maximal point for every $u\in\mathcal{S}_a$.  However, they still need the condition on the mass $a$ such that the degenerate submanifold $\mathcal{P}_{a,\mu}^0=\emptyset$.  In \cite{Liu-Wu25}, by introducing the extremal value of the Pohozaev manifold (we call it the first extremal value of the Pohozaev manifold below) which is inspired by \cite{Albuquerque-Silva2020}, the conditions imposed on the mass $a$ in the existence theory of solutions to \eqref{e1.1} in the semilinear case $p=2$ could allow that the degenerate submanifold $\mathcal{P}_{a,\mu}^0\not=\emptyset$ and the fibering maps $\Phi_{\mu,u}(s)$ interact the Pohozaev manifold $\mathcal{P}_{a,\mu}$ for every $u\in\mathcal{S}_a$ (not necessary to be local minimal points or local maximal points) at the extremal value.  In this paper, we shall further develop the idea and technique in \cite{Liu-Wu25} to make some effects on the optimal value of the mass $a$ in the existence theory of solutions to \eqref{e1.1} for the quasilinear case $1<p<N$.

\vskip0.12in

\subsection{Main results.}\quad
Let
\begin{equation}\label{e1.12}
\mu_{a}^*:=\frac{(q_2\gamma_{q_2}-p)(p-q_1\gamma_{q_1})^{\frac{p-q_1\gamma_{q_1}}{q_2\gamma_{q_2}-p}}}
{\gamma_{q_1}\gamma_{q_2}^{\frac{p-q_1\gamma_{q_1}}{q_2\gamma_{q_2}-p}}(q_2\gamma_{q_2}-q_1\gamma_{q_1})^{\frac{q_2\gamma_{q_2}-q_1\gamma_{q_1}}{q_2\gamma_{q_2}-p}}}
\inf_{u\in \mathcal{S}_a}\frac{(\|\nabla u\|_p^p)^{\frac{q_2\gamma_{q_2}-q_1\gamma_{q_1}}{q_2\gamma_{q_2}-p}}}{\|u\|_{q_1}^{q_1}(\|u\|_{q_2}^{q_2})^{\frac{p-q_1\gamma_{q_1}}{q_2\gamma_{q_2}-p}}}
\end{equation}
be the first extremal value of the $L^p$-Pohozaev manifold $\mathcal{P}_{a,\mu}$.  Then
\begin{eqnarray*}
\mu_{a}^*=\sup\left\{\mu>0\mid \mathcal{P}_{a,\tau}^0=\emptyset\quad\text{for all }0<\tau<\mu\right\}.
\end{eqnarray*}
For $\mu>\mu_a^*$, we know that $\mathcal{P}_{a,\mu}^0\not=\emptyset$.  Thus, we also define
\begin{eqnarray}\label{m0}
m_{rad}^{0}(b,\mu):=\left\{\aligned
&\inf_{u\in\mathcal{P}_{b,\mu,rad}^0}\Psi_\mu(u),\quad &\mathcal{P}_{b,\mu,rad}^0\not=\emptyset,\\
&+\infty, \quad &\mathcal{P}_{b,\mu,rad}^0=\emptyset
\endaligned\right.
\end{eqnarray}
for all $0<b\leq a$, where
\begin{equation*}
\mathcal{P}_{b,\mu,rad}^0:=\mathcal{P}_{b,\mu}^0\cap W^{1,p}_{rad}\left(\mathbb{R}^N\right)
\end{equation*}
and
\begin{equation*}
W^{1,p}_{rad}\left(\mathbb{R}^N\right)=\left\{u\in W^{1,p}\left(\mathbb{R}^N\right)\mid \text{$u$ is radially symmetric}\right\}.
\end{equation*}
Let
\begin{eqnarray*}
\mu_{a,+}^{**}:=\min\left\{\mu_a^{**}, \hat{\mu}_{a,+}^{**}, \tilde{\mu}_{a,+}^{**}\right\}
\end{eqnarray*}
be the second extremal value of $\mathcal{P}_{a,\mu}^{+}$ and
\begin{eqnarray*}
\mu_{a,-}^{**}:=\min\left\{\hat{\mu}_{a,-}^{**}, \tilde{\mu}_{a,-}^{**}, \overline{\mu}_{a,-}^{**}\right\}
\end{eqnarray*}
be the second extremal value of $\mathcal{P}_{a,\mu}^{-}$ for $q_2<p^*$,
where
\begin{eqnarray*}
\left\{\aligned
&\hat{\mu}_{a,\pm}^{**}:=\sup\left\{\mu\geq\mu_a^*\mid m_{rad}^{\pm}(b,\mu)<m_{rad}^{0}(b,\mu)\text{ for all }0<b\leq a\right\},\\
&\tilde{\mu}_{a,\pm}^{**}:=\sup\left\{\mu\geq\mu_a^*\mid m^{\pm}_{rad}(c,\mu)\geq m^{\pm}_{rad}(a,\mu)\text{ for all }0<c<a\right\},\\
&\mu_a^{**}:=\sup\left\{\mu\geq\mu_a^*\mid m^+_{rad}(a,\mu)<m^-_{rad}(a,\mu)\right\},\\
&\overline{\mu}_{a,-}^{**}:=\sup\left\{\mu\geq\mu_a^*\mid \mu\leq\mu(u) \  \text{for\ all}\  u\in\tilde{\Phi}(a,\mu)\right\}
\endaligned\right.
\end{eqnarray*}
with
\begin{eqnarray*}
\tilde{\Phi}(a,\mu):=\left\{u\in W^{1,p}_{rad}\left(\mathbb{R}^N\right)\backslash\{0\}\mid
\begin{array}{l}
\text{There\ exists\ a\ minimizing\ sequence}\ \{u_n\}\subset \mathcal{P}_{a,\mu,rad}^{-}\ \\
 \text{of}\  m_{rad}^-\left(a,\mu\right)\ \text{such\ that}\ u_n\rightharpoonup u\ \text{weakly\ in\ }  W^{1,p}_{rad}\left(\mathbb{R}^N\right)
\end{array}
\right\}
\end{eqnarray*}
and
\begin{equation}\label{radial}
m^{\pm}_{rad}(a,\mu):=\inf_{u\in \mathcal{P}_{a,\mu,rad}^{\pm}}\Psi_{\mu}(u)\quad\text{and}\quad \mathcal{P}_{a,\mu,rad}^\pm:=\mathcal{P}_{a,\mu}^\pm\cap W^{1,p}_{rad}\left(\mathbb{R}^N\right),
\end{equation}
then our main result for the Sobolev subcritical case $q_2<p^*$ can be stated as follows.
\begin{theorem}\label{thm1.1}
(The Sobolev subcritical case)\quad Let $1<p<N$ with $N\geq2$, $a>0$, $p<q_1<p+\frac{p^2}{N}$ and $p+\frac{p^2}{N}<q_2<p^*$.
\begin{enumerate}
\item[$(1)$]\quad If $0<\mu\leq\mu_{a}^{*}$, then the stationary version of MFG models~\eqref{sMFG} with the Hamiltonians~\eqref{Hamiltonian} has two radial solutions $u_{a,\mu,\pm}$ for suitable Lagrange multipliers $\lambda_{a,\mu,\pm}<0$.  Moreover,
$u_{a,\mu,+}$ is a ground-state solution minimizing $\Psi_{\mu}(u)$ on $\mathcal{P}_{a,\mu}^+$ and $u_{a,\mu,-}$ is a mountain-pass solution minimizing $\Psi_{\mu}(u)$ on $\mathcal{P}_{a,\mu}^-$.
\item[$(2)$]\quad If $\mu_{a}^{*}<\mu<\mu_{a,+}^{**}$, then the stationary version of MFG models~\eqref{sMFG} with the Hamiltonians~\eqref{Hamiltonian} has a radial solution $u_{a,\mu,+}$ for a suitable Lagrange multiplier $\lambda_{a,\mu,+}<0$.  Moreover,
$u_{a,\mu,+}$ is a ground-state solution in the radial class minimizing $\Psi_{\mu}(u)$ on $\mathcal{P}_{a,\mu,rad}^+$.
\item[$(3)$]\quad If $\mu_{a}^{*}<\mu<\mu_{a,-}^{**}$, then the stationary version of MFG models~\eqref{sMFG} with the Hamiltonians~\eqref{Hamiltonian} has a radial solution $u_{a,\mu,-}$ for a suitable Lagrange multiplier $\lambda_{a,\mu,-}<0$.  Moreover,
$u_{a,\mu,-}$ is a mountain-pass  solution in the radial class minimizing $\Psi_{\mu}(u)$ on $\mathcal{P}_{a,\mu,rad}^-$.
\end{enumerate}
\end{theorem}

\begin{remark}
Clearly, by Theorem~\ref{thm1.1}, the stationary version of MFG models~\eqref{sMFG} with the Hamiltonians~\eqref{Hamiltonian} has two radial solutions for $0<\mu<\min\{\mu_{a,+}^{**}, \mu_{a,-}^{**}\}$ and one radial solution for $\min\{\mu_{a,+}^{**}, \mu_{a,-}^{**}\}\leq\mu<\max\{\mu_{a,+}^{**}, \mu_{a,-}^{**}\}$ in the Sobolev subcritical case, provided $\min\{\mu_{a,+}^{**}, \mu_{a,-}^{**}\}<\max\{\mu_{a,+}^{**}, \mu_{a,-}^{**}\}$.
\end{remark}

To state our result for $q_2=p^*$, we re-denote the second extremal value of $\mathcal{P}_{a,\mu}^-$ for $q_2<p^*$ by $\mu_{a,-,q_2}^{**}$.  Then we define
\begin{eqnarray*}
\mu_{a,-}^{**}:=\min\left\{\mu_{a,-}^{***}, \mu_{a,-,0}^{**}\right\}
\end{eqnarray*}
be the second extremal value of $\mathcal{P}_{a,\mu}^-$ for $q_2=p^*$, where $\mu_{a,-,0}^{**}:=\limsup_{q_2\to p^*}\mu_{a,-,q_2}^{**}$ and
\begin{eqnarray*}
\mu_{a,-}^{***}:=\sup\left\{\mu\geq\mu_a^*\mid m_{rad}^{-}(a,\mu)<m_{rad}^{+}(a,\mu)+\frac{1}{N}S^{\frac{N}{p}}\right\}.
\end{eqnarray*}
\begin{theorem}\label{thm1.2}
(The Sobolev critical case)\quad Let $1<p<N$ with $N\geq2$, $a>0$, $p<q_1<p+\frac{p^2}{N}$ and $q_2=p^*$.
\begin{enumerate}
\item[$(1)$]\quad If $0<\mu\leq\mu_{a}^{*}$, then the stationary version of MFG models~\eqref{sMFG} with the Hamiltonians~\eqref{Hamiltonian} has two radial solutions $u_{a,\mu,\pm}$ for suitable Lagrange multipliers $\lambda_{a,\mu,\pm}<0$.  Moreover,
$u_{a,\mu,+}$ is a ground-state solution minimizing $\Psi_{\mu}(u)$ on $\mathcal{P}_{a,\mu}^+$ and $u_{a,\mu,-}$ is a mountain-pass solution minimizing $\Psi_{\mu}(u)$ on $\mathcal{P}_{a,\mu}^-$.
\item[$(2)$]\quad If $\mu_{a}^{*}<\mu<\mu_{a,+}^{**}$, then the stationary version of MFG models~\eqref{sMFG} with the Hamiltonians~\eqref{Hamiltonian} has a radial solution $u_{a,\mu,+}$ for a suitable Lagrange multiplier $\lambda_{a,\mu,+}<0$.  Moreover,
$u_{a,\mu,+}$ is a ground-state solution in the radial class minimizing $\Psi_{\mu}(u)$ on $\mathcal{P}_{a,\mu,rad}^+$.
\item[$(3)$]\quad There exists $q_{1}^*\in\left(p, p+\frac{p^2}{N}\right)$ independent of $a$ such that
if $\mu_{a}^{*}<\mu<\mu_{a,-}^{**}$  and $q_1\in\left(q_{1}^*, p+\frac{p^2}{N}\right)$, then the stationary version of MFG models~\eqref{sMFG} with the Hamiltonians~\eqref{Hamiltonian} has a radial solution $u_{a,\mu,-}$ for a suitable Lagrange multiplier $\lambda_{a,\mu,-}<0$.  Moreover,
$u_{a,\mu,-}$ is a mountain-pass  solution in the radial class minimizing $\Psi_{\mu}(u)$ on $\mathcal{P}_{a,\mu,rad}^-$.
\end{enumerate}
\end{theorem}

\begin{remark}
Again, by Theorem~\ref{thm1.2}, the stationary version of MFG models~\eqref{sMFG} with the Hamiltonians~\eqref{Hamiltonian} has two radial solutions for $0<\mu\leq\mu_{a}^{*}$ and one radial solution for $\mu_{a}^{*}<\mu<\mu_{a,+}^{**}$ in the Sobolev critical case for all $q_1\in\left(p, p+\frac{p^2}{N}\right)$.  Moreover, if $q_1\in\left(q_{1}^*, p+\frac{p^2}{N}\right)$ for a suitable $q_1^*>p$, then the stationary version of MFG models~\eqref{sMFG} with the Hamiltonians~\eqref{Hamiltonian} has two radial solutions for $0<\mu<\min\{\mu_{a,+}^{**}, \mu_{a,-}^{**}\}$ and one radial solution for $\min\{\mu_{a,+}^{**}, \mu_{a,-}^{**}\}\leq\mu<\max\{\mu_{a,+}^{**}, \mu_{a,-}^{**}\}$ in the Sobolev critical case, provided $\min\{\mu_{a,+}^{**}, \mu_{a,-}^{**}\}<\max\{\mu_{a,+}^{**}, \mu_{a,-}^{**}\}$.
\end{remark}

\vskip0.12in

As pointed out above, there are several methods to prove the existence theory of \eqref{e1.1} even with general nonlinearities.  However, the advantage of the minimization method on the Pohozaev manifold is that it is easy to compute the Morse index of the constructed solutions, which is useful in understanding the global dynamics of \eqref{MFG}.  To compute the Morse index of the constructed solutions in Theorems~\ref{thm1.1} and \ref{thm1.2} for general $1<p<N$, it is necessary that the functional $\Psi_\mu(u)$ and the norm $\|u\|_p^p$ are of class $C^2$, which requires $2\leq p<N$ as that in \cite{Aftalion-Pacella, Cingolan-Degiovanni-Sciunzi}.  For the sake of simplicity, we denote $u_{\pm}=u_{a,\mu,\pm}$ and $\lambda_{\pm}=\lambda_{a,\mu,\pm}$, where $u_{a,\mu,\pm}$ are the solutions of \eqref{e1.1} constructed in Theorems~\ref{thm1.1} and \ref{thm1.2} and $\lambda_{a,\mu,\pm}<0$ are their Lagrange multipliers.  Since $p\geq2$, $p<q_1<q_2$ and $u_{\pm}$ are radial, by \cite[Theorem~1.1]{Li-Zhao},
\begin{eqnarray}\label{decay}
\lim_{r\to+\infty}u_{\pm}r^{\frac{N-1}{p(p-1)}}e^{\left(\frac{\lambda_{\pm}}{p-1}\right)^{\frac{1}{p}}r}=c_{\pm}\quad\text{and}\quad\lim_{r\to+\infty}\frac{u'_{\pm}}{u_{\pm}}=-d_{\pm}
\end{eqnarray}
for positive constants $c_{\pm}$ and $d_{\pm}$.  Thus, we can follow the idea in \cite{Figalli-Neumayer} to compute the Morse index of $u_{\pm}$ for $2\leq p<N$ in weighted Sobolev spaces.  For every $v\in C_0^{\infty}\left(\mathbb{R}^N\right)$, we introduce the following norm and inner product
\begin{equation*}
\|v\|_{\mathbb{H}_{\pm}}:=\left(\int_{\mathbb{R}^N}\left|u'_{\pm}\right|^{p-2}|\nabla v|^2+(p-2)\left|u'_{\pm}\right|^{p-4}(\nabla u_{\pm}\nabla v)^2-\lambda_{\pm}(p-1)\left| u_{\pm}\right|^{p-2}| v|^2dx\right)^{\frac{1}{2}}
\end{equation*}
and
\begin{equation*}
\left\langle w, v\right\rangle_{\mathbb{H}_{\pm}}:=\int_{\mathbb{R}^N}\left|u'_{\pm}\right|^{p-2}\nabla w\nabla v+(p-2)\left|u'_{\pm}\right|^{p-4}(\nabla u_{\pm}\nabla v)(\nabla u_{\pm}\nabla w)-\lambda_{\pm}(p-1)\left| u_{\pm}\right|^{p-2}wvdx.
\end{equation*}
We also introduce the weighted Sobolev spaces
\begin{equation*}
\mathbb{H}_{\pm}:=\left\{v\in W^{1,p}\left(\mathbb{R}^N\right)\mid \|v\|_{\mathbb{H}_{\pm}}<+\infty\right\}\quad\text{and}\quad\mathbb{H}_{\pm,rad}:=\left\{v\in W_{rad}^{1,p}\left(\mathbb{R}^N\right)\mid \|v\|_{\mathbb{H}_{\pm}}<+\infty\right\}.
\end{equation*}
Clearly, $\mathbb{H}_{\pm}$ and $\mathbb{H}_{\pm, rad}$ are Hilbert spaces and since $p\geq2$ and $\lambda_{\pm}<0$, by the H\"older inequality, it is easy to see that $\mathbb{H}_{\pm}=W^{1,p}\left(\mathbb{R}^N\right)$ and $\mathbb{H}_{\pm, rad}=W_{rad}^{1,p}\left(\mathbb{R}^N\right)$ in the sense of sets.  Moreover, by \cite[Theorem~6.1]{Figalli-Neumayer} and \eqref{decay}, $\mathbb{H}_{\pm}$ compactly embeds into $L^2\left(\mathbb{R}^N; \pmb{w}_{\pm}\right)$, where $L^2\left(\mathbb{R}^N; \pmb{w}_{\pm}\right)$ is the weighted Lebesgue space with the weight $\pmb{w}_{\pm}=\mu(q_1-1)u_{\pm}^{q_1-2}+(q_2-1)u_{\pm}^{q_2-2}$ and the usual weighted norm
\begin{eqnarray*}
\|v\|_{L^2; \pmb{w}_{\pm}}:=\left(\int_{\mathbb{R}^N}\left(\mu(q_1-1)u_{\pm}^{q_1-2}+(q_2-1)u_{\pm}^{q_2-2}\right)v^2dx\right)^{\frac12}.
\end{eqnarray*}
It follows that the linear operator
\begin{eqnarray*}
\mathcal{L}_{\pm}(w)=-\text{div}\left(\left|u'_{\pm}\right|^{p-2}\nabla w+(p-2)|u'_{\pm}|^{p-4}(\nabla u_{\pm}\nabla w)\nabla u_{\pm}\right)-\lambda_{\pm}(p-1)\left| u_{\pm}\right|^{p-2}w
\end{eqnarray*}
has a discrete spectrum $\{\sigma_m\}$ in the weighted Lebesgue space $L^2\left(\mathbb{R}^N; \pmb{w}_{\pm}\right)$.  We denote $\{\sigma_{m,rad}\}$ the discrete spectrum of $\mathcal{L}_{\pm}(w)$ in the radial setting.
We now introduce the following definition of Morse index of $u_{\pm}$.
\begin{definition}
The Morse index of $u_{\pm}$ in $\mathbb{H}_{\pm}$, denoted by $\pmb{M}(u_{\pm})$ is defined by the largest number of $m$ such that $\sigma_m<1$, that is
\begin{eqnarray*}
\pmb{M}(u_{\pm}):=\max\{m\in\mathbb{N}\mid \sigma_m<1\}.
\end{eqnarray*}
Correspondingly, the Morse index of $u_{\pm}$ in $\mathbb{H}_{\pm,rad}$, denoted by $\pmb{M}_{rad}(u_{\pm})$ is defined by the largest number of $m$ such that $\sigma_{m,rad}<1$, that is
\begin{eqnarray*}
\pmb{M}_{rad}(u_{\pm}):=\max\{m\in\mathbb{N}\mid \sigma_{m,rad}<1\}.
\end{eqnarray*}
\end{definition}
For the sake of simplicity, we also denote $\mathcal{P}^{\pm}=\mathcal{P}_{a,\mu}^{\pm}$ and $\mathcal{P}_{rad}^{\pm}=\mathcal{P}_{a,\mu,rad}^{\pm}$.  Now,
our results in computing Morse index of $u_{\pm}$ can be stated as follows.
\begin{theorem}\label{thm1.3}
Let $N\geq 3$ and $2\leq p<N$. Suppose that $u_{\pm}$ are the solutions constructed in Theorems~\ref{thm1.1} and \ref{thm1.2}, where $u_{+}$ is a ground-state solution minimizing $\Psi_{\mu}(u)$ on $\mathcal{P}^{+}$ for $0<\mu\leq\mu_a^*$ and on $\mathcal{P}_{rad}^{+}$ for $\mu_a^*<\mu<\mu_{a,+}^{**}$ while, $u_{-}$ is a mountain-pass solution minimizing $\Psi_{\mu}(u)$ on $\mathcal{P}^{-}$ for $0<\mu\leq\mu_a^*$ and on $\mathcal{P}_{rad}^{-}$ for $\mu_a^*<\mu<\mu_{a,-}^{**}$.  Then
\begin{enumerate}
\item[$(1)$]\quad $\pmb{M}(u_{+})=1$ and $\pmb{M}(u_{-})=2$ for $0<\mu\leq\mu_a^*$.
\item[$(2)$]\quad $\pmb{M}_{rad}(u_{+})=1$ for $\mu_a^*<\mu<\mu_{a,+}^{**}$ and $\pmb{M}_{rad}(u_{-})=2$ for $\mu_a^*<\mu<\mu_{a,-}^{**}$.
\end{enumerate}
\end{theorem}

\subsection{Further remarks.}\quad
Our assumptions on the parameters $a$ and $\mu$ in Theorems~\ref{thm1.1} and \ref{thm1.2} are more general than that used for Theorems~\ref{thmSoave20JDE}, \ref{thmSoave20JFA}, \ref{thmZhang-Zhang} and \ref{thmZhang-Zhang24} in \cite{Deng-Wu, Feng-Li, Jeanjean-Jendrej-Le-Visciglia, Jeanjean-Le2021,  Jeanjean-Le, Liu-Wu25, Lou-Zhang-Zhang24, Wei-Wu2022, Zhang-Zhang} and the references therein to study \eqref{e1.1}, even in the semilinear case $p=2$.  Thus,
to construct positive solutions of \eqref{e1.1} under the general conditions in Theorems~\ref{thm1.1} and \ref{thm1.2}, we need to deal with three additional difficulties in applying the minimization method on the Pohozaev manifold.

\vskip0.12in

The first difficulty lies in proving the compactness of minimizing sequences on $\mathcal{P}_{a,\mu}^{-}$ for $q_2=p^*$.  As observed in \cite{Jeanjean-Le, Liu-Wu25, Wei-Wu2022} for the semilinear case $p=2$, respectively, it is necessary to establish the following energy estimate
\begin{equation*}
m^{-}(a,\mu)<m^{+}(a,\mu)+\frac{1}{N}S^{\frac{N}{p}}
\end{equation*}
in overcoming this compactness issue.  The basic idea in the above estimate, which is originally introduced by Brezis and Nirenberg in \cite{Brezis-Nirenberg83} and Tarantello in \cite{T1992}, is to glue the ground-state solution $u_{a,\mu,+}$ and a modification of the Talanti bubble to form a path such that this path interacts $\mathcal{P}_{a,\mu}^{-}$, and the maximal  energy of this path could be well controlled.  Two different techniques are introduced in \cite{Jeanjean-Le, Liu-Wu25, Wei-Wu2022} for the semilinear case $p=2$ of \eqref{e1.1} to achieve this goal, where the nonradial construction in \cite{Jeanjean-Le} only works for $N\geq4$ while the radial-superposition construction in \cite{Liu-Wu25, Wei-Wu2022} works for all $N\geq3$.  In these two constructions, the good terms $\int_{\mathbb{R}^N}u_{a,\mu}^+u_{\epsilon}^{2^*-1}dx$ and $\|u_\epsilon\|_{q_1}^{q_1}$ can control the other infinitesimals for $N=3$ and $N\geq4$, respectively, where $u_\epsilon$ is a modification of the Talanti bubble by a cut-off function.  However, unlike the semilinear case $p=2$,
\begin{eqnarray*}
0\not=\|\nabla u_{a,\mu}^++\tau \nabla u_\epsilon\|_p^p-\|\nabla u_{a,\mu}^+\|_p^p-\|\tau \nabla u_\epsilon\|_p^p-p\int_{\mathbb{R}^N}|\nabla u_{a,\mu}^+|^{p-2}\nabla u_{a,\mu}^+\tau\nabla u_{\epsilon}dx
\end{eqnarray*}
in the quasilinear case for $1<p\not=2<N$, which requires some additional assumptions on $N$ and $p$ (cf. $3^{2/3}\leq N^{2/3}<p<3$ in \cite{Feng-Li} and $p\leq N^{1/2}$ or $N^{1/2}<p<3$ in \cite{Deng-Wu}) to control the energy of the path by using the standard cut-off function $\varphi=1$ in $B_1(0)$ and $\varphi=0$ in $\mathbb{R}^{N}\setminus B_2(0)$.  To remove these technical conditions, we introduce {\bf a new parameter} $\alpha$ in the cut-off function $\varphi$ such that $\varphi$ depends on $\epsilon$ with $\varphi_{\epsilon}=1$ in $B_{\epsilon^\alpha}(0)$ and $\varphi_{\epsilon}=0$ in $\mathbb{R}^N\setminus B_{2\epsilon^\alpha}(0)$.  We then well control the additional errors in $\|\nabla u_{a,\mu}^++\tau \nabla u_\epsilon\|_p^p$ for the quasilinear case for $1<p\not=2<N$ by a careful choice of $\alpha$, which depends on $p$, and kill them by the good terms $\int_{\mathbb{R}^N}u_{a,\mu}^+u_{\epsilon}^{p^*-1}dx$  in the radial construction of pathes in the superposition as that in \cite{ Wei-Wu2022} for all $0<\mu\leq\mu_a^*$ and $N^{1/2}<p<N$.

\vskip0.12in

The second difficulty lies in applying the sub-additive inequality of the energy in proving the compactness of a minimizing sequence $\{u_n\}\subset \mathcal{P}_{a,\mu}^{\pm}$ for $m^\pm(a,\mu)$ in the case of $q_2=p^*$.  Indeed, since the ground-state energy $m^{+}(a,\mu)<0$ by the $L^p$-subcritical perturbation for $p<q_1<p+\frac{p^2}{N}$, we need to remove it in the limit equation to ensure the least energy of nonzero solutions of the limit equation is positive, so that the sub-additive inequality works in killing the possibility of noncompactness of the minimizing sequence on $\mathcal{P}_{a,\mu}^{\pm}$ at the levels $m^\pm(a,\mu)$, which is usually achieved by adapting the symmetric rearrangement to the minimizing sequence on $\mathcal{P}_{a,\mu}^{\pm}$ and applying the compactness of the embedding from $W^{1,p}_{rad}\left(\mathbb{R}^N\right)$ into $L^q\left(\mathbb{R}^N\right)$ for $p<q<p^*$.
To achieve this goal, it is necessary to ensure that the symmetric rearrangement to the minimizing sequence on $\mathcal{P}_{a,\mu}^{\pm}$ is still a minimizing sequence, which is guaranteed by adding some conditions on the parameters $a$ and $\mu$ in the literatures such that either the fibering maps $\Phi_{\mu,u}(s)$ have a global maximal point for every $u\in\mathcal{S}_a$ or $\Psi_{\mu}(u)$ has a local minimizer in the interior of the submanifold $\mathcal{D}:=\{u\in \mathcal{S}_{a}\mid\|\nabla u\|_p\leq k\}$ for a suitable choice of $k$.  In \cite{Cingolani-Jeanjean2019, Jeanjean-Le2021}, by noting that the convolution term is of order $1$ in the fibering maps for the $2d$ and $3d$ Schr\"odinger-Poisson-Slater equation, which implies that the second derivative of the fibering maps changes sign at most once, Cingolani and Jeanjean in \cite{Cingolani-Jeanjean2019} and Jeanjean and Le in \cite{Jeanjean-Le2021} showed that the symmetric rearrangement to some special local minimizing sequences on the Pohozaev manifold of the $2d$ and $3d$ Schr\"odinger-Poisson-Slater equation are still local minimizing sequences by adding some conditions on the parameters $a$ and $\mu$ such that the Pohozaev manifold is nondegenerate.  However, the nonlinearity of the subcritical perturbation of \eqref{e1.1} is stronger than that of the $2d$ and $3d$ Schr\"odinger-Poisson-Slater equations and the fibering maps $\Phi_{\mu,u}(s)$ only have a local maximal point for every $u\in\mathcal{S}_a$ if $\mu<\mu_a^*$ and sufficiently close to $\mu_a^*$ (see below for details), and even has no local maximal and minimal points for some $u\in\mathcal{S}_a$
if $\mu\geq\mu_a^*$.  To overcome this difficulty, our idea is to introduce {\bf a new parameter} $\alpha\in(0, 1)$ in front of $\|\nabla u\|_p^p$ in $\Psi_{\mu}(u)$ to simulate the symmetric rearrangement of $u$.  By using the implicit function theorem as that in \cite[Lemma~3.2]{Wei-Wu2022}, we proved a uniform estimate for all $\alpha\in(0, 1]$, which ensures that the symmetric rearrangement to the minimizing sequence on $\mathcal{P}_{a,\mu}^{\pm}$ is still a minimizing sequence if $0<\mu<\mu_a^*$.  We believe that this idea will be potentially helpful in other studies.

\vskip0.12in

The third difficulty, which is the main difficulty in proving Theorems~\ref{thm1.1} and \ref{thm1.2}, lies in proving the existence theory of \eqref{e1.1} for $\mu\geq\mu_a^*$ and these results are even new for the semilinear case $p=2$.  In this region of $\mu$, $\mathcal{P}_{a,\mu}$ is degenerate since $\mathcal{P}_{a,\mu}^0\not=\emptyset$ in many cases.  What is worse is that we also loss the one-to-one correspondance between $\mathcal{S}_a$ and $\mathcal{P}_{a,\mu}$ for $\mu>\mu_a^*$.  Thus, $\mathcal{P}_{a,\mu}$ is even no longer homomorphous to $\mathcal{S}_a$ any more for $\mu>\mu_a^*$.  These facts mean that the nondegenerate submanifolds $\mathcal{P}_{a,\mu}^{\pm}$, which are used to construct solutions of \eqref{e1.1} by considering the minimizers on them, have nonempty boundaries for $\mu\geq\mu_a^*$.  Thus, we need to carefully control the minimizing sequences on $\mathcal{P}_{a,\mu}^{\pm}$ to make sure that they are away from their boundaries, which is crucial in applying the method of the Lagrange multiplier to construct solutions of \eqref{e1.1} from local minimizers of $\Psi_{\mu}(u)$ on $\mathcal{P}_{a,\mu}$.  To achieve this goal, we introduce the following variational problems
\begin{equation}\label{newvp}
\hat{\Psi}_{\mu,rad}^{\pm}:=\inf\left\{\Psi_{\mu}(u) \mid u\in \mathcal{P}_{a,\mu,rad}^{\pm}\cup\partial\mathcal{P}_{a,\mu,rad}^{\pm}\right\}
\end{equation}
for $\mu\geq\mu_a^*$, where $\partial\mathcal{P}_{a,\mu,rad}^{\pm}=\mathcal{P}_{a,\mu}^{0}$ is the boundary of $\mathcal{P}_{a,\mu,rad}^{\pm}$ in the $W_{rad}^{1,p}\left(\mathbb{R}^N\right)$ topology.  We remark that in some other weaker topology, it is not necessary that $\partial\mathcal{P}_{a,\mu,rad}^{\pm}=\mathcal{P}_{a,\mu}^{0}$.
We then introduce the values $\hat{\mu}_{a,\pm}^{**}$, $\mu_a^{**}$ and $\overline{\mu}_{a,-}^{**}$, which serve as the separation conditions of
$\mathcal{P}_{a,\mu}^{+}$ for $\mu_a^*<\mu<\min\left\{\hat{\mu}_{a,+}^{**}, \mu_a^{**}\right\}$ and of $\mathcal{P}_{a,\mu}^{-}$ for $\mu_a^*<\mu<\min\left\{\hat{\mu}_{a,-}^{**}, \overline{\mu}_{a,-}^{**}\right\}$, as the choice of $k$ to the submanifold $\mathcal{D}:=\{u\in \mathcal{S}_{a}\mid\|\nabla u\|_p\leq k\}$ in many papers (cf. \cite{Gou-Zhang21, Jeanjean-Jendrej-Le-Visciglia, Jeanjean-Le2021,  Jeanjean-Le, Lou-Zhang-Zhang24, Soave20JDE, Soave20JFA,  Zhang-Zhang}), such that
\begin{eqnarray}\label{eqn1010}
\inf_{\mathcal{P}_{a,\mu,rad}^{\pm}}\Psi_{\mu}(u)<\inf_{\partial\mathcal{P}_{a,\mu,rad}^{\pm}}\Psi_{\mu}(u).
\end{eqnarray}
It is worth pointing out that to our best knowledge, this idea first appears in \cite{Lin-Wei05} in the usage of the Nehari manifold to construct nontrivial solutions of the coupled elliptic system.  Similar ideas are also applies in \cite{Albuquerque-Silva2020, Bieganowski-Mederski2021, Cingolani-Jeanjean2019, Jeanjean-Jendrej-Le-Visciglia,  Jeanjean-Le2021, Jeanjean-Le,  Mederski-Schino2021, Soave20JDE, Soave20JFA, Yao-Chen-Sun2023} for other studies, in particular, this idea was also applied and developed for the existence theory of solutions of \eqref{pSchrodinger} in the semilinear case $p=2$.
By \eqref{eqn1010}, we can well control the minimizing sequences on $\mathcal{P}_{a,\mu}^{\pm}$
to make sure that they are away from their boundaries after solving their compactness issues, for which we introduce the values $\tilde{\mu}_{a,\pm}^{**}$ for the compactness in $L^p\left(\mathbb{R}^N\right)$ and the values $\mu_{a,-}^{***}$ and $\mu_{a,-,0}^{**}$ for the compactness in $L^{p^*}\left(\mathbb{R}^N\right)$ of the mountain-pass solution in the Sobolev critical case, and proving
\begin{eqnarray*}
\min\left\{\hat{\mu}_{a,\pm}^{**}, \tilde{\mu}_{a,\pm}^{**}, \overline{\mu}_{a,-}^{**}, \mu_a^{**}, \mu_{a,-}^{***}, \mu_{a,-,0}^{**}\right\}>\mu_a^*.
\end{eqnarray*}
These two goals are all achieved by developing a good perturbation theory around the degenerate submanifold $\mathcal{P}_{a,\mu_a^*}^0$ at the first extremal value $\mu_a^*$, which is based on the fact that $\mathcal{P}_{a,\mu_a^*}^0$ corresponds to minimizers of the variational problem~\eqref{e1.12}.  The cornerstone of our perturbation theory is the achievement of the minimum on $\mathcal{P}_{a,\mu_a^*}^0$.  In the Sobolev subcritical case $q_2<p^*$, it is achieved by using the compactness of the embedding from $W^{1,p}_{rad}\left(\mathbb{R}^N\right)$ into $L^q\left(\mathbb{R}^N\right)$ for $p<q<p^*$ while in the Sobolev critical case $q_2=p^*$, we need to further deal with a compactness issue according to the noncompactness of the embedding from $W^{1,p}_{rad}\left(\mathbb{R}^N\right)$ into $L^{p^*}\left(\mathbb{R}^N\right)$.  To overcome this difficulty, our key observation is that the degenerate submanifold $\mathcal{P}_{a,\mu,rad}^0$ can be embedded into the nondegenerate submanifold $\mathcal{P}_{a,\mu(t),rad}^-$ with an additional restriction for a precise value $\mu(t)$ in the following way:
\begin{eqnarray*}
\mathcal{Q}_{a,\mu(t),rad}^-:=\mathcal{T}\left(\mathcal{P}_{a,\mu,rad}^0\right)=\left\{w_t\in \mathcal{P}_{a,\mu(t),rad}^-\mid D_{\mu(t)}(w_t)=0\right\},
\end{eqnarray*}
where $\mathcal{T}: u\to w_t:=(u)_{t}=t^{\frac{N}{p}}u(tx)$ with $t>1$,
\begin{eqnarray*}
0<\mu(t)=\left(\frac{p^*-q_1\gamma_{q_1}}{p^*-p}t^{p-q_1\gamma_{q_1}}-\frac{p-q_1\gamma_{q_1}}{p^*-p}t^{p^*-q_1\gamma_{q_1}}   \right)\mu<\mu
\end{eqnarray*}
and $D_{\mu(t)}(w_t)=\|w_t\|_{p^*}^{p^*}-\frac{\mu_a^*\gamma_{q_1}(p-q_1\gamma_{q_1})}{p^*-p}t^{p^*-q_1\gamma_{q_1}}\|w_t\|_{q_1}^{q_1}$.  In particular, $\Psi_{\overline{\mu}_a^*}(w)=0$ for all $w\in\mathcal{Q}_{a,\overline{\mu}_a^*,rad}^-$, where
\begin{eqnarray*}
\overline{\mu}_a^*=\mu_a^*\left(\left(\frac{p^*}{p}\right)^{\frac{1}{p^*-p}}\right)=\frac{q_1\gamma_{q_1}}{p}\left(\frac{p^*}{p}\right)^{\frac{p-q_1\gamma_{q_1}}{p^*-p}}\mu_a^*.
\end{eqnarray*}
This, together with the Euler-Lagrange equation satisfied by $u\in\mathcal{P}_{a,\mu_a^*,rad}^0$ and the spliting arguments, implies the following energy estimate
\begin{eqnarray*}
0\geq m^+_{rad}\left(a,\overline{\mu}_a^*\right)+\frac{1}{N}S^{\frac{N}{p}},
\end{eqnarray*}
which serves as {\bf the threshold of the compactness} of minimizing sequences on the degenerate submanifold $\mathcal{P}_{a,\mu_a^*,rad}^0$ in $L^{p^*}\left(\mathbb{R}^N\right)$.  By proving $m^+_{rad}\left(a,\overline{\mu}_a^*\right)\to0$ as $q_1\to p+\frac{p^2}{N}$, we recover the compactness of minimizing sequences on the degenerate submanifold $\mathcal{P}_{a,\mu_a^*,rad}^0$ in $L^{p^*}\left(\mathbb{R}^N\right)$ for $q_1$ sufficiently close to $p+\frac{p^2}{N}$.  To our best knowledge, this is {\bf a completely new finding}, which we believe will be very useful in many other studies.  Moreover, by $\Psi_{\overline{\mu}_a^*}(w)=0$ for all $w\in\mathcal{Q}_{a,\overline{\mu}_a^*,rad}^-\subset \mathcal{P}_{a,\overline{\mu}_a^*,rad}^-$, we also know that $m^-_{rad}\left(a,\overline{\mu}_a^*\right)\leq0$, which together with the monotonicity of $m^-_{rad}\left(a,\mu\right)$ as a function of $\mu>0$, implies that $m^-_{rad}\left(a,\mu\right)<0$ for all $\mu>\overline{\mu}_a^*$.

\subsection{Another background of (\ref{e1.1}) in the semilinear case $p=2$}
Since some of our results are even new for the semilinear case $p=2$, we close this section by introducing another physical background of \eqref{e1.1} in this case.  In the semilinear case $p=2$, the motivation of the studies on (\ref{e1.1}) also comes from finding solitary waves of the nonlinear Schr\"{o}dinger equation
\begin{equation}\label{e8.1}
i\partial_t\varphi+\Delta \varphi+\mu |\varphi|^{q_1-2}\varphi+|\varphi|^{q_2-2}\varphi=0,\ \text{in}\  (0,+\infty)\times\mathbb{R}^N,
\end{equation}
where $i$ is the imaginary unit and $\varphi: (0,+\infty)\times\mathbb{R}^N\to \mathbb{C}$ is the wave function.  It is known in \cite{Tao-Visan-Zhang07} that any solution $\varphi$ of
equation (\ref{e8.1}) with the Cauchy initial function $\varphi(0, x)$ preserves the $L^2$-mass and the Hamiltonian, that is,
\begin{equation*}
\int_{\mathbb{R}^N}\left|\varphi(t,x)\right|^2dx=\int_{\mathbb{R}^N}\left|\varphi(0,x)\right|^2dx\text{ and }\Psi_{\mu}\left(\varphi(t, x)\right)=\Psi_{\mu}\left(\varphi(0, x)\right),\ \forall t\in (0,+\infty).
\end{equation*}
where $\Psi_{\mu}(u)$ is the energy of the Schr\"{o}dinger flow~\eqref{e8.1} defined in (\ref{e1.2}) for $p=2$.  Recall that a solitary wave of (\ref{e8.1}) is a solution of the form $\varphi(x,t)=e^{-i\lambda t}u(x)$, where $\lambda\in \mathbb{R}$ and $u: \mathbb{R}^N\to \mathbb{C}$. Thus,  $\varphi(x,t)=e^{-i\lambda t}u(x)$ solves \eqref{e8.1} is equivalent to $(u,\lambda)$ solves (\ref{e1.1}) with $a=\left(\int_{\mathbb{R}^N}|\varphi(0,x)|^2dx\right)^{\frac12}$.  We point out that \eqref{e8.1} appears in the study of Bose-Einstein condensates (cf. \cite{Fibich15}), where the wave function $\varphi$ describes
the state of the condensate and the $L^2$-mass is the total number of atoms.  The well understood of solitary waves is helpful in
understanding the dynamics of Schr\"{o}dinger flow~\eqref{e8.1}, see, for example,
\cite{Akahori-Ibrahim-Kikuchi-Nawa2013, Akahori-Ibrahim-Kikuchi-Nawa2021, Akahori-Ibrahim-Kikuchi-Nawa-Wei2025, Bellazzini-Forcella-Georgiev2023, Luo2022} and the references therein.  In particular, in the very recent paper \cite{Akahori-Ibrahim-Kikuchi-Nawa-Wei2025}, the previous studies on \eqref{e1.1} for the Laplacian case $p=2$ in \cite{Wei-Wu2022, Wei-Wu2023} are also used to classify the positive solutions with finite energy.  Thus, we hope that our ideas and techniques introduced in this paper will also be useful in understanding the Schr\"{o}dinger flow~\eqref{e8.1}.

\subsection{Organization of the paper and Notations}\quad
In Section 2, we give some preliminary results used in this paper. Sections 3, 4 and 5 are devoted to the proof of Theorems~\ref{thm1.1} and \ref{thm1.2} for the cases $\mu<\mu_{a}^*$, $\mu=\mu_{a}^*$ and $\mu>\mu_a^*$, respectively.  The proof of Theorem~\ref{thm1.3} is contained in Section 6 while, in Sections 8 and 9, we list some useful results and provide the energy estimate, respectively.

\medskip

Throughout this paper,  $a\lesssim b$
means that $a\leq Cb$,  $a\gtrsim b$
means that $b\lesssim a$ and $a\thickapprox b$ means that $a\lesssim b$ and $a\gtrsim b$, where $C$ is a positive constant. $B_s(x)$ denotes the ball in $\mathbb{R}^N$ of center at $x$ and radius $s$. $\|\cdot\|_p$ denotes the usual norm in the Lebesgue space $L^p\left(\mathbb{R}^N\right)$. $W^{1,p}\left(\mathbb{R}^N\right):=\{u\in L^p\left(\mathbb{R}^N\right)\mid |\nabla u|\in L^p\left(\mathbb{R}^N\right)\}$ with the norm $\|u\|_{W^{1,p}}:=(\|u\|_p^p+\|\nabla u\|_p^p)^{1/p}$. $W_{rad}^{1,p}\left(\mathbb{R}^N\right)$ denotes the subspace of functions in $W^{1,p}\left(\mathbb{R}^N\right)$ which are radially symmetric with respect to the original point.

\section{Preliminaries}
\setcounter{section}{2} \setcounter{equation}{0}
In this section, we shall introduce some necessary preliminaries.  We begin with the result that any solution of (\ref{e1.1}), say $(u,\lambda)$,
belongs to $\mathcal{P}_{a,\mu}$, where $\mathcal{P}_{a,\mu}$ is the $L^p$-Pohozaev manifold given by \eqref{e1.3}.
\begin{lemma}\label{lem1.5}
Let $N\geq 2$, $1<p<N$, $p\leq q_1\leq q_2\leq p^*$, $\mu\in \mathbb{R}$ and $a>0$.  If $u\in \mathcal{S}_a$ is a solution of (\ref{e1.1}) then $u\in \mathcal{P}_{a,\mu}$.
\end{lemma}
\begin{proof}
Since $u\in \mathcal{S}_a$ is a solution of (\ref{e1.1}), similarly to \cite[Lemma~A.1]{Garcia-Peral94}, we can prove that $u\in L_{loc}^{\beta p^*}\left(\mathbb{R}^N\right)$ for any $\beta\in\left(1, \frac{p^*}{p}\right)$.  Thus, by \cite[Theorem~1]{Serrin 64}, we conclude that $u\in L_{loc}^{\infty}\left(\mathbb{R}^N\right)$.  On the other hand, it is also clear that $|\nabla u|\in L^p\left(\mathbb{R}^N\right)$ and
\begin{equation*}
F(u)=\frac{\lambda}{p}|u|^p+\frac{\mu}{q_1}|u|^{q_1}+\frac{1}{q_2}|u|^{q_2}\in L^1\left(\mathbb{R}^N\right).
\end{equation*}
Thus, by Lemma \ref{lem1.4}, we have
\begin{equation}\label{e7.1}
(N-p)\|\nabla u\|_p^p=\lambda Na^p+\frac{\mu N p}{q_1}\|u\|_{q_1}^{q_1}+\frac{N p}{q_2}\|u\|_{q_2}^{q_2}.
\end{equation}
By multiplying (\ref{e1.1}) with $u$ on both sides and integrating by parts, we also obtain that
\begin{equation*}
\|\nabla u\|_p^p=\lambda a^p+\mu\|u\|_{q_1}^{q_1}+\|u\|_{q_2}^{q_2},
\end{equation*}
which, together with (\ref{e7.1}), implies that $u\in \mathcal{P}_{a,\mu}$.
\end{proof}

We next introduce the fundamental property of the fibering maps $\Phi_{\mu,u}(s)$ given by \eqref{e1.9}. Let
\begin{equation}\label{muu}
\mu(u):=\frac{(q_2\gamma_{q_2}-p)(p-q_1\gamma_{q_1})^{\frac{p-q_1\gamma_{q_1}}{q_2\gamma_{q_2}-p}}}
{\gamma_{q_1}\gamma_{q_2}^{\frac{p-q_1\gamma_{q_1}}{q_2\gamma_{q_2}-p}}(q_2\gamma_{q_2}-q_1\gamma_{q_1})^{\frac{q_2\gamma_{q_2}-q_1\gamma_{q_1}}{q_2\gamma_{q_2}-p}}}
\frac{(\|\nabla u\|_p^p)^{\frac{q_2\gamma_{q_2}-q_1\gamma_{q_1}}{q_2\gamma_{q_2}-p}}}{\|u\|_{q_1}^{q_1}(\|u\|_{q_2}^{q_2})^{\frac{p-q_1\gamma_{q_1}}{q_2\gamma_{q_2}-p}}}.
\end{equation}

\begin{proposition}\label{pro2.1}
Let $N\geq 2$, $1<p<N$, $p<q_1<p+\frac{p^2}{N}<q_2\leq p^*$, $\mu>0$ and $a>0$.
Then for any $u\in \mathcal{S}_a$, $\Phi_{\mu,u}(s)$ is of class $C^{\infty}$ in $(0, +\infty)$. Moreover, the set of fibering maps $\{\Phi_{\mu,u}(s)\mid u\in \mathcal{S}_a\}$ can be classified as follows:
\begin{enumerate}
\item[$(a)$]\quad If $\mu\in \left(0, \mu(u)\right)$, then $\Phi_{\mu,u}(s)$ only has two critical points $t^\pm_{\mu}(u)$ satisfying
\begin{eqnarray*}
0<t^+_{\mu}(u)<s(u)<t^-_{\mu}(u)<+\infty,
\end{eqnarray*}
where $t^+_{\mu}(u)$ is the local minimal point of $\Phi_{\mu,u}(s)$ in $\left(0,t^-_{\mu}(u)\right)$, $t^-_{\mu}(u)$ is the local maximal point of
$\Phi_{\mu,u}(s)$ in $\left(t^+_{\mu}(u),+\infty\right)$ and
\begin{equation*}
s(u):=\left(\frac{(p-q_1\gamma_{q_1})\|\nabla u\|_p^p}{\gamma_{q_2}(q_2\gamma_{q_2}-q_1\gamma_{q_1})\|u\|_{q_2}^{q_2}}\right)^{\frac{1}{q_2\gamma_{q_2}-p}}.
\end{equation*}
Moreover, $(u)_{t^{\pm}_{\mu}(u)}\in \mathcal{P}_{a,\mu}^{\pm}$ and $\Phi_{\mu,u}(t^+_{\mu}(u))<\min\left\{0, \Phi_{\mu,u}(t^-_{\mu}(u))\right\}$.
\item[$(b)$]\quad If $\mu=\mu(u)$, then $\Phi_{\mu,u}(s)$ only has a degenerate critical point $t^0_{\mu}(u)=s(u)$.  Moreover, $(u)_{t^{0}_{\mu}(u)}\in \mathcal{P}_{a,\mu}^{0}$.
\item[$(c)$]\quad If $\mu>\mu(u)$, then $\Phi_{\mu,u}(s)$ has no critical points.
\end{enumerate}
\end{proposition}
\begin{proof}
Direct calculations give that
\begin{equation*}
\begin{split}
\Phi_{\mu,u}'(s)&=s^{p-1}\|\nabla u\|_p^p-\mu\gamma_{q_1}s^{q_1\gamma_{q_1}-1}\|u\|_{q_1}^{q_1}-\gamma_{q_2}s^{q_2\gamma_{q_2}-1}\|u\|_{q_2}^{q_2}\\
&=s^{q_1\gamma_{q_1}-1}\left(h_{\mu,u}(s)-\mu\gamma_{q_1}\|u\|_{q_1}^{q_1}\right),
\end{split}
\end{equation*}
where $h_{\mu,u}(s):=s^{p-q_1\gamma_{q_1}}\|\nabla u\|_p^p-\gamma_{q_2}s^{q_2\gamma_{q_2}-q_1\gamma_{q_1}}\|u\|_{q_2}^{q_2}$.  Since
\begin{equation*}
\begin{split}
\max_{s>0}h_{\mu,u}(s)&=h_{\mu,u}(s(u))\\
&=\left(\frac{(p-q_1\gamma_{q_1})\|\nabla u\|_p^p}{\gamma_{q_2}(q_2\gamma_{q_2}-q_1\gamma_{q_1})\|u\|_{q_2}^{q_2}}\right)^{\frac{p-q_1\gamma_{q_1}}{q_2\gamma_{q_2}-p}}\frac{q_2\gamma_{q_2}-p}{q_2\gamma_{q_2}-q_1\gamma_{q_1}}\|\nabla u\|_p^p,
\end{split}
\end{equation*}
the equation $\Phi_{\mu,u}'(s)=0$
has solutions if and only if $\max_{s>0}h_{\mu,u}(s)\geq \mu\gamma_{q_1}\|u\|_{q_1}^{q_1}$, that is,
 $\mu\leq \mu(u)$, where $\mu(u)$ is given by \eqref{muu}.  Moreover, it has two solutions $t^\pm_{\mu}(u)$ for $\mu<\mu(u)$
satisfying  $0<t^+_{\mu}(u)<s(u)<t^-_{\mu}(u)<+\infty$ and has only one solution $t^0_{\mu}(u)=s(u)$ for
$\mu=\mu(u)$.
\end{proof}
We close this section by introducing the fundamental property of the first extremal value $\mu_{a}^*$ given by \eqref{e1.12}.  We recall that by \eqref{e1.12}, $\mu_{a}^*=\inf_{u\in \mathcal{S}_a}\mu(u)$, where  $\mu(u)$ is given by \eqref{muu}.

\begin{proposition}\label{pro2.2}
Let $N\geq 2$, $1<p<N$, $p<q_1<p+\frac{p^2}{N}<q_2\leq p^*$ and $a>0$. Then
\begin{enumerate}
\item[$(a)$]\quad $\mu_a^*\in (0,+\infty)$ and the functional $\mu(u)$ is 0-homogeneous for
$(u)_s$, that is, $\mu((u)_s) = \mu(u)$ for all $s > 0$, where $(u)_s=s^{\frac{N}{p}}u(sx)$ is the trajectory on $\mathcal{S}_a$ generated by $u$.
\item[$(b)$]\quad $\mu_a^*=\mu_1^*a^{-\frac{N(q_2-q_1)(q_2-p)}{q_2\gamma_{q_2}(q_2\gamma_{q_2}-p)}}$
is decreasing  in terms of $a > 0$.
\end{enumerate}
\end{proposition}
\begin{proof}
(a) It follows immediately from direct calculations, the Sobolev and the Gagliardo
-Nirenberg inequalities.

\medskip

(b) We first notice that by the conclusion $(a)$,
\begin{equation*}
\mu_a^*=\inf_{u\in \mathcal{S}_a,\|u\|_{q_2}^{q_2}=1}\mu(u)
\end{equation*}
under a suitable choice of $s>0$ in the trajectory on $\mathcal{S}_a$ generated by $u$.
Moreover, for any $u_a\in \mathcal{S}_a$, we define
\begin{equation*}
u_1:=s^{\frac{N}{q_2}}u_a(sx)\quad\text{with}\quad s^{\frac{N}{q_2}p-N}=a^{-p}.
\end{equation*}
Then $u_1\in \mathcal{S}_1$, $\|u_1\|_{q_2}=\|u_a\|_{q_2}$ and
\begin{equation*}
\mu(u_1)=a^{\frac{N(q_2-q_1)(q_2-p)}{q_2\gamma_{q_2}(q_2\gamma_{q_2}-p)}}\mu(u_a)
\end{equation*}
It follows that
\begin{equation*}
\mu_1^*=a^{\frac{N(q_2-q_1)(q_2-p)}{q_2\gamma_{q_2}(q_2\gamma_{q_2}-p)}}\mu_a^*,
\end{equation*}
which completes the proof.
\end{proof}

\section{The existence and multiplicity theory of (\ref{e1.1}) for $\mu<\mu_{a}^*$}
\setcounter{section}{3} \setcounter{equation}{0}
In constructing solutions of \eqref{e1.1} via the minimizaztion method on the Pohozaev manifold, a crucial point is to prove the compactness of the minimizing sequences, which usually requires the energy values $m^{\pm}(a,\mu)$ to be nonincreasing as a function of the mass $a$.
\begin{lemma}\label{lem3.1}
Let $N\geq 2$, $1<p<N$, $p<q_1<p+\frac{p^2}{N}<q_2\leq p^*$, $a>0$ and $\mu\in (0,\mu_a^*)$, where $\mu_a^*$ is the first extremal value of the Pohozaev manifold given by \eqref{e1.12}.  Then
\begin{enumerate}
\item[$(a)$]\quad $m^{-}(a,\mu)\geq m^{+}(a,\mu)$ and $m^{+}(a,\mu)<0$.
\item[$(b)$]\quad For all $u_a\in \mathcal{P}_{a,\mu}^{\pm}$ and $b>a$ such that $\mu_b^*>\mu$, we have
\begin{equation*}
\Psi_{\mu}(u_a)>\Psi_{\mu}\left(\left(\frac{b}{a}u_a\right)_{t_{\mu}^{\pm}\left(\frac{b}{a}u_a\right)}\right).
\end{equation*}
In particular, $m^{\pm}(a,\mu)\geq m^{\pm}(b,\mu)$ and if $m^{\pm}(a,\mu)$ is attained, then we also have that $m^{\pm}(a,\mu)> m^{\pm}(b,\mu)$.
\end{enumerate}
\end{lemma}

\begin{proof}
$(a)$\quad Since $\mu\in (0,\mu_a^*)$, the conclusion follows immediately from the definition of $\mu_a^*$ and $(a)$ of Proposition \ref{pro2.1}.

\medskip

$(b)$\quad Since $\mu\in (0,\mu_b^*)$, by $(a)$ of Proposition \ref{pro2.1} and $(b)$ of Proposition \ref{pro2.2}, there exist $\epsilon>0$ small enough and
$t_{\mu}^{\pm}\left(\frac{c}{a}u_a\right)>0$ such that $\left(\frac{c}{a}u_a\right)_{t_{\mu}^{\pm}\left(\frac{c}{a}u_a\right)}\in \mathcal{P}_{c,\mu}^{\pm}$ for any $c\in (a/2,b+\epsilon)$.  For the sake of simplicity, we denote $v_c:=\frac{c}{a}u_a$ and $t_{\mu}^{\pm}(c):=t_{\mu}^{\pm}(v_c)$.  Then by $\left(v_c\right)_{t_{\mu}^{\pm}(c)}\in \mathcal{P}_{c,\mu}^{\pm}$, we have
\begin{eqnarray*}
H(c,t_{\mu}^{\pm}(c))\equiv 0\quad\text{and}\quad \frac{\partial H(c,\tau)}{\partial\tau}|_{\tau=t_{\mu}^{\pm}(c)}\not=0
\end{eqnarray*}
for all $c\in (a/2,b+\epsilon)$, where
\begin{equation*}
H(c,\tau):=\left(\frac{c}{a}\tau\right)^p\|\nabla u_a\|_p^p-\left(\frac{c}{a}\right)^{q_1}\tau^{q_1\gamma_{q_1}}\mu\gamma_{q_1}\|u_a\|_{q_1}^{q_1}-
\left(\frac{c}{a}\right)^{q_2}\tau^{q_2\gamma_{q_2}}\gamma_{q_2}\|u_a\|_{q_2}^{q_2}.
\end{equation*}
It follows from the implicit function theorem that $\frac{dt_{\mu}^{\pm}(c)}{dc}$ exists for all $c\in (a/2,b+\epsilon)$ and
\begin{equation*}
\begin{split}
\frac{d t_{\mu}^{\pm}(c)}{dc}&=-\frac{\tau\left(
p\left(t_{\mu}^{\pm}(c)\right)^p\|\nabla v_c\|_p^p-\mu q_1\gamma_{q_1}(t_{\mu}^{\pm}(c))^{q_1\gamma_{q_1}}\|v_c\|_{q_1}^{q_1}-q_2\gamma_{q_2}
(t_{\mu}^{\pm}(c))^{q_2\gamma_{q_2}}\|v_c\|_{q_2}^{q_2}\right)}
{c\left(
p\left(t_{\mu}^{\pm}(c)\right)^p\|\nabla v_c\|_p^p-\mu q_1\gamma_{q_1}^2(t_{\mu}^{\pm}(c))^{q_1\gamma_{q_1}}\|v_c\|_{q_1}^{q_1}-q_2\gamma_{q_2}^2
(t_{\mu}^{\pm}(c))^{q_2\gamma_{q_2}}\|v_c\|_{q_2}^{q_2}\right)}\\
& = \frac{\tau\left(
\mu (q_1-p)\gamma_{q_1}(t_{\mu}^{\pm}(c))^{q_1\gamma_{q_1}}\|v_c\|_{q_1}^{q_1}+(q_2-p)\gamma_{q_2}
(t_{\mu}^{\pm}(c))^{q_2\gamma_{q_2}}\|v_c\|_{q_2}^{q_2}\right)}
{c\left(
p\left(t_{\mu}^{\pm}(c)\right)^p\|\nabla v_c\|_p^p-\mu q_1\gamma_{q_1}^2(t_{\mu}^{\pm}(c))^{q_1\gamma_{q_1}}\|v_c\|_{q_1}^{q_1}-q_2\gamma_{q_2}^2
(t_{\mu}^{\pm}(c))^{q_2\gamma_{q_2}}\|v_c\|_{q_2}^{q_2}\right)},
\end{split}
\end{equation*}
which, together with $\left(v_c\right)_{t_{\mu}^{\pm}(c)}\in \mathcal{P}_{c,\mu}^{\pm}$ and $p<q_1<q_2$, implies that $t_{\mu}^{+}(c)$ is increasing and $t_{\mu}^{-}(c)$ is decreasing for all $c\in (a/2,b+\epsilon)$.  Thus, $\Psi_{\mu}\left((v_c)_{t_{\mu}^{\pm}(c)}\right)$ is $C^1$ as a function of $c$ in $(a/2,b+\epsilon)$.  Moreover, by $\left(v_c\right)_{t_{\mu}^{\pm}(c)}\in \mathcal{P}_{c,\mu}^{\pm}$ again and \eqref{e1.8}, we have
\begin{eqnarray*}
\frac{d \Psi_{\mu}\left((v_c)_{t_{\mu}^{\pm}(c)}\right)}{dc}&=&\frac{\partial \Psi_{\mu}\left((v_c)_{t_{\mu}^{\pm}(c)}\right)}{\partial v_c}\frac{d v_c}{dc}+\frac{\partial \Psi_{\mu}\left((v_c)_{t_{\mu}^{\pm}(c)}\right)}{\partial t_{\mu}^{\pm}(c)}\frac{d t_{\mu}^{\pm}(c)}{dc}\\
&=&\frac{1}{c}\left(
\left(t_{\mu}^{\pm}(c)\right)^p\|\nabla v_c\|_p^p-\mu(t_{\mu}^{\pm}(c))^{q_1\gamma_{q_1}}\|v_c\|_{q_1}^{q_1}-
(t_{\mu}^{\pm}(c))^{q_2\gamma_{q_2}}\|v_c\|_{q_2}^{q_2}\right)\\
&=&\frac{1}{c}\left(\mu(\gamma_{q_1}-1)(t_{\mu}^{\pm}(c))^{q_1\gamma_{q_1}}\|v_c\|_{q_1}^{q_1}+(\gamma_{q_2}-1)
(t_{\mu}^{\pm}(c))^{q_2\gamma_{q_2}}\|v_c\|_{q_2}^{q_2}\right)\\
&<&0.
\end{eqnarray*}
It follows that $\Psi_{\mu}\left((v_c)_{t_{\mu}^{\pm}(c)}\right)$ is decreasing in terms of $c\in(a/2,b+\epsilon)$. In particular,
\begin{equation*}
\Psi_{\mu}(u_a)>\Psi_{\mu}\left((v_b)_{t_{\mu}^{\pm}(v_b)}\right)\ \text{for}\ b>a,
\end{equation*}
which implies that
\begin{equation}\label{eWu01}
\Psi_{\mu}(u_a)>\Psi_{\mu}\left((v_b)_{t_{\mu}^{\pm}(v_b)}\right)\geq m^{\pm}(b,\mu).
\end{equation}
By the arbitrariness of $u_a\in \mathcal{P}_{a,\mu}^{\pm}$, we obtain that  $m^{\pm}(a,\mu)\geq m^{\pm}(b,\mu)$.  If $m^{\pm}(a,\mu)$ is attained, then we can choose $u_a^{\pm}\in \mathcal{P}_{a,\mu}^{\pm}$ such that $\Psi_{\mu}\left(u_a^{\pm}\right)=m^{\pm}(a,\mu)$.  It follows from \eqref{eWu01} that $m^{\pm}(a,\mu)> m^{\pm}(b,\mu)$ for $b>a$.
\end{proof}

As stated in the introduction, in dealing with the compactness issue of the minimizing sequence of $\Psi_{\mu}(u)$ on $\mathcal{P}_{a,\mu}^{\pm}$, we shall introduce an additional parameter $\alpha\in(0, 1)$ in front of $\|\nabla u\|_p^p$ in $\Psi_{\mu}(u)$ to simulate the symmetric rearrangement of $u$ by defining
\begin{equation*}
\tilde{\Psi}_{\mu,\alpha}(u):=\frac{1}{p}\alpha\|\nabla u\|_p^p-\frac{\mu}{q_1}\|u\|_{q_1}^{q_1}-\frac{1}{q_2}\|u\|_{q_2}^{q_2}.
\end{equation*}
We define its fibering maps
\begin{equation*}
\tilde{\Phi}_{\mu,\alpha,u}(s):=\tilde{\Psi}_{\mu,\alpha}((u)_s)=\frac{1}{p}\alpha s^p\|\nabla u\|_p^p-\frac{\mu}{q_1}s^{q_1\gamma_{q_1}}\|u\|_{q_1}^{q_1}-\frac{1}{q_2}s^{q_2\gamma_{q_2}}\|u\|_{q_2}^{q_2}
\end{equation*}
as well as its $L^p$-Pohozaev manifold
\begin{equation*}
\tilde{\mathcal{P}}_{a,\mu,\alpha}:=\{u\in \mathcal{S}_a\mid \alpha\|\nabla u\|_p^p=\mu\gamma_{q_1}\|u\|_{q_1}^{q_1}+\gamma_{q_2}\|u\|_{q_2}^{q_2}\}
\end{equation*}
with
\begin{equation*}
\tilde{\mathcal{P}}_{a,\mu,\alpha}^+:=\{u\in \tilde{\mathcal{P}}_{a,\mu,\alpha}\mid \alpha p\|\nabla u\|_p^p>\mu q_1\gamma_{q_1}^2\|u\|_{q_1}^{q_1}+q_2\gamma_{q_2}^2\|u\|_{q_2}^{q_2}\},
\end{equation*}
\begin{equation*}
\tilde{\mathcal{P}}_{a,\mu,\alpha}^0:=\{u\in \tilde{\mathcal{P}}_{a,\mu,\alpha}\mid\alpha  p\|\nabla u\|_p^p=\mu q_1\gamma_{q_1}^2\|u\|_{q_1}^{q_1}+q_2\gamma_{q_2}^2\|u\|_{q_2}^{q_2}\},
\end{equation*}
\begin{equation*}
\tilde{\mathcal{P}}_{a,\mu,\alpha}^-:=\{u\in \tilde{\mathcal{P}}_{a,\mu,\alpha}\mid \alpha p\|\nabla u\|_p^p<\mu q_1\gamma_{q_1}^2\|u\|_{q_1}^{q_1}+q_2\gamma_{q_2}^2\|u\|_{q_2}^{q_2}\}.
\end{equation*}
\begin{lemma}\label{lem3.2}
Let $N\geq 2$, $1<p<N$, $p<q_1<p+\frac{p^2}{N}<q_2\leq p^*$ and $a,\mu,\alpha>0$. Then for any $u\in \mathcal{S}_a$, $\tilde{\Phi}_{\mu,\alpha,u}(s)$ is $C^{\infty}$ in $(0, +\infty)$.  Moreover, if $\mu\in \left(0, \tilde{\mu}_{\alpha}(u)\right)$, where
\begin{equation}\label{mu+alpha}
\tilde{\mu}_{\alpha}(u):=\frac{(q_2\gamma_{q_2}-p)(p-q_1\gamma_{q_1})^{\frac{p-q_1\gamma_{q_1}}{q_2\gamma_{q_2}-p}}}
{\gamma_{q_1}\gamma_{q_2}^{\frac{p-q_1\gamma_{q_1}}{q_2\gamma_{q_2}-p}}(q_2\gamma_{q_2}-q_1\gamma_{q_1})^{\frac{q_2\gamma_{q_2}-q_1\gamma_{q_1}}{q_2\gamma_{q_2}-p}}}
\frac{(\alpha\|\nabla u\|_p^p)^{\frac{q_2\gamma_{q_2}-q_1\gamma_{q_1}}{q_2\gamma_{q_2}-p}}}{\|u\|_{q_1}^{q_1}(\|u\|_{q_2}^{q_2})^{\frac{p-q_1\gamma_{q_1}}{q_2\gamma_{q_2}-p}}},
\end{equation}
then $\tilde{\Phi}_{\mu,\alpha,u}(s)$ only has two critical points $\tilde{t}^\pm_{\mu,\alpha}(u)$ satisfying
\begin{eqnarray*}
0<\tilde{t}^+_{\mu,\alpha}(u)<\tilde{s}_{\alpha}(u)<\tilde{t}^-_{\mu,\alpha}(u)<+\infty
\end{eqnarray*}
such that $(u)_{\tilde{t}^{\pm}_{\mu,\alpha}(u)}\in \tilde{\mathcal{P}}_{a,\mu,\alpha}^{\pm}$,  where $\tilde{t}^+_{\mu,\alpha}(u)$ is the local minimal point of $\tilde{\Phi}_{\mu,\alpha,u}(s)$ in $\left(0,\tilde{t}^-_{\mu,\alpha}(u)\right)$, $\tilde{t}^-_{\mu,\alpha}(u)$ is the local maximal point of
$\tilde{\Phi}_{\mu,\alpha,u}(s)$ in $\left(\tilde{t}^+_{\mu,\alpha}(u),+\infty\right)$ and
\begin{equation*}
\tilde{s}_{\alpha}(u):=\left(\frac{(p-q_1\gamma_{q_1})\alpha \|\nabla u\|_p^p}{\gamma_{q_2}(q_2\gamma_{q_2}-q_1\gamma_{q_1})\|u\|_{q_2}^{q_2}}\right)^{\frac{1}{q_2\gamma_{q_2}-p}}.
\end{equation*}
\end{lemma}
\begin{proof}
The proof is the same as the proof of the conclusion $(a)$ of  Proposition \ref{pro2.1} with the trivial modifications by replacing $\|\nabla u\|_p^p$ by $\alpha\|\nabla u\|_p^p$.  Thus, we omit it here.
\end{proof}

With Lemma~\ref{lem3.2} in hands, we have the following uniform estimates.
\begin{lemma}\label{lem3.3}
Let $N\geq 2$, $1<p<N$, $p<q_1<p+\frac{p^2}{N}<q_2\leq p^*$, $a>0$ and $u\in \mathcal{P}_{a,\mu}^{\pm}$. If $\alpha_0\in (0,1)$ and $\mu>0$ satisfy
$\mu\in \left(0, \tilde{\mu}_{\alpha_0}(u)\right)$, where $\tilde{\mu}_{\alpha_0}(u)$ is given by \eqref{mu+alpha}, then there exists  $\epsilon>0$ small enough such that
\begin{enumerate}
\item[$(a)$]\quad $\tilde{t}^{\pm}_{\mu,\alpha}(u)$ are $C^1$ in terms of $\alpha\in (\alpha_0-\epsilon,1+\epsilon)$ with $\tilde{t}^+_{\mu,\alpha}(u)$ decreasing and $\tilde{t}^-_{\mu,\alpha}(u)$ increasing.
\item[$(b)$]\quad $\tilde{\Psi}_{\mu,\alpha}\left((u)_{\tilde{t}^{\pm}_{\mu,\alpha}(u)}\right)$ is $C^1$ and increasing in terms of $\alpha\in (\alpha_0-\epsilon,1+\epsilon)$. In particular,
\begin{equation*}
\tilde{\Psi}_{\mu,\alpha_0}\left((u)_{\tilde{t}^{\pm}_{\mu,\alpha_0}(u)}\right)<\tilde{\Psi}_{\mu,1}\left((u)_{\tilde{t}^{\pm}_{\mu,1}(u)}\right)=\Psi_{\mu}(u).
\end{equation*}
\end{enumerate}
\end{lemma}
\begin{proof}
$(a)$\quad Since $\tilde{\mu}_{\alpha}(u)$ is increasing as a function of $\alpha$, there exists $\epsilon>0$ small enough such that $\mu\in \left(0, \tilde{\mu}_{\alpha}(u)\right)$ for any $\alpha\in (\alpha_0-\epsilon,1+\epsilon)$.  By Lemma \ref{lem3.2} there exists $\tilde{t}^{\pm}_{\mu,\alpha}(u)$ such that $(u)_{\tilde{t}^{\pm}_{\mu,\alpha}(u)}\in \tilde{\mathcal{P}}_{a,\mu,\alpha}^{\pm}$.  The remaining of the proof is similar to that of Lemma~\ref{lem3.1} by applying the implicit function theorem to the function
\begin{equation*}
H(\alpha,t):=\alpha t^{p}\|\nabla u\|_p^p-\mu\gamma_{q_1}t^{q_1\gamma_{q_1}}\|u\|_{q_1}^{q_1}-\gamma_{q_2}t^{q_2\gamma_{q_2}}\|u\|_{q_2}^{q_2}.
\end{equation*}
Thus, we omit it here.

\medskip

$(b)$\quad By the conclusion of $(a)$, $\tilde{\Psi}_{\mu,\alpha}\left((u)_{\tilde{t}^{\pm}_{\mu,\alpha}(u)}\right)$  is $C^1$  in terms of $\alpha\in (\alpha_0-\epsilon,1+\epsilon)$. Moreover, by $(u)_{\tilde{t}^{\pm}_{\mu,\alpha}(u)}\in \tilde{\mathcal{P}}_{a,\mu,\alpha}^{\pm}$, we also have
\begin{eqnarray*}
\frac{d \tilde{\Psi}_{\mu,\alpha}\left((u)_{\tilde{t}^{\pm}_{\mu,\alpha}(u)}\right)}{d \alpha}
&=&\frac{\partial \tilde{\Psi}_{\mu,\alpha}\left((u)_{\tilde{t}^{\pm}_{\mu,\alpha}(u)}\right)}{\partial \alpha}+\frac{\partial \tilde{\Psi}_{\mu,\alpha}\left((u)_{\tilde{t}^{\pm}_{\mu,\alpha}(u)}\right)}{\partial \tilde{t}^{\pm}_{\mu,\alpha}(u)}\frac{d\left(\tilde{t}^{\pm}_{\mu,\alpha}(u)\right)}{d\alpha}\\
&=&\frac{1}{p}\left(\tilde{t}^{\pm}_{\mu,\alpha}(u)\right)^p\|\nabla u\|_p^p\\
&>&0.
\end{eqnarray*}
Thus, $\tilde{\Psi}_{\mu,\alpha}\left((u)_{\tilde{t}^{\pm}_{\mu,\alpha}(u)}\right)$ is increasing in terms of $\alpha\in (\alpha_0-\epsilon,1+\epsilon)$.
\end{proof}

\subsection{The existence of ground-state solutions} In this subsection, we shall construct a solution of \eqref{e1.1} by studying
the variational problem
\begin{equation}\label{e3.1}
m^{+}(a,\mu):=\inf_{u\in \mathcal{P}_{a,\mu}^{+}}\Psi_{\mu}(u).
\end{equation}
We begin with the following simple property of $m^{+}(a,\mu)$.
\begin{lemma}\label{lem2.1}
Let $N\geq 2$, $1<p<N$ and $p<q_1<p+\frac{p^2}{N}<q_2\leq p^*$. Then $m^{+}(a,\mu)>-\infty$ for all $a,\mu>0$.
\end{lemma}
\begin{proof}
For any $u\in  \mathcal{P}_{a,\mu}^+$, we have
\begin{equation*}
(q_2\gamma_{q_2}-p)\|\nabla u\|_p^p<\mu\gamma_{q_1}(q_2\gamma_{q_2}-q_1\gamma_{q_1})\|u\|_{q_1}^{q_1}.
\end{equation*}
It follows from the Gagliardo-Nirenberg  inequality that
\begin{equation*}
\|\nabla u\|_p^p\lesssim \mu\|u\|_{q_1}^{q_1}\lesssim \mu a^{q_1(1-\gamma_{q_1})}\|\nabla u\|_p^{q_1\gamma_{q_1}},
\end{equation*}
which implies $\|\nabla u\|_p^p\lesssim1$.  Then by $u\in  \mathcal{P}_{a,\mu}^+$ and the Gagliardo-Nirenberg inequality once more, we obtain that
\begin{equation*}
\Psi_{\mu}(u)=\left(\frac{1}{p}-\frac{1}{q_2\gamma_{q_2}}\right)\|\nabla u\|_p^p+\mu\gamma_{q_1}\left(\frac{1}{q_2\gamma_{q_2}}-\frac{1}{q_1\gamma_{q_1}}\right)\|u\|_{q_1}^{q_1}\gtrsim-1,
\end{equation*}
which completes the proof.
\end{proof}

With Lemma~\ref{lem2.1} in hands, we can prove the following result.
\begin{proposition}\label{pro3.1}
Let $N\geq 2$, $1<p<N$, $p<q_1<p+\frac{p^2}{N}<q_2\leq p^*$, $a>0$ and $0<\mu<\mu_{a}^*$. Then the variational problem (\ref{e3.1}) is achieved by some
$u_{a,\mu,+}$, which is real valued, positive, radially symmetric and radially decreasing. Moreover, $u_{a,\mu,+}$ also satisfies the equation (\ref{e1.1}) in the weak sense
for a suitable Lagrange multiplier $\lambda_{a,\mu,+}<0$.
\end{proposition}
\begin{proof}
Since $m^{+}(a,\mu)>-\infty$ by Lemma~\ref{lem2.1}, we can choose $\{u_n\}\subset \mathcal{P}_{a,\mu}^+$ be a minimizing sequence of $m^+(a,\mu)$.   We recall that the Schwarz symmetric rearrangement of $\{u_n\}$, say $\{u_n^*\}$, is real valued, nonnegative, radially symmetric, radially decreasing and satisfies
\begin{equation*}
\|u_n^*\|_q=\|u_n\|_q\ \text{for}\ q\in [p,p^*],\ \text{and}\  \|\nabla u_n^*\|_p\leq \|\nabla u_n\|_p.
\end{equation*}
Denote $\|\nabla u_n^*\|_p^p=\alpha_n \|\nabla u_n\|_p^p$.  If $\alpha_n=1$, then  $\{u_n^*\}\subset \mathcal{P}_{a,\mu}^+$.  If $\alpha_n<1$, then by $\mu<\mu(u_n^*)=\tilde{\mu}_{\alpha_n}(u_n)$ and Lemma \ref{lem3.3}, we obtain that
\begin{equation*}
\Psi_{\mu}\left((u_n^*)_{t^+_\mu(u_n^*)}\right)=\tilde{\Psi}_{\mu,\alpha_n}\left((u_n)_{\tilde{t}^{+}_{\mu,\alpha_n}(u_n)}\right)<\tilde{\Psi}_{\mu,1}\left((u_n)_{\tilde{t}^{+}_{\mu,1}(u_n)}\right)=\Psi_{\mu}(u_n).
\end{equation*}
Hence,  $\left\{(u_n^*)_{t^+_\mu(u_n^*)}\right\}\subset \mathcal{P}_{a,\mu}^+$ is also a minimizing sequence of $m^+(a,\mu)$ and without loss of generality, we may assume that $\{u_n\}$ is real valued, nonnegative, radially symmetric and radially decreasing.  Moreover, we also have
\begin{equation}\label{e3.6}
m^{+}(a,\mu)=m^{+}_{rad}(a,\mu),
\end{equation}
where $m^{+}_{rad}(a,\mu)$ is given by \eqref{radial}.  Since $\{u_n\}\subset \mathcal{P}_{a,\mu}^+$, that is,
\begin{equation*}
\begin{cases}
&\|\nabla u_n\|_p^p=\mu\gamma_{q_1}\|u_n\|_{q_1}^{q_1}+\gamma_{q_2}\|u_n\|_{q_2}^{q_2},\\
&p\|\nabla u_n\|_p^p>\mu q_1\gamma_{q_1}^2\|u_n\|_{q_1}^{q_1}+q_2\gamma_{q_2}^2\|u_n\|_{q_2}^{q_2},
\end{cases}
\end{equation*}
we obtain that
\begin{equation}\label{e3.5}
\left\{\aligned
&\gamma_{q_2}(q_2\gamma_{q_2}-q_1\gamma_{q_1})\|u_n\|_{q_2}^{q_2}<(p-q_1\gamma_{q_1})\|\nabla u_n\|_p^p,\\
&(q_2\gamma_{q_2}-p)\|\nabla u_n\|_p^p<\mu\gamma_{q_1}(q_2\gamma_{q_2}-q_1\gamma_{q_1})\|u_n\|_{q_1}^{q_1},\\
&\gamma_{q_2}(q_2\gamma_{q_2}-p)\|u_n\|_{q_2}^{q_2}<\mu\gamma_{q_1}(p-q_1\gamma_{q_1})\|u_n\|_{q_1}^{q_1}.
\endaligned\right.
\end{equation}
Since $q_1\gamma_{q_1}<p$, by the Gagliardo-Nirenberg  inequality,
\begin{equation}\label{e3.56}
\|\nabla u_n\|_p^p\lesssim \mu\|u_n\|_{q_1}^{q_1}\lesssim \mu a^{q_1(1-\gamma_{q_1})}\|\nabla u_n\|_p^{q_1\gamma_{q_1}},
\end{equation}
which implies $\|\nabla u_n\|_p^p\lesssim 1$.  Hence, $\{u_n\}$ is bounded in $W^{1,p}\left(\mathbb{R}^N\right)$. Since $m^+(a,\mu)<0$, we obtain that $\|\nabla u_n\|_p^p\gtrsim 1$, which, together with (\ref{e3.5}), also implies that $\|u_n\|_{q_1}^{q_1}\gtrsim 1$.  Moreover, if $\|u_n\|_{q_2}^{q_2}\to0$ as $n\to+\infty$, then by the H\"older inequality, we also have $\|u_n\|_{q_1}^{q_1}\to0$ as $n\to+\infty$, which is impossible.  Thus, we also have $\|u_n\|_{q_2}^{q_2}\gtrsim 1$.  We claim that there exists $\delta>0$ sufficiently small such that
\begin{equation}\label{e3.7}
\inf_{\mathcal{P}_{a,\mu}^{\delta,+}}\Psi_{\mu}(u)=\inf_{\mathcal{P}_{a,\mu}^{+}}\Psi_{\mu}(u)=m^{+}(a,\mu),
\end{equation}
where $\mathcal{P}_{a,\mu}^{\delta,+}:=\left\{u\in \mathcal{S}_a\mid \text{dist}_{W^{1,p}}(u,\mathcal{P}_{a,\mu}^{+})\leq \delta\right\}$ with
\begin{equation*}
\text{dist}_{W^{1,p}}\left(u,\mathcal{P}_{a,\mu}^{+}\right)=\inf_{v\in \mathcal{P}_{a,\mu}^{+}}\|u-v\|_{W^{1,p}}
\end{equation*}
and $\|\cdot\|_{W^{1,p}}=\left(\|\nabla\cdot\|_{p}^p+\|\cdot\|_{p}^p\right)^{\frac{1}{p}}$ the usual norm in the Sobolev space $W^{1,p}\left(\mathbb{R}^N\right)$.
It is obvious that
\begin{equation*}
\inf_{\mathcal{P}_{a,\mu}^{\delta,+}}\Psi_{\mu}(u)\leq \inf_{\mathcal{P}_{a,\mu}^{+}}\Psi_{\mu}(u)
\end{equation*}
for any $\delta>0$.  Suppose the contrary.  Then by Lemma~\ref{lem2.1}, there exists $\delta_n\to0$ as $n\to+\infty$,  $\varphi_n\in \mathcal{S}_a$ and $\phi_n\in \mathcal{P}_{a,\mu}^+$ such that $\varphi_n-\phi_n\to0$ strongly in $W^{1,p}\left(\mathbb{R}^N\right)$ as $n\to+\infty$ and
\begin{eqnarray*}
\Psi_\mu(\varphi_n)<m^+(a,\mu)
\end{eqnarray*}
for all $n$.  Since $0<\mu<\mu^{*}_{a}$, by the conclusion $(a)$ of Proposition~\ref{pro2.1}, there exists $t_{\mu}^+(\varphi_n)>0$ such that $\left(\varphi_n\right)_{t^+_{\mu}(\varphi_n)}\in\mathcal{P}_{a,\mu}^+$.  It follows from $\phi_n\in \mathcal{P}_{a,\mu}^+$ and $\varphi_n-\phi_n\to0$ strongly in $W^{1,p}\left(\mathbb{R}^N\right)$ as $n\to+\infty$ that
\begin{equation}\label{026}
\left\{\aligned
&\left(t_{\mu}^+\left(\varphi_n\right)\right)^{p-q_1\gamma_{q_1}}\|\nabla \varphi_n\|_p^p=\mu\gamma_{q_1}\|\varphi_n\|_{q_1}^{q_1}+\left(t_{\mu}^+\left(\varphi_n\right)\right)^{q_2\gamma_{q_2}-q_1\gamma_{q_1}}\gamma_{q_2}\|\varphi_n\|_{q_2}^{q_2},
\\
&\|\nabla \varphi_n\|_p^p=\mu\gamma_{q_1}\|\varphi_n\|_{q_1}^{q_1}+\gamma_{q_2}\|\varphi_n\|_{q_2}^{q_2}+o(1).
\endarray\right.
\end{equation}
Since $\{\phi_n\}$ is a minimizing sequence of $m^+(a,\mu)$ and $\varphi_n-\phi_n\to0$ strongly in $W^{1,p}\left(\mathbb{R}^N\right)$ as $n\to+\infty$,
as that for $\{u_n\}$, we know that $\|\varphi_n\|_{q_1}\thickapprox\|\varphi_n\|_{q_2}\thickapprox\|\nabla \varphi_n\|_p\thickapprox1$, which, together with \eqref{026}, implies that $t_{\mu}^+\left(\varphi_n\right)\thickapprox1$.  Thus, up to a subsequence, we may assume that $t_{\mu}^+\left(\varphi_n\right)\to t_{\mu}$ as $n\to+\infty$.  We also denote $A=\lim_{n\to+\infty}\|\nabla \phi_n\|_p^p$, $B=\|\phi_n\|_{q_1}^{q_1}$ and $C=\|\phi_n\|_{q_2}^{q_2}$.  Then by \eqref{026}, $\|\varphi_n\|_{q_1}\thickapprox\|\varphi_n\|_{q_2}\thickapprox\|\nabla \varphi_n\|_p\thickapprox1$ and $\varphi_n-\phi_n\to0$ strongly in $W^{1,p}\left(\mathbb{R}^N\right)$ as $n\to+\infty$, we have
\begin{eqnarray*}
\left(t_{\mu}^+\right)^{p}A=\left(t_{\mu}^+\right)^{q_1\gamma_{q_1}}\mu\gamma_{q_1}B+\left(t_{\mu}^+\right)^{q_2\gamma_{q_2}}\gamma_{q_2}C.
\end{eqnarray*}
Since $\phi_n\in \mathcal{P}_{a,\mu}^+$, we also have
\begin{eqnarray*}
\left\{\aligned
&A=\mu\gamma_{q_1}B+\gamma_{q_2}C,\\
&pA\geq\mu q_1\gamma_{q_1}^2B+q_2\gamma_{q_2}^2C.
\endaligned\right.
\end{eqnarray*}
By similar computations in the proof of $(a)$ of Proposition~\ref{pro2.1}, we know that $t_{\mu}^+=1$.  Moreover, since $0<\mu<\mu^{*}_{a}\leq \mu(\phi_n)$, by similar computations in the proof of $(a)$ of Proposition~\ref{pro2.1}, we must have $t_{\mu}^-\left(\varphi_n\right)-1\gtrsim1$, where $t_{\mu}^-\left(\varphi_n\right)$ is given in Proposition~\ref{pro2.1}.  Thus, by Proposition~\ref{pro2.1} again,
\begin{equation*}
m^+(a,\mu)\le\Psi_\mu\left((\varphi_n)_{t^+_{\mu}(\varphi_n)}\right)\le\Psi_\mu(\varphi_n)<m^+(a,\mu)
\end{equation*}
for $n$ sufficiently large, which is impossible.  Thus, \eqref{e3.7} holds true for $\delta>0$ sufficiently small.  Similarly, we also have
\begin{equation*}\label{e3.8}
\inf_{\mathcal{P}_{a,\mu}^{\delta,+}\cap W^{1,p}_{rad}\left(\mathbb{R}^N\right)}\Psi_{\mu}(u)=\inf_{\mathcal{P}_{a,\mu}^{+}\cap W^{1,p}_{rad}\left(\mathbb{R}^N\right)}\Psi_{\mu}(u)=m^{+}_{rad}(a,\mu),
\end{equation*}
which, together with (\ref{e3.6}), implies that
\begin{equation}\label{e3.9}
\inf_{\mathcal{P}_{a,\mu}^{\delta,+}\cap W^{1,p}_{rad}\left(\mathbb{R}^N\right)}\Psi_{\mu}(u)=m^{+}(a,\mu)
\end{equation}
for $\delta>0$ sufficiently small.  Since $\mathcal{P}_{a,\mu}^{\delta,+}\cap W^{1,p}_{rad}\left(\mathbb{R}^N\right)$ is closed in the $W^{1,p}\left(\mathbb{R}^N\right)$ topology, $\Psi_{\mu}(u)$ is bounded below by (\ref{e3.9}) and Lemma~\ref{lem2.1}, and $\{u_n\}$ is a minimizing sequence, by Ekeland's variational principle, there exists $\{w_n\}\subset \mathcal{P}_{a,\mu}^{\delta,+}\cap W^{1,p}_{rad}\left(\mathbb{R}^N\right)$ such that $\{w_n\}$ is a $(PS)_{m^{+}(a,\mu)}$ sequence of $\Psi_{\mu}(u)|_{\mathcal{S}_a}$ and $w_n-u_n\to 0$ strongly in $W^{1,p}\left(\mathbb{R}^N\right)$ as $n\to +\infty$.  By Lemma \ref{lem1.3}, there exists $u_0\in W^{1,p}_{rad}\left(\mathbb{R}^N\right)$ such that
\begin{eqnarray}\label{e3.10}
\left\{\aligned
&w_n\rightharpoonup u_0\quad \text{weakly in } W^{1,p}\left(\mathbb{R}^N\right),\\
&w_n\to u_0\quad \text{a.e. in } \mathbb{R}^N,\\
&\nabla w_n\to \nabla u_0\quad \text{a.e. in}\ \mathbb{R}^N,\\
&w_n\to u_0\quad \text{strongly in } L^t\left(\mathbb{R}^N\right)\text{ for all } t\in \left(p,p^*\right)
\endaligned\right.
\end{eqnarray}
as $n\to+\infty$.
By $\|u_n\|_{q_1}^{q_1}\gtrsim 1$, $w_n-u_n\to 0$ strongly in $W^{1,p}\left(\mathbb{R}^N\right)$ and (\ref{e3.10}), we obtain that $\|u_0\|_{q_1}^{q_1}\gtrsim 1$ and hence $u_0\neq 0$. It follows from the method of the Lagrange multiplier that there exists $\lambda_n\in \mathbb{R}$ such that
\begin{eqnarray*}
\int_{\mathbb{R}^N}|\nabla w_n|^{p-2}\nabla w_n\cdot\nabla\varphi dx&=&\lambda_n\int_{\mathbb{R}^N}|w_n|^{p-2}w_n\varphi dx+\mu\int_{\mathbb{R}^N}|w_n|^{q_1-2}w_n\varphi dx\\
&&+\int_{\mathbb{R}^N}|w_n|^{q_2-2}w_n\varphi dx+o_n(1)\|\varphi\|_{W^{1,p}}
\end{eqnarray*}
for any $\varphi\in W^{1,p}\left(\mathbb{R}^N\right)$, where $\|\cdot\|_{W^{1,p}}=\left(\|\nabla\cdot\|_{p}^p+\|\cdot\|_{p}^p\right)^{\frac{1}{p}}$ is the usual norm in the Sobolev space $W^{1,p}\left(\mathbb{R}^N\right)$.  Thus, by $\{w_n\}\subset \mathcal{S}_a$, $\{u_n\}\subset \mathcal{P}_{a,\mu}^+$ and $w_n-u_n\to 0$ strongly in $W^{1,p}\left(\mathbb{R}^N\right)$ as $n\to +\infty$, we obtain that
\begin{equation*}
\begin{split}
\lambda_n a^p&=\|\nabla w_n\|_p^p-\mu \|w_n\|_{q_1}^{q_1}-\|w_n\|_{q_2}^{q_2}+o_n(1)\\
&=\|\nabla u_n\|_p^p-\mu \|u_n\|_{q_1}^{q_1}-\|u_n\|_{q_2}^{q_2}+o_n(1)\\
&=\mu(\gamma_{q_1}-1) \|u_n\|_{q_1}^{q_1}+(\gamma_{q_2}-1)\|u_n\|_{q_2}^{q_2}+o_n(1)\\
&\leq  \mu(\gamma_{q_1}-1) \|u_0\|_{q_1}^{q_1}+o_n(1),
\end{split}
\end{equation*}
which, together with $u_0\neq 0$ and $\gamma_{q_1}<1$ for $p<q_1<p+\frac{p^2}{N}$, implies that $\lambda_n\to \lambda_0<0$ as $n\to +\infty$ up to a subsequence. It follows that $u_0$ is a weak solution of the following equation
\begin{equation*}
-\Delta_pu=\lambda_0 |u|^{p-2}u+\mu |u|^{q_1-2}u+|u|^{q_2-2}u,\ \text{in}\  \mathbb{R}^N.
\end{equation*}
By Lemma \ref{lem1.5} and the Fatou lemma, we have $u_0\in \mathcal{P}_{a_1,\mu}$, where $a_:=\|u_0\|_p\in(0, a]$.
Let $v_n:=w_n-u_0$. Recall that $w_n-u_n\to 0$ strongly in $W^{1,p}\left(\mathbb{R}^N\right)$ as $n\to +\infty$ and $\{u_n\}\subset \mathcal{P}_{a,\mu}^+$.  Thus, by (\ref{e3.10}) and Lemma \ref{lem1.1}, we obtain that
\begin{equation*}
\|\nabla v_n\|_{p}^{p}=\gamma_{q_2}\|v_n\|_{q_2}^{q_2}+o_n(1).
\end{equation*}
If $q_2<p^*$, then we have $\|\nabla v_n\|_{p}^{p}=o_n(1)$.  If $q_2=p^*$, then by the Sobolev inequality, we have
\begin{equation*}
\|\nabla v_n\|_{p}^{p}=\|v_n\|_{p^*}^{p^*}+o_n(1)\leq S^{-p^*/p}\|\nabla v_n\|_{p}^{p^*}+o_n(1).
\end{equation*}
Thus, either $\|\nabla v_n\|_{p}^{p}=\|v_n\|_{p^*}^{p^*}=o_n(1)$ or $\|\nabla v_n\|_{p}^{p}=\|v_n\|_{p^*}^{p^*}\geq S^{N/p}+o_n(1)$.
It follows Lemma \ref{lem1.1} and $(b)$ of Lemma \ref{lem3.1} that
\begin{equation*}
\begin{split}
m^+(a,\mu)&=\lim_{n\to +\infty}\Psi_{\mu}(w_n)\\
&=\Psi_{\mu}(u_0)+\lim_{n\to +\infty}\left(\frac{1}{p}\|\nabla v_n\|_{p}^{p}-\frac{1}{q_2}\|v_n\|_{q_2}^{q_2}+o_n(1)\right)\\
&\left\{
\begin{array}{ll}
\geq m^+(a,\mu),&\ \text{if}\ a_1=a,\\
>\Psi_{\mu}\left(\left(\frac{a}{a_1}u_0\right)_{t^+_{\mu}\left(\frac{a}{a_1}u_0\right)}\right)\geq m^+(a,\mu),          &\ \text{if}\ a_1<a
\end{array}
\right.
\end{split}
\end{equation*}
which implies that $a_1=a$, $\Psi_{\mu}(u_0)=m^+(a,\mu)$ and $\|\nabla v_n\|_{p}^{p}=o_n(1)$. Thus, $w_n\to u_0$ strongly in $W^{1,p}\left(\mathbb{R}^N\right)$ as $n\to +\infty$, which, together with  $w_n-u_n\to 0$ strongly in $W^{1,p}\left(\mathbb{R}^N\right)$ as $n\to +\infty$, implies that $u_n\to u_0$ strongly in $W^{1,p}\left(\mathbb{R}^N\right)$. By  $\{u_n\}\subset \mathcal{P}_{a,\mu}^+$ and $0<\mu<\mu_a^*$, we obtain that $u_0\in \mathcal{P}^+_{a,\mu}$. Moreover, $u_0$ is real valued, nonnegative, radially symmetric and radially decreasing, which, together with the strong comparison principle (\cite[Theorem~1.4]{D98}), implies that $u_0$ is also positive in $\mathbb{R}^N$. Thus, $u_{a,\mu,+}:=u_0$ is the required function and $\lambda_{a,\mu,+}:=\lambda_0<0$ is its Lagrange multiplier.
\end{proof}

\subsection{The existence of mountain-pass solutions} In this section, we shall construct another solution of \eqref{e1.1} by studying
the variational problem
\begin{equation}\label{e3.11}
m^{-}(a,\mu):=\inf_{u\in \mathcal{P}_{a,\mu}^{-}}\Psi_{\mu}(u).
\end{equation}
We first consider \eqref{e3.11} in the Sobolev subcritical case $q_2<p^*$.
\begin{proposition}\label{pro3.21}
Let $N\geq 2$, $1<p<N$, $p<q_1<p+\frac{p^2}{N}<q_2<p^*$, $a>0$ and $0<\mu<\mu_{a}^*$. Then the variational problem (\ref{e3.11}) is achieved by some
$u_{a,\mu,-}$, which is real valued, positive, radially symmetric and radially decreasing. Moreover, $u_{a,\mu,-}$ also satisfies the equation (\ref{e1.1}) in the weak sense
for a suitable Lagrange multiplier $\lambda_{a,\mu,-}<0$.
\end{proposition}
\begin{proof}
Since $m^{-}(a,\mu)>-\infty$ by Lemmas~\ref{lem3.1} and \ref{lem2.1}, we can choose $\{u_n\}\subset \mathcal{P}_{a,\mu}^-$ such that it is a minimizing sequence of $m^-(a,\mu)$.   Similarly to that in the proof of Proposition~\ref{pro3.1}, we may assume that $\{u_n\}$ is real valued, nonnegative, radially symmetric and radially decreasing.  Moreover, we also have
\begin{equation*}\label{e3.66}
m^{-}(a,\mu)=m^{-}_{rad}(a,\mu),
\end{equation*}
where $m^{-}_{rad}(a,\mu)$ is given by \eqref{radial}.  Since $\{u_n\}\subset \mathcal{P}_{a,\mu}^-$, that is,
\begin{eqnarray*}
\left\{\aligned
&\|\nabla u_n\|_p^p=\mu\gamma_{q_1}\|u_n\|_{q_1}^{q_1}+\gamma_{q_2}\|u_n\|_{q_2}^{q_2},\\
&p\|\nabla u_n\|_p^p<\mu q_1\gamma_{q_1}^2\|u_n\|_{q_1}^{q_1}+q_2\gamma_{q_2}^2\|u_n\|_{q_2}^{q_2},
\endaligned\right.
\end{eqnarray*}
we obtain that
\begin{eqnarray}\label{e3.55}
\left\{\aligned
&\gamma_{q_2}(q_2\gamma_{q_2}-q_1\gamma_{q_1})\|u_n\|_{q_2}^{q_2}>(p-q_1\gamma_{q_1})\|\nabla u_n\|_p^p,\\
&(q_2\gamma_{q_2}-p)\|\nabla u_n\|_p^p>\mu\gamma_{q_1}(q_2\gamma_{q_2}-q_1\gamma_{q_1})\|u_n\|_{q_1}^{q_1},\\
&\gamma_{q_2}(q_2\gamma_{q_2}-p)\|u_n\|_{q_2}^{q_2}>\mu\gamma_{q_1}(p-q_1\gamma_{q_1})\|u_n\|_{q_1}^{q_1}.
\endaligned\right.
\end{eqnarray}
By the Gagliardo-Nirenberg  inequality,
\begin{equation}\label{e3.12}
\|\nabla u_n\|_p^p\lesssim \|u_n\|_{q_2}^{q_2}\lesssim a^{q_2(1-\gamma_{q_2})}\|\nabla u_n\|_p^{q_2\gamma_{q_2}},
\end{equation}
which implies $\|\nabla u_n\|_p^p\gtrsim 1$ and $\|u_n\|_{q_2}^{q_2}\gtrsim 1$. By $\{u_n\}\subset \mathcal{P}_{a,\mu}^-$ and $\Psi_{\mu}(u_n)=m^-(a,\mu)+o_n(1)$, we obtain that
\begin{equation*}
m^-(a,\mu)+o_n(1)=\left(\frac{1}{p}-\frac{1}{q_2\gamma_{q_2}}\right)\|\nabla u_n\|_p^p+\mu\gamma_{q_1}\left(\frac{1}{q_2\gamma_{q_2}}-\frac{1}{q_1\gamma_{q_1}}\right)\|u_n\|_{q_1}^{q_1}.
\end{equation*}
Thus, by the Gagliardo-Nirenberg  inequality,
\begin{eqnarray}\label{e3.13}
\left(\frac{1}{p}-\frac{1}{q_2\gamma_{q_2}}\right)\|\nabla u_n\|_p^p&=&\Psi_{\mu}(u_n)+\mu\gamma_{q_1}\left(\frac{1}{q_1\gamma_{q_1}}-\frac{1}{q_2\gamma_{q_2}}\right)\|u_n\|_{q_1}^{q_1}\notag\\
&\lesssim&1+ \mu a^{q_1(1-\gamma_{q_1})}\|\nabla u_n\|_p^{q_1\gamma_{q_1}},
\end{eqnarray}
which, together with $q_1\gamma_{q_1}<p<q_2\gamma_{q_2}$ and $u_n\in \mathcal{S}_a$, implies that  $\{u_n\}$ is bounded in $W^{1,p}\left(\mathbb{R}^N\right)$.  Moreover, if $\|u_n\|_{q_1}^{q_1}\to0$ as $n\to+\infty$ then by the H\"older inequality, $\|u_n\|_{q_2}^{q_2}\to0$ as $n\to+\infty$, which is impossible.  Thus, we also have $\|u_n\|_{q_1}^{q_1}\gtrsim 1$.  Clearly, there exists $u_0\in W^{1,p}_{rad}\left(\mathbb{R}^N\right)$ such that
\begin{eqnarray}\label{e3.100}
\left\{\aligned
&u_n\rightharpoonup u_0\quad \text{weakly in } W^{1,p}\left(\mathbb{R}^N\right),\\  &u_n\to u_0\quad \text{a.e. in } \mathbb{R}^N,\\
&u_n\to u_0\quad \text{strongly in } L^t\left(\mathbb{R}^N\right) \text{ for all } t\in (p,p^*)
\endaligned\right.
\end{eqnarray}
as $n\to+\infty$.  Since $q_2<p^*$, by $\|u_n\|_{q_2}^{q_2}\gtrsim 1$, we obtain that $\|u_0\|_{q_2}^{q_2}\gtrsim 1$, which implies that $u_0\neq 0$.  By the Fatou lemma, we have $\mu<\mu_a^*\leq \mu_{a_1}^*\leq \mu(u_0)$, which, together with Proposition \ref{pro2.1}, implies that there exists $t^-_\mu\left(u_0\right)>0$ such that $\left(u_0\right)_{t^-_\mu\left(u_0\right)}\in\mathcal{P}_{a_1,\mu}^-$.  Thus, by \eqref{e3.100},
\begin{eqnarray*}
0&=&\left(t^-_\mu\left(u_0\right)\right)^{p}\|\nabla u_0\|_p^p-\left(t^-_\mu\left(u_0\right)\right)^{q_1\gamma_{q_1}}\mu\gamma_{q_1}\|u_0\|_{q_1}^{q_1}-\left(t^-_\mu\left(u_0\right)\right)^{q_2\gamma_{q_2}}\gamma_{q_2}\|u_0\|_{q_2}^{q_2}\\
&\leq&\left(t^-_\mu\left(u_0\right)\right)^{p}\|\nabla u_n\|_p^p-\left(t^-_\mu\left(u_0\right)\right)^{q_1\gamma_{q_1}}\mu\gamma_{q_1}\|u_n\|_{q_1}^{q_1}-\left(t^-_\mu\left(u_0\right)\right)^{q_2\gamma_{q_2}}\gamma_{q_2}\|u_n\|_{q_2}^{q_2}+o_n(1),
\end{eqnarray*}
which, together with $\{u_n\}\subset \mathcal{P}_{a,\mu}^-$, $\|\nabla u_n\|_p^p\thickapprox\|u_n\|_{q_2}^{q_2}\thickapprox\|u_n\|_{q_1}^{q_1}\thickapprox 1$ and Proposition~\ref{pro2.1}, implies that
\begin{eqnarray*}
1\lesssim t^+_{\mu}(u_n)+o_n(1)\leq t_{\mu}^-(u_0)\leq1+o_n(1).
\end{eqnarray*}
It follows from Proposition~\ref{pro2.1} and $(b)$ of Lemma \ref{lem3.1} that
\begin{equation*}
\begin{split}
m^-(a,\mu)&=\lim_{n\to +\infty}\Psi_{\mu}(u_n)\\
&\geq\lim_{n\to +\infty}\Psi_{\mu}\left(\left(u_n\right)_{t^-_\mu\left(u_0\right)}\right)\\
&\geq\Psi_{\mu}\left(\left(u_0\right)_{t^-_\mu\left(u_0\right)}\right)\\
&\left\{
\begin{array}{ll}
\geq m^-(a,\mu),&\ \text{if}\ a_1=a,\\
>\Psi_{\mu}\left(\left(\frac{a}{a_1}u_0\right)_{t^-_{\mu}\left(\frac{a}{a_1}u_0\right)}\right)\geq m^-(a,\mu),          &\ \text{if}\ a_1<a,
\end{array}
\right.
\end{split}
\end{equation*}
where $a_1=\|u_0\|_p$.  As that in the proof of Proposition~\ref{pro3.1}, we have $a_1=a$ and $u_n\to u_0$ strongly in $W^{1,p}\left(\mathbb{R}^N\right)$ as $n\to+\infty$.  By  $\{u_n\}\subset \mathcal{P}_{a,\mu}^-$ and $0<\mu<\mu_a^*$, we obtain that $u_0\in \mathcal{P}^-_{a,\mu}$ and $\mathcal{P}_{a,\mu}^{-}$ is achieved by $u_0$.  Since $\mathcal{P}^-_{a,\mu}$ is a natural constraint in $\mathcal{S}_a$, $u_0$ is a solution of \eqref{e1.1} for a   suitable Lagrange multiplier $\lambda_{0}\in\mathbb{R}$, which, together with \eqref{e1.8},  implies that
\begin{eqnarray*}
\lambda_0 a^p=\mu(\gamma_{q_1}-1) \|u_0\|_{q_1}^{q_1}+(\gamma_{q_2}-1)\|u_0\|_{q_2}^{q_2}<0.
\end{eqnarray*}
Moreover, $u_0$ is real valued, nonnegative, radially symmetric and radially decreasing, which, together with the strong comparison principle (\cite[Theorem~1.4]{D98}), implies that $u_0$ is also positive in $\mathbb{R}^N$. Thus, $u_{a,\mu,-}:=u_0$ is the required function and $\lambda_{a,\mu,-}:=\lambda_0<0$ is its Lagrange multiplier.
\end{proof}

We next consider \eqref{e3.11} for the Sobolev critical case $q_2=p^*$ by using the subcritical perturbation argument.  For this purpose, we re-denote $\mu_{a}^*$, $m^-(a,\mu)$, $\Psi_{\mu}(u)$, $\mathcal{P}_{a,\mu}^-$, $t^-_\mu\left(u\right)$, $\lambda_{a, \mu,-}$ and $u_{a,\mu,-}$ in Proposition~\ref{pro3.21} for $q_2<p^*$ by $\mu_{a,q_2}^*$, $m^-_{q_2}(a,\mu)$, $\Psi_{\mu,q_2}(u)$, $\mathcal{P}_{a,\mu,q_2}^-$, $t^-_{\mu,q_2}\left(u\right)$,  $\lambda_{a, \mu,q_2,-}$ and $u_{a,\mu,q_2,-}$, respectively. For $q_2=p^*$, we still use the original notations.
\begin{lemma}\label{lemN3.9}
Let $N\geq 2$, $1<p<N$, $p<q_1<p+\frac{p^2}{N}<q_2<p^*$ and $a>0$.  Then
\begin{enumerate}
\item[$(1)$]\quad $\lim_{q_2\to p^*}\mu_{a,q_2}^*=\mu_{a}^*$.
\item[$(2)$]\quad $\lim_{q_2\to p^*}m^-_{q_2}(a,\mu)\leq m^-(a,\mu)$.
\end{enumerate}
\end{lemma}
\begin{proof}
$(1)$\quad Recall that $\mu_{a,q_2}^*=\inf_{u\in \mathcal{S}_a}\mu_{q_2}(u)$, where
\begin{equation*}
\mu_{q_2}(u)=\frac{(q_2\gamma_{q_2}-p)(p-q_1\gamma_{q_1})^{\frac{p-q_1\gamma_{q_1}}{q_2\gamma_{q_2}-p}}}
{\gamma_{q_1}\gamma_{q_2}^{\frac{p-q_1\gamma_{q_1}}{q_2\gamma_{q_2}-p}}(q_2\gamma_{q_2}-q_1\gamma_{q_1})^{\frac{q_2\gamma_{q_2}-q_1\gamma_{q_1}}{q_2\gamma_{q_2}-p}}}
\frac{(\|\nabla u\|_p^p)^{\frac{q_2\gamma_{q_2}-q_1\gamma_{q_1}}{q_2\gamma_{q_2}-p}}}{\|u\|_{q_1}^{q_1}(\|u\|_{q_2}^{q_2})^{\frac{p-q_1\gamma_{q_1}}{q_2\gamma_{q_2}-p}}}.
\end{equation*}
Then by noticing that $\lim_{q_2\to p^*}\|u\|_{q_2}^{q_2}=\|u\|_{p^*}^{p^*}$ for any fixed $u\in \mathcal{S}_a$, we have $\lim_{q_2\to p^*}\mu_{a,q_2}^*\leq\mu_{a}^*$.  On the other hand, it is easy to prove that $\mu_{a,q_2}^*$ is attained by some $v_{q_2}\in \mathcal{S}_a$.  Thus, by re-normalized if necessary, we may assume that $\|v_{q_2}\|_{p^*}=1$.  It follows from the H\"older inequality that $\lim_{q_2\to p^*}\mu_{a,q_2}^*\geq\mu_{a}^*$.  Thus, we must have $\lim_{q_2\to p^*}\mu_{a,q_2}^*=\mu_{a}^*$.

\medskip

$(2)$\quad Since $\lim_{q_2\to p^*}\|u\|_{q_2}^{q_2}=\|u\|_{p^*}^{p^*}$ for any fixed $u\in \mathcal{S}_a$, by Proposition~\ref{pro2.1}, we have
\begin{eqnarray*}
\lim_{q_2\to p^*}t^-_{\mu,q_2}\left(u\right)=1
\end{eqnarray*}
for any $u\in\mathcal{P}_{a,\mu}^-$.  The conclusion then follows from testing $m^-_{q_2}(a,\mu)$ by $u\in\mathcal{P}_{a,\mu}^-$ and letting $q_2\to p^*$.
\end{proof}

As stated in the introduction, a crucial point in solving \eqref{e3.11} for the Sobolev critical case $q_2=p^*$ is the following energy estimate.
\begin{lemma}\label{lem3.6}
Let $N\geq 2$, $1<p<N$ and $p<q_1<p+\frac{p^2}{N}<q_2=p^*$. Then
\begin{equation*}
m^{-}(a,\mu)<m^{+}(a,\mu)+\frac{1}{N}S^{\frac{N}{p}}
\end{equation*}
for any $a>0$ and  $0<\mu<\mu_{a}^*$.
\end{lemma}
The proof of Lemma~\ref{lem3.6} is very involved and technical.  Thus, we shall delay it to the Appendix~B.  With Lemma~\ref{lem3.6} in hands, we can prove the following proposition.
\begin{proposition}\label{pro3.2}
Let $N\geq 2$, $1<p<N$, $p<q_1<p+\frac{p^2}{N}<q_2=p^*$, $a>0$ and $0<\mu<\mu_{a}^*$. Then the variational problem (\ref{e3.11}) is achieved by some
$u_{a,\mu,-}$, which is real valued, positive, radially symmetric and radially decreasing. Moreover, $u_{a,\mu,-}$ also satisfies the equation (\ref{e1.1}) in the weak sense
for a suitable Lagrange multiplier $\lambda_{a,\mu,-}<0$.
\end{proposition}
\begin{proof}
For any $\mu\in\left(0, \mu_a^*\right)$, by $(1)$ of Lemma~\ref{lemN3.9}, $\mu<\mu_{a,q_2}^*$ for $q_2$ sufficiently close to $p^*$.  Thus, by Proposition~\ref{pro3.21}, $(2)$ of Lemma~\ref{lemN3.9} and similar arguments used in  the proof of Proposition~\ref{pro3.21}, we can show that $\{u_{a,\mu,q_2,-}\}$ is uniformly bounded in $W^{1,p}\left(\mathbb{R}^N\right)$ as $q_2\to p^*$, where $u_{a,\mu,q_2,-}$ is a mountain-pass solution constructed in Proposition~\ref{pro3.21}.  It follows that there exists $u_{a,\mu,-}\in W^{1,p}\left(\mathbb{R}^N\right)$, which is real valued, nonnegative, radially symmetric and radially decreasing, such that
\begin{eqnarray}\label{e3.101}
\left\{\aligned
&u_{a,\mu,q_2,-}\rightharpoonup u_{a,\mu,-}\quad \text{weakly in } W^{1,p}\left(\mathbb{R}^N\right),\\
&u_{a,\mu,q_2,-}\to u_{a,\mu,-}\quad \text{a.e. in } \mathbb{R}^N,\\
&u_{a,\mu,q_2,-}\to u_{a,\mu,-}\quad \text{strongly in } L^t\left(\mathbb{R}^N\right) \text{ for all } t\in (p,p^*)
\endaligned\right.
\end{eqnarray}
as $q_2\to p^*$ up to a subsequence.  Since $u_{a,\mu,q_2,-}\in \mathcal{P}_{a,\mu,q_2}^-$, is a mountain-pass solution constructed in Proposition~\ref{pro3.21} with a suitable Lagrange multiplier $\lambda_{a, \mu,q_2,-}<0$, we have
\begin{eqnarray}\label{e3.102}
 \left\{\aligned
 &\|\nabla u_{a,\mu,q_2,-}\|_p^p=\mu\gamma_{q_1}\|u_{a,\mu,q_2,-}\|_{q_1}^{q_1}+\gamma_{q_2}\|u_{a,\mu,q_2,-}\|_{q_2}^{q_2},\\
&\lambda_{a, \mu,q_2,-}a^p=\mu(\gamma_{q_1}-1) \|u_{a,\mu,q_2,-}\|_{q_1}^{q_1}+(\gamma_{q_2}-1)\|u_{a,\mu,q_2,-}\|_{q_2}^{q_2},\\
&p\|\nabla u_{a,\mu,q_2,-}\|_p^p<\mu q_1\gamma_{q_1}^2\|u_{a,\mu,q_2,-}\|_{q_1}^{q_1}+q_2\gamma_{q_2}^2\|u_{a,\mu,q_2,-}\|_{q_2}^{q_2}.
 \endaligned\right.
\end{eqnarray}
It follows from the boundedness of $\{u_{a,\mu,q_2,-}\}$ in $W^{1,p}\left(\mathbb{R}^N\right)$ as $q_2\to p^*$ that $\{\lambda_{a, \mu,q_2,-}\}$ is also bounded as $q_2\to p^*$.  Thus, we may assume that $\lambda_{a, \mu,q_2,-}\to \lambda_{a, \mu,-}\leq0$ as $q_2\to p^*$ up to a subsequence.  Since $u_{a,\mu,q_2,-}$, is a mountain-pass solution constructed in Proposition~\ref{pro3.21} with the Lagrange multiplier $\lambda_{a, \mu,q_2,-}<0$, it is standard to use the dominated convergence theorem and \eqref{e3.101} to show that
\begin{eqnarray}\label{eqnN01}
-\Delta_pu_{a,\mu,-}=\lambda_{a, \mu,-} u_{a,\mu,-}^{p-1}+\mu u_{a,\mu,-}^{q_1-1}+u_{a,\mu,-}^{p^*-1},\ \text{in}\  \mathbb{R}^N
\end{eqnarray}
in the weak sense.  Moreover, by $\gamma_{q_2}\to1$ as $q_2\to p^*$, we also have
\begin{eqnarray*}
\lambda_{a, \mu,-}a^p=\mu(\gamma_{q_1}-1) \|u_{a,\mu,-}\|_{q_1}^{q_1},
\end{eqnarray*}
which implies that $\lambda_{a, \mu,-}=0$ if and only if $u_{a,\mu,-}=0$.  Since $u_{a,\mu,-}\in W^{1,p}\left(\mathbb{R}^N\right)$ satisfies \eqref{eqnN01}, by Lemma~\ref{lem1.5}, $u_{a,\mu,-}$ satisfies the Pohozaev identity, that is,
\begin{eqnarray}\label{e3.103}
\|\nabla u_{a,\mu,-}\|_p^p=\mu\gamma_{q_1}\|u_{a,\mu,-}\|_{q_1}^{q_1}+\|u_{a,\mu,-}\|_{p^*}^{p^*}.
\end{eqnarray}
Since $\{u_{a,\mu,q_2,-}\}$ is uniformly bounded in $W^{1,p}\left(\mathbb{R}^N\right)$ as $q_2\to p^*$, $\lambda_{a, \mu,q_2,-}\to \lambda_{a, \mu,-}$ as $q_2\to p^*$ up to a subsequence and $u_{a,\mu,-}$ satisfies \eqref{eqnN01}, we can go through the proof of Lemma~\ref{lem1.3} to show that $\nabla u_{a,\mu,q_2,-}\to \nabla u_{a,\mu,-}$ a.e. in $\mathbb{R}^N$ as $q_2\to p^*$ up to a subsequence, since the involved cut-off function $\{\tau_k(u_{a,\mu,q_2,-}-u_{a,\mu,-})\}$ is uniformly bounded for $k\leq10$. Let $v_{a,\mu,q_2,-}=u_{a,\mu,q_2,-}-u_{a,\mu,-}$.  Then by \eqref{e3.102}, \eqref{e3.103} and Lemma \ref{lem1.1},
\begin{eqnarray*}
\|\nabla v_{a,\mu,q_2,-}\|_p^p=\|v_{a,\mu,q_2,-}\|_{p^*}^{p^*}+o(1)
\end{eqnarray*}
as $q_2\to p^*$ up to a subsequence, which implies that either $v_{a,\mu,q_2,-}\to0$ strongly in $D^{1,p}\left(\mathbb{R}^N\right)$ or $\|\nabla v_{a,\mu,q_2,-}\|_p^p=\|v_{a,\mu,q_2,-}\|_{p^*}^{p^*}+o(1)\geq S^{\frac{N}{p}}+o(1)$ as $q_2\to p^*$ up to a subsequence, where $D^{1,p}\left(\mathbb{R}^N\right)=\dot{W}^{1,p}\left(\mathbb{R}^N\right)$ is the usual homogeneous Sobolev space.  If $u_{a,\mu,-}=0$, then we must have that either $u_{a,\mu,q_2,-}\to0$ strongly in $D^{1,p}\left(\mathbb{R}^N\right)$ or $\|\nabla u_{a,\mu,q_2,-}\|_p^p=\|u_{a,\mu,q_2,-}\|_{p^*}^{p^*}+o(1)\geq S^{\frac{N}{p}}+o(1)$ as $q_2\to p^*$ up to a subsequence.  However, since by \eqref{e3.102} and similar arguments used for \eqref{e3.12}, we can show that $\|\nabla u_{a,\mu,q_2,-}\|_p^p\gtrsim1$ as $q_2\to p^*$ up to a subsequence.  Thus, we must have $\|\nabla u_{a,\mu,q_2,-}\|_p^p=\|u_{a,\mu,q_2,-}\|_{p^*}^{p^*}+o(1)\geq S^{\frac{N}{p}}+o(1)$ as $q_2\to p^*$ up to a subsequence, provided $u_{a,\mu,-}=0$.  It follows from $u_{a,\mu,q_2,-}\in \mathcal{P}_{a,\mu,q_2}^-$ that
\begin{eqnarray*}
m^-_{q_2}(a,\mu)&=&\left(\frac{1}{p}-\frac{1}{q_2\gamma_{q_2}}\right)\|\nabla u_{a,\mu,q_2,-}\|_p^p+\mu\gamma_{q_1}\left(\frac{1}{q_2\gamma_{q_2}}-\frac{1}{q_1\gamma_{q_1}}\right)\|u_{a,\mu,q_2,-}\|_{q_1}^{q_1}\\
&\geq&\frac{1}{N}S^{\frac{N}{p}}+o(1)
\end{eqnarray*}
as $q_2\to p^*$ up to a subsequence, which contradicts Lemmas~\ref{lem3.1} and \ref{lem3.6}.  Thus, we must have $u_{a,\mu,-}\not=0$ and $\lambda_{a, \mu,-}<0$, which implies that $u_{a,\mu,-}\in\mathcal{P}_{a_1,\mu}$ where $a_1=\|u_{a,\mu,-}\|_p\in(0, a]$.
Recall that either $v_{a,\mu,q_2,-}\to0$ strongly in $D^{1,p}\left(\mathbb{R}^N\right)$ or $\|\nabla v_{a,\mu,q_2,-}\|_p^p=\|v_{a,\mu,q_2,-}\|_{p^*}^{p^*}+o(1)\geq S^{\frac{N}{p}}+o(1)$ as $q_2\to p^*$ up to a subsequence.  Thus, by Lemmas~\ref{lem3.1} and \ref{lem1.1},
\begin{eqnarray*}
m^-_{q_2}(a,\mu)&=&\left(\frac{1}{p}-\frac{1}{q_2\gamma_{q_2}}\right)\|\nabla u_{a,\mu,q_2,-}\|_p^p+\mu\gamma_{q_1}\left(\frac{1}{q_2\gamma_{q_2}}-\frac{1}{q_1\gamma_{q_1}}\right)\|u_{a,\mu,q_2,-}\|_{q_1}^{q_1}\\
&\geq&m^+(a_1,\mu)+\frac{1}{N}S^{\frac{N}{p}}+o(1)\\
&\geq&m^+(a,\mu)+\frac{1}{N}S^{\frac{N}{p}}+o(1),
\end{eqnarray*}
which contradicts Lemmas~\ref{lemN3.9} and \ref{lem3.6}.  Thus, we must have $v_{a,\mu,q_2,-}\to0$ strongly in $D^{1,p}\left(\mathbb{R}^N\right)$ as $q_2\to p^*$ up to a subsequence. By  $u_{a,\mu,q_2,-}\subset \mathcal{P}_{a,\mu,q_2}^-$ and $0<\mu<\mu_a^*\leq \mu_{a_1}^*$, we obtain that $u_{a,\mu,-}\in \mathcal{P}^-_{a_1,\mu}$. Moreover, as that in the proof of Proposition~\ref{pro3.1}, we also have $a_1=a$.  Thus, $v_{a,\mu,q_2,-}\to0$ strongly in $W^{1,p}\left(\mathbb{R}^N\right)$ as $q_2\to p^*$ up to a subsequence and $m^-(a,\mu)$ is achieved by $u_{a,\mu,-}$.    Moreover, since $u_{a,\mu,-}$ is real valued, nonnegative, radially symmetric and radially decreasing, by the strong comparison principle (\cite[Theorem~1.4]{D98}), we have that $u_{a,\mu,-}$ is also positive in $\mathbb{R}^N$. Thus, $u_{a,\mu,-}$ is the required function and $\lambda_{a,\mu,-}<0$ is its Lagrange multiplier.
\end{proof}

\section{The existence and multiplicity theory  of (\ref{e1.1}) for $\mu=\mu_{a}^*$}
\setcounter{section}{4} \setcounter{equation}{0}

In this section, we shall construct solutions of \eqref{e1.1} for $\mu=\mu_{a}^*$.  We remark that in this case, the degenerate submanifold $\mathcal{P}^0_{a,\mu_a^*}$ may not be an emptyset by Proposition \ref{pro2.1}.  Thus, to establish the existence of solutions to (\ref{e1.1}) for $\mu=\mu_{a}^*$, we need to get a deeper well understood of the $L^p$-Pohozaev manifold $\mathcal{P}_{a,\mu_a^*}$.  Recall that $\mu_{a}^*=\inf_{u\in \mathcal{S}_a}\mu(u)$ by \eqref{e1.12}, where the functional $\mu(u)$ is given by \eqref{muu}.
\begin{lemma}\label{lem5.4}
Let $N\geq 2$, $1<p<N$, $p<q_1<p+\frac{p^2}{N}<q_2\leq p^*$ and  $a>0$.  Then
\begin{eqnarray}\label{eN5.44}
\sup_{u\in \mathcal{S}_{a}}\mu(u)=+\infty.
\end{eqnarray}
In particular, $\mathcal{P}_{a,\mu}^{0}\not=\emptyset$ for all $\mu>\mu_a^*$.
\end{lemma}
\begin{proof}
We borrow the idea in the proof of \cite[Lemma~7.5]{Deng-Wu} to obtain the conclusion.  Let
\begin{eqnarray*}
u_k(x):=\frac{A_k\varphi_k(|x|)}{(1+|x|^2)^{b}}\in \mathcal{S}_a,
\end{eqnarray*}
where $b=\frac{N-p}{4p}$, $A_k>0$ is a constant dependent on $k$,  and
\begin{equation*}
\varphi_k(s):=\left\{ \begin{array}{ll}
1,& s\in [0,k),\\
-s+k+1,& s\in [k,k+1],\\
0,&s\in (k+1,+\infty).
\end{array}
\right.
\end{equation*}
Thus, $u_k(x)\in W_{rad}^{1,p}\left(\mathbb{R}^N\right)$ for any $k\geq 1$.  By direct calculations, for any $q\in [p,p^*]$, we obtain that
\begin{equation*}
\|u_k\|_q^q\thickapprox A_k^q\int_0^k\frac{r^{N-1}}{(1+r^2)^{qb}}dr\thickapprox A_k^q \frac{k^{N-2qb}}{N-2qb}
\end{equation*}
and
\begin{eqnarray*}
\|\nabla u_k\|_p^p&=&A_k^p\int_{\mathbb{R}^N}\frac{1}{(1+|x|^2)^{pb}}\left|\varphi_k'(|x|)\frac{x}{|x|}-2b\varphi_k(|x|)\frac{x}{1+|x|^2}\right|^pdx\\
&=&A_k^p\int_{B_k(0)}\frac{1}{(1+|x|^2)^{pb}}\left|2b\varphi_k(|x|)\frac{x}{1+|x|^2}\right|^pdx\\
&&+A_k^p\int_{B_{k+1}(0)\backslash B_k(0)}\frac{1}{(1+|x|^2)^{pb}}\left|\varphi_k'(|x|)\frac{x}{|x|}-2b\varphi_k(|x|)\frac{x}{1+|x|^2}\right|^pdx\\
&\gtrsim& A_k^p\int_{B_k(0)}\frac{|x|^p}{(1+|x|^2)^{p(b+1)}}dx+A_k^p\int_{B_{k+1}(0)\backslash B_k(0)}\frac{1}{(1+|x|^2)^{pb}}dx\\
&\thickapprox& A_k^p \frac{ k^{N-p-2pb}}{N-p-2pb}+A_k^p k^{N-1-2pb}\\
&\thickapprox& A_k^p k^{N-1-2pb}
\end{eqnarray*}
as $k\to +\infty$.
We choose $A_k>0$ such that $\|u_k\|_p=a$, which implies that $A_k\thickapprox k^{-\frac{N-2pb}{p}}$ as $k\to+\infty$.  Hence,
\begin{eqnarray*}
\mu(u_k)&\gtrsim& \frac{(A_k^p k^{N-1-2pb} )^{\frac{q_2\gamma_{q_2}-q_1\gamma_{q_1}}{q_2\gamma_{q_2}-p}}}
{A_k^{q_1}k^{N-2q_1b}(A_k^{q_2}k^{N-2q_2b})^{\frac{p-q_1\gamma_{q_1}}{q_2\gamma_{q_2}-p}}}\\
&=&A_k^{\frac{p(q_1-q_2)}{q_2\gamma_{q_2}-p}}k^{\frac{(N-2b p)(q_1-q_2)}{q_2\gamma_{q_2}-p}}k^{(p-1)\frac{q_2\gamma_{q_2}-q_1\gamma_{q_1}}{q_2\gamma_{q_2}-p}}\\
&=&k^{(p-1)\frac{q_2\gamma_{q_2}-q_1\gamma_{q_1}}{q_2\gamma_{q_2}-p}}\\
&\to&+\infty
\end{eqnarray*}
as $k\to +\infty$.  Thus, \eqref{eN5.44} holds true.  Since the functional $\mu(u)$, given by \eqref{muu}, is continuous on $\mathcal{S}_a$, we must have $\mathcal{P}_{a,\mu}^{0}\not=\emptyset$ for all $\mu>\mu_a^*$ by \eqref{eN5.44}, which completes the proof.
\end{proof}

We also need the following estimate in understanding the degenerate submanifold $\mathcal{P}^0_{a,\mu_a^*}$ in the Sobolev critical case $q_2=p^*$.
\begin{lemma}\label{lemN5.4}
Let $N\geq 2$, $1<p<N$ and $p<q_1<p+\frac{p^2}{N}<q_2=p^*$.  Then
\begin{eqnarray*}
\lim_{q_1\to p+\frac{p^2}{N}}m^+(a,\mu)=0
\end{eqnarray*}
uniformly for all $a>0$ and $0<\mu<\mu_a^*$, where $m^+(a,\mu)$ is the ground-state energy of $\Psi_\mu(u)$ given by \eqref{e1.11}.
\end{lemma}

\begin{proof}
By Proposition \ref{pro3.1}, there exists $u_{a,\mu,+}\in \mathcal{P}_{a,\mu}^+$ such that $\Psi_{\mu}(u_{a,\mu,+})=m^+(a,\mu)$. Since $u_{a, \mu,+}\in\mathcal{P}^+_{a,\mu}$, by the Gagliardo-Nirenberg  inequality, we have
\begin{eqnarray}\label{eq1001}
\|\nabla u_{a,\mu,+}\|_p^{p-q_1\gamma_{q_1}}<C_{N,p,q_1}^{q_1}\gamma_{q_1}\mu a^{q_1-q_1\gamma_{q_1}}\frac{p^*-q_1\gamma_{q_1}}{p^*-p}
\end{eqnarray}
and
\begin{eqnarray}\label{eq1002}
m^+(a,\mu)&=&\Psi_\mu(u_{a,\mu,+})\notag\\
&\geq&\frac{1}{N}\|\nabla u_{a,\mu,+}\|_p^{p}-\frac{\mu a^{q_1-q_1\gamma_{q_1}}C_{N,p,q_1}^{q_1}}{q_1}\left(1-\frac{q_1\gamma_{q_1}}{p^*}\right)\|\nabla u_{a,\mu,+}\|_p^{q_1\gamma_{q_1}}.
\end{eqnarray}
Let us consider the function
\begin{eqnarray*}
f(t)=\frac{1}{N}t^{p}-\frac{\mu a^{q_1-q_1\gamma_{q_1}}C_{N,p,q_1}^{q_1}}{q_1}\left(1-\frac{q_1\gamma_{q_1}}{p^*}\right)t^{q_1\gamma_{q_1}}.
\end{eqnarray*}
As in \cite{Wei-Wu2022}, a direct calculation shows that $f(t)$ is decreasing in $(0, t_0)$, where
\begin{eqnarray*}
t_0=\left(C_{N,p,q_1}^{q_1}\gamma_{q_1}\mu a^{q_1-q_1\gamma_{q_1}}\frac{p^*-q_1\gamma_{q_1}}{p^*-p}\right)^{\frac{1}{p-q_1\gamma_{q_1}}}.
\end{eqnarray*}
Thus, by \eqref{eq1001} and \eqref{eq1002},
\begin{eqnarray}\label{eq1004}
m^+(a,\mu)&\geq&-\left(C_{N,p,q_1}^{q_1}\gamma_{q_1}\mu a^{q_1-q_1\gamma_{q_1}}\frac{p^*-q_1\gamma_{q_1}}{p^*-p}\right)^{\frac{p}{p-q_1\gamma_{q_1}}}\frac{(p^*-p)(p-q_1\gamma_{q_1})}{p^*pq_1\gamma_{q_1}}\notag\\
&\geq&-\left(C_{N,p,q_1}^{q_1}\gamma_{q_1}\mu_a^* a^{q_1-q_1\gamma_{q_1}}\frac{p^*-q_1\gamma_{q_1}}{p^*-p}\right)^{\frac{p}{p-q_1\gamma_{q_1}}}\frac{(p^*-p)(p-q_1\gamma_{q_1})}{p^*pq_1\gamma_{q_1}}
\end{eqnarray}
for all $0<\mu<\mu_a^*$.  Since it is well known that extremal functions $w_{q_1,p}$ of the Gagliardo-Nirenberg inequality $\|u\|_{q_1}\leq C_{N,p,q_1}\|\nabla u\|_{p}^{\gamma_{q_1}}\|u\|_{p}^{1-\gamma_{q_1}}$, satisfying $\|\nabla w_{q_1,p}\|_p^p=1$ and $\|w_{q_1,p}\|_p^p=1$, are also the solutions of the following equation
\begin{eqnarray*}
-q_1\gamma_{q_1}\Delta_p u+(q_1-q_1\gamma_{q_1})|u|^{p-2}u=q_1C_{N,p,q_1}^{-q_1}|u|^{q_1-2}u,\quad\text{in }\mathbb{R}^N.
\end{eqnarray*}
By the standard scaling, we know that
\begin{eqnarray*}
w_{q_1,p}=\left(\frac{q_1-q_1\gamma_{q_1}}{q_1C_{N,p,q_1}^{-q_1}}\right)^{\frac{1}{q_1-p}}W_{1,q_1}\left(\left(\frac{q_1-q_1\gamma_{q_1}}{q_1\gamma_{q_1}}\right)^{\frac1p}x\right),
\end{eqnarray*}
where $W_{1,q_1}$ is the unique positive radial solution of the following equation
\begin{eqnarray*}
-\Delta_p u+|u|^{p-2}u=|u|^{q_1-2}u,\quad\text{in }\mathbb{R}^N.
\end{eqnarray*}
Thus, we also have
\begin{eqnarray*}
C_{N,p,q_1}^{q_1}&=&\|w_{q_1,p}\|_{q_1}^{q_1}\\
&=&\left(\frac{q_1-q_1\gamma_{q_1}}{q_1C_{N,p,q_1}^{-q_1}}\right)^{\frac{q_1}{q_1-p}}\left\|  W_{1,q_1}\left(\left(\frac{q_1-q_1\gamma_{q_1}}{q_1\gamma_{q_1}}\right)^{\frac1p}x\right)    \right\|_{q_1}^{q_1}\\
&=&\left(\frac{q_1-q_1\gamma_{q_1}}{q_1C_{N,p,q_1}^{-q_1}}\right)^{\frac{q_1}{q_1-p}}
\left(\frac{q_1-q_1\gamma_{q_1}}{q_1\gamma_{q_1}}\right)^{-\frac{N}p}\|W_{1,q_1}\|_{q_1}^{q_1},
\end{eqnarray*}
which implies that
\begin{eqnarray}\label{eq1003}
\left(C_{N,p,q_1}^{q_1}\right)^{\frac{p}{q_1-p}}=\gamma_{q_1}^{-\frac{N}{p}}
\left(\frac{q_1}{q_1-q_1\gamma_{q_1}}\right)^{\frac{q_1}{q_1-p}-\frac{N}{p}}\left(\int_{\mathbb{R}^N}W_{1,q_1}^{q_1}(x)dx\right)^{-1}.
\end{eqnarray}
Since $W_{1,q_1}$ is the unique positive radial solution and $W^{1,p}_{rad}\left(\mathbb{R}^N\right)$ is compact embedded in $L^q\left(\mathbb{R}^N\right)$ for all $q\in(p, p^*)$, it is standard to show that $W_{1,q_1}$ is continuous in $W^{1,p}_{rad}\left(\mathbb{R}^N\right)$ as a function of $q_1\in(p, p^*)$.  It follows from \eqref{eq1003} that $\left(C_{N,p,q_1}^{q_1}\right)^p$ is also continuous as a function of $q_1\in(p, p^*)$.  On the other hand, by Proposition~\ref{pro2.2}, we have $\mu_a^* a^{q_1-q_1\gamma_{q_1}}=\mu_1^*$, where
\begin{equation*}
\mu_{1}^*=\frac{(p^*-p)(p-q_1\gamma_{q_1})^{\frac{p-q_1\gamma_{q_1}}{p^*-p}}}
{\gamma_{q_1}(p^*-q_1\gamma_{q_1})^{\frac{p^*-q_1\gamma_{q_1}}{p^*-p}}}
\inf_{u\in \mathcal{S}_1}\frac{(\|\nabla u\|_p^p)^{\frac{p^*-q_1\gamma_{q_1}}{p^*-p}}}{\|u\|_{q_1}^{q_1}(\|u\|_{p^*}^{p^*})^{\frac{p-q_1\gamma_{q_1}}{p^*-p}}}.
\end{equation*}
Since $w_{q_1,p}$ are the extremal functions of the Gagliardo-Nirenberg inequality such that $\|w_{q_1,p}\|_p^p=1$ and $\|\nabla w_{q_1,p}\|_p^p=1$,  we can test $\mu_{1}^*$ by $w_{q_1,p}$, which implies that
\begin{eqnarray*}
&&C_{N,p,q_1}^{q_1}\gamma_{q_1}\mu_a^* a^{q_1-q_1\gamma_{q_1}}\frac{p^*-q_1\gamma_{q_1}}{p^*-p}\\
&=&C_{N,p,q_1}^{q_1}\left(\frac{p-q_1\gamma_{q_1}}{p^*-q_1\gamma_{q_1}}\right)^{\frac{p-q_1\gamma_{q_1}}{p^*-p}}\inf_{u\in \mathcal{S}_1}\frac{(\|\nabla u\|_p^p)^{\frac{p^*-q_1\gamma_{q_1}}{p^*-p}}}{\|u\|_{q_1}^{q_1}(\|u\|_{p^*}^{p^*})^{\frac{p-q_1\gamma_{q_1}}{p^*-p}}}\\
&\leq& C_{N,p,q_1}^{q_1}\left(\frac{p-q_1\gamma_{q_1}}{p^*-q_1\gamma_{q_1}}\right)^{\frac{p-q_1\gamma_{q_1}}{p^*-p}}\frac{(\|\nabla w_{q_1,p}\|_p^p)^{\frac{p^*-q_1\gamma_{q_1}}{p^*-p}}}{\|w_{q_1,p}\|_{q_1}^{q_1}(\|w_{q_1,p}\|_{p^*}^{p^*})^{\frac{p-q_1\gamma_{q_1}}{p^*-p}}}\\
&=&C_{N,p,q_1}^{q_1}\frac{\|\nabla w_{q_1,p}\|_p^p}{\|w_{q_1,p}\|_{q_1}^{q_1}} \left(\frac{p-q_1\gamma_{q_1}}{p^*-q_1\gamma_{q_1}}\right)^{\frac{p-q_1\gamma_{q_1}}{p^*-p}}\left(\frac{\|\nabla w_{q_1,p}\|_p^p}{\|w_{q_1,p}\|_{p^*}^{p^*}}\right)^{\frac{p-q_1\gamma_{q_1}}{p^*-p}}\\
&=&\left(\frac{p-q_1\gamma_{q_1}}{p^*-q_1\gamma_{q_1}}\right)^{\frac{p-q_1\gamma_{q_1}}{p^*-p}}\left(\frac{1}{\|w_{q_1,p}\|_{p^*}^{p^*}}\right)^{\frac{p-q_1\gamma_{q_1}}{p^*-p}}
\end{eqnarray*}
for all $q_1\in (p,p+\frac{p^2}{N})$. It follows that
\begin{equation*}
\begin{split}
m^+(a,\mu)\geq-   \left(\frac{p-q_1\gamma_{q_1}}{p^*-q_1\gamma_{q_1}}\right)^{\frac{p}{p^*-p}}\left(\frac{1}{\|w_{q_1,p}\|_{p^*}^{p^*}}\right)^{\frac{p}{p^*-p}}       \frac{(p^*-p)(p-q_1\gamma_{q_1})}{\gamma_{q_1}p^*pq_1},
\end{split}
\end{equation*}
which implies that $\lim_{q_1\to p+\frac{p^2}{N}}m^+(a,\mu)=0$ uniformly for all $a>0$ and $0<\mu<\mu_a^*$.
\end{proof}

\begin{remark}\label{rmkq}
It follows from the proof of Lemma \ref{lemN5.4} that there exists $q_1^*\in \left(p, p+\frac{p^2}{N}\right)$, independent of $a$ and $\mu$, such that $m^+(a,\mu)+\frac{1}{N}S^{\frac{N}{p}}>0$ for all  $q_1\in\left(q_1^*, p+\frac{p^2}{N}\right)$ uniformly for all $a>0$ and  $0<\mu<\mu_a^*$.
\end{remark}

\subsection{The structure of $\mathcal{P}_{a,\mu_{a}^*}$}
We define
\begin{eqnarray*}
m^{0}(b,\mu):=\left\{\aligned
&\inf_{u\in\mathcal{P}_{b,\mu}^0}\Psi_\mu(u),\quad &\mathcal{P}_{b,\mu}^0\not=\emptyset,\\
&+\infty, \quad &\mathcal{P}_{b,\mu}^0=\emptyset.
\endaligned\right.
\end{eqnarray*}
\begin{proposition}\label{pro4.1}
Let $N\geq 2$, $1<p<N$, $p<q_1<p+\frac{p^2}{N}<q_2\leq p^*$, $a>0$ and $\mu=\mu_{a}^*$.  Then
\begin{eqnarray}\label{e4.2}
\mathcal{P}_{a,\mu_{a}^*}^0&=&\{u\in \mathcal{P}_{a,\mu_{a}^*} \mid \mu(u)=\mu_{a}^*\}\notag\\
&=&\mathcal{P}_{a,\mu_{a}^*}^0\cap W^{1,p}_{rad}\left(\mathbb{R^N}\right)\notag\\
&=:&\mathcal{P}_{a,\mu_{a}^*,rad}^0
\end{eqnarray}
and if $u\in \mathcal{P}_{a,\mu_{a}^*}^0$, then $u$ satisfies the following equation
\begin{equation}\label{e4.1}
-p\Delta_pu=\lambda \|u\|_{q_1}^{q_1} p|u|^{p-2}u+\mu_{a}^*q_1\gamma_{q_1} |u|^{q_1-2}u+q_2\gamma_{q_2}|u|^{q_2-2}u,\ \text{in}\  \mathbb{R}^N
\end{equation}
in the weak sense,
where $\lambda\in \mathbb{R}$ is the Lagrange multiplier.  Moreover,  $\mathcal{P}_{a,\mu_a^*}^{0}\not=\emptyset$, and $\mu_a^*$ and $m^0(a, \mu_a^*)=m^0_{rad}(a, \mu_a^*)$ are attained, provided
\begin{enumerate}
\item[$(a)$]\quad $q_2<p^*$;
\item[$(b)$]\quad $q_2=p^*$ and $q_1\in\left(q_1^*, p+\frac{p^2}{N}\right)$ with $q_1^*>p$ independent of $a$ defined in Remark \ref{rmkq},
\end{enumerate}
where $m_{rad}^{0}(a,\mu_a^*)$ is given by \eqref{m0}.
\end{proposition}
\begin{proof}
It follows from the definition of the extremal value $\mu_{a}^*$ given by (\ref{e1.12}) that $(u)_s\in \mathcal{P}_{a,\mu_{a}^*}^0$ for some $s>0$ if and only if $u$ is a minimizer of the variational problem (\ref{e1.12}).  Thus,
\begin{eqnarray}\label{eqn1020}
\mathcal{P}_{a,\mu_{a}^*}^0=\{u\in \mathcal{P}_{a,\mu_{a}^*} \mid \mu(u)=\mu_{a}^*\}.
\end{eqnarray}
Moreover, by computing the derivative of the functional
$\mu(u)$, given by \eqref{muu}, at the minimizer $u$ and notice that $s(u)=1$ for $u\in \mathcal{P}_{a,\mu_{a}^*}^0$ where $s(u)$ is given in $(a)$ of Proposition~\ref{pro2.1}, we derive the equation of $u$ which is given by (\ref{e4.1}).

To prove $\mu_a^*$ is attained and $\mathcal{P}_{a,\mu_a^*}^{0}\not=\emptyset$, we choose $\{u_n\}\subset \mathcal{S}_a$ such that it is a minimizing sequence of $\mu_a^*$.  By applying the Schwarz symmetric rearrangement if necessary, we may assume that $u_n$ are real valued, nonnegative, radially symmetric and radially decreasing.  Since the functional $\mu(u)$, given by \eqref{muu}, is 0-homogeneous for the trajectory $(u)_s\in \mathcal{S}_a$ by $(a)$ of Proposition \ref{pro2.2}, we may also assume that $\left\|\nabla u_n\right\|_{p}^{p}=1$.  Thus, $\{u_n\}$ is bounded in $W_{rad}^{1,p}\left(\mathbb{R}^N\right)$.  It follows that there exists $u_0\in W_{rad}^{1,p}\left(\mathbb{R}^N\right)$ such that  $u_n\rightharpoonup u_0$ weakly in $W_{rad}^{1,p}\left(\mathbb{R}^N\right)$ and strongly in $L^q\left(\mathbb{R}^N\right)$ for all $q\in (p,p^*)$ as $n\to+\infty$ up to a subsequence.  If $u_0=0$, then
\begin{eqnarray*}
\mu_a^*+o_n(1)&=&\mu(u_n)\\
&=&\frac{(q_2\gamma_{q_2}-p)(p-q_1\gamma_{q_1})^{\frac{p-q_1\gamma_{q_1}}{q_2\gamma_{q_2}-p}}}
{\gamma_{q_1}\gamma_{q_2}^{\frac{p-q_1\gamma_{q_1}}{q_2\gamma_{q_2}-p}}(q_2\gamma_{q_2}-q_1\gamma_{q_1})^{\frac{q_2\gamma_{q_2}-q_1\gamma_{q_1}}{q_2\gamma_{q_2}-p}}}
\frac{1}{\|u_n\|_{q_1}^{q_1}(\|u_n\|_{q_2}^{q_2})^{\frac{p-q_1\gamma_{q_1}}{q_2\gamma_{q_2}-p}}}\\
&\to&+\infty
\end{eqnarray*}
as $n\to+\infty$, which is impossible.  Thus, we must have $u_0\not=0$.  Moreover, by the Fatou lemma, we also have $a_1:=\|u_0\|_p\in (0,a]$.  In the Sobolev subcritical case $q_2<p^*$, by $(b)$ of Proposition \ref{pro2.2}, we have
\begin{equation*}\label{e5.1}
\mu_a^*=\lim_{n\to+\infty}\mu(u_n)\geq \mu(u_0)\geq \mu_{a_1}^*\geq \mu_a^*,
\end{equation*}
 which implies that $a_1=a$, $\lim_{n\to +\infty}\|\nabla u_n\|_p=\|\nabla u_0\|_p$ and $\mu_a^*$ is achieved by $u_0\in \mathcal{S}_a$. By Proposition \ref{pro2.1}, $(u_0)_{t^0_{\mu_a^*}(u_0)}\in \mathcal{P}_{a,\mu_a^*}^{0}$, which implies that $\mathcal{P}_{a,\mu_a^*}^{0}\not=\emptyset$.  In the Sobolev critical case $q_2=p^*$, by Lemma~\ref{lem5.4}, we can choose $u_{\mu}\in\mathcal{P}^0_{a,\mu,rad}$ with $\mu\downarrow\mu_a^*$.  Since $u_{\mu}\in\mathcal{P}^0_{a,\mu,rad}$, by similar arguments used in the proof of Propositions~\ref{pro3.1} and \ref{pro3.21}, we know that $\{u_{\mu}\}$ is bounded in $W^{1,p}_{rad}\left(\mathbb{R}^N\right)$ and $\|\nabla u_{\mu}\|_{p}^{p}\thickapprox\|u_{\mu}\|_{q_1}^{q_1}\thickapprox\|u_{\mu}\|_{p^*}^{p^*}\thickapprox1$ as $\mu\downarrow\mu_a^*$.  Without loss of generality, we may assume that  $u_\mu\to u_0\not=0$ weakly in $W^{1,p}_{rad}\left(\mathbb{R}^N\right)$ and strongly in $L^{q_1}\left(\mathbb{R}^N\right)$ as $\mu\downarrow\mu_a^*$ up to a subsequence.  Since $\mu\downarrow\mu_a^*$, by $u_{\mu}\in\mathcal{P}_{a,\mu,rad}^0$, we can see that $\{u_\mu\}$
is a minimizing sequence of $\mu_a^*$.  It follows from the method of the Lagrange multiplier and the fact that $s(u_\mu)=1$ for $u_\mu\in \mathcal{P}_{a,\mu,rad}^0$, where $s(u)$ is given in $(a)$ of Proposition~\ref{pro2.1}, that there exists $\lambda_{\mu}\in\mathbb{R}$ such that
\begin{equation}\label{RepWuNew1002}
-p\Delta_pu_\mu=\lambda_{\mu} \|u_{\mu}\|_{q_1}^{q_1} p|u_{\mu}|^{p-2}u_{\mu}+\mu_{a}^*q_1\gamma_{q_1} |u_{\mu}|^{q_1-2}u_{\mu}+p^*|u_{\mu}|^{p^*-2}u_{\mu}+o_\mu(1)
\end{equation}
in the dual space of $W^{1,p}_{rad}\left(\mathbb{R}^N\right)$ as $\mu\downarrow\mu_a^*$.
Let us consider
the map $\mathcal{T}: u\to w=(u)_{t_*}=(t_*)^{\frac{N}{p}}u(t_*x)$ with $t_*=\left(\frac{p^*}{p}\right)^{\frac{1}{p^*-p}}$. It follows from (\ref{RepWuNew1002}) that $w_{\mu}:=\mathcal{T}u_{\mu}$ satisfies
\begin{eqnarray}\label{e7.2}
-\Delta_pw_{\mu}&=&\lambda_{\mu} \|w_{\mu}\|_{q_1}^{q_1}(t_*)^{p-q_1\gamma_{q_1}} |w_{\mu}|^{p-2}w_{\mu}+\mu_{a}^*\frac{q_1\gamma_{q_1}}{p} (t_*)^{p-q_1\gamma_{q_1}}|w_{\mu}|^{q_1-2}w_{\mu}\notag\\
&&+\frac{p^*}{p}(t_*)^{p-p^*}|w_{\mu}|^{p^*-2}w_{\mu}+o_{\mu}(1).
\end{eqnarray}
Note that $\{w_{\mu}\}$ is bounded in $W^{1,p}_{rad}\left(\mathbb{R}^N\right)$ and $\|w_{\mu}\|_{q_1}^{q_1}\thickapprox 1$.
 So there exists $w_0\in W_{rad}^{1,p}(\mathbb{R^N})\backslash\{0\}$ and $\lambda_0\in \mathbb{R}$ such that
$w_{\mu}\rightharpoonup w_0$ weakly in $W^{1,p}_{rad}\left(\mathbb{R}^N\right)$,  $w_{\mu}\to w_0$ strongly in $L^{q}\left(\mathbb{R}^N\right)$ with $q\in (p,p^*)$, and $\lambda_{\mu}\to \lambda_0$  as $\mu\downarrow\mu_a^*$ up to a subsequence.  Letting $\mu\downarrow\mu_a^*$ in \eqref{e7.2}, we obtain that
\begin{equation}\label{RepWuNew1004}
-\Delta_pw_0=\lambda_{0}' \|w_{0}\|_{q_1}^{q_1} p|w_{0}|^{p-2}w_{0}+\overline{\mu}_{a}^*\gamma_{q_1} |w_{0}|^{q_1-2}w_{0}+|w_{0}|^{p^*-2}w_{0},
\end{equation}
where
\begin{eqnarray}\label{eq1005}
\overline{\mu}_a^*=\frac{q_1\gamma_{q_1}}{p}\left(\frac{p^*}{p}\right)^{\frac{p-q_1\gamma_{q_1}}{p^*-p}}\mu_a^*<\mu_a^*
\end{eqnarray}
since $q_1\gamma_{q_1}\in(0, p)$ for $p<q_1<p+\frac{p^2}{N}$.
Moreover, since $u_\mu\in \mathcal{P}_{a,\mu,rad}^0$, we also have $w_\mu\in\mathcal{Q}_{a,\overline{\mu},rad}^-$, where
\begin{eqnarray*}
\mathcal{Q}_{a,\overline{\mu},rad}^-:=\mathcal{T}\left(\mathcal{P}_{a,\mu,rad}^0\right)=\left\{w\in \mathcal{P}_{a,\overline{\mu},rad}^-\mid \frac{p-q_1\gamma_{q_1}}{q_1\gamma_{q_1}}\overline{\mu}\gamma_{q_1}\|w\|_{q_1}^{q_1}=\frac{p^*-p}{p^*}\|w\|_{p^*}^{p^*}\right\}
\end{eqnarray*}
with $\Psi_{\overline{\mu}}(w_\mu)=0$ and
\begin{eqnarray*}
\overline{\mu}=\frac{q_1\gamma_{q_1}}{p}\left(\frac{p^*}{p}\right)^{\frac{p-q_1\gamma_{q_1}}{p^*-p}}\mu<\mu.
\end{eqnarray*}
It follows from Lemma~\ref{lem1.5}, \eqref{RepWuNew1004} and $w_\mu\in\mathcal{Q}_{a,\overline{\mu}_a^*,rad}^-$ that
\begin{eqnarray*}
\left\{\aligned
&\|\nabla w_0\|_p^p=\overline{\mu}_a^*\gamma_{q_1} \|w_0\|_{q_1}^{q_1}+\|w_0\|_{p^*}^{p^*},\\
&\|\nabla v_\mu\|_p^p=\|v_\mu\|_{p^*}^{p^*}+o_\mu(1),
\endaligned\right.
\end{eqnarray*}
where $v_\mu=w_\mu-w_0$.   Then either $\|\nabla v_{\mu}\|_p^p=\|v_{\mu}\|_{p^*}^{p^*}=o_{\mu}(1)$ or $\|\nabla v_{\mu}\|_p^p=\|v_{\mu}\|_{p^*}^{p^*}+o_{\mu}(1)\geq S^{N/p}+o_{\mu}(1)$.
Set $a_0:=\|w_0\|_p\in (0,a]$. If $\|\nabla v_{\mu}\|_p^p=\|v_{\mu}\|_{p^*}^{p^*}\geq S^{N/p}+o_{\mu}(1)$, then by Lemma \ref{lem3.1}
\begin{eqnarray*}
0&=&\Psi_{\bar{\mu}_a^*}(w_{\mu})\\
&=&\Psi_{\bar{\mu}_a^*}(w_0)+\Psi_{\bar{\mu}_a^*}(v_n)+o_n(1)\\
&\geq& m_{rad}^+(a_0,\bar{\mu}_a^*)+\frac{1}{N}S^{N/p}+o_n(1)\\
&\geq& m_{rad}^+(a,\bar{\mu}_a^*)+\frac{1}{N}S^{N/p}+o_n(1),
\end{eqnarray*}
which contradicts Remark \ref{rmkq}.
So we must have $\|\nabla v_{\mu}\|_p^p=\|v_{\mu}\|_{p^*}^{p^*}=o_{\mu}(1)$. Then  $w_{\mu}\to w_0$  strongly in $D^{1,p}(\mathbb{R}^N)$ and $L^{p^*}(\mathbb{R}^N)$ as $\mu\downarrow\mu_a^*$, and
\begin{equation*}
\mu_a^*=\lim_{\mu\downarrow\mu_a^*}\mu(u_{\mu})=\lim_{\mu\downarrow\mu_a^*}\mu(w_{\mu})\geq \mu(w_0)\geq \mu_{a_0}^*\geq \mu_a^*,
\end{equation*}
which implies $a_0=a$, $\mu(w_0)=\mu_a^*$ and $w_{\mu}\to w_0$ strongly in $W^{1,p}(\mathbb{R}^N)$ as $\mu\downarrow\mu_a^*$. It follows that $\mathcal{P}_{a,\mu_a^*}^{0}\neq \emptyset$ for $q_1\in\left(q_1^*, p+\frac{p^2}{N}\right)$.

It remains to prove that $m^{0}(a,\mu_a^*)$ is attained.  Since
\begin{equation*}
\mu_a^*=\mu(u)\geq \mu(u^*)\geq \mu_a^*,
\end{equation*}
for all minimizers of $\mu_a^*$, where $u^*$ is the Schwarz symmetric rearrangement of $u$, all minimizers of $\mu_a^*$ must be real valued, nonnegative, radially symmetric and radially decreasing, which, together with \eqref{e4.1} and the strong comparison principle (\cite[Theorem~1.4]{D98}), implies that all minimizers of $\mu_a^*$ must be real valued, positive, radially symmetric and radially decreasing in $\mathbb{R}^N$.  It follows that $\mathcal{P}_{a,\mu_a^*}^{0}=\mathcal{P}_{a,\mu_a^*,rad}^{0}$ and $m^0(a, \mu_a^*)=m^0_{rad}(a, \mu_a^*)$, which, together with \eqref{eqn1020}, implies that \eqref{e4.2} holds true.  Let $\{u_n\}\subset \mathcal{P}_{a,\mu_a^*,rad}^{0}$ be a minimizing sequence of $m^{0}(a,\mu_a^*)$, provided $\mathcal{P}_{a,\mu_a^*}^{0}\not=\emptyset$.  Then $\{u_n\}$ is a minimizing sequence of $\mu_a^*$.  Thus, by applying the same arguments used to prove that $\mu_a^*$ is attained, we can also show that $m^0_{rad}(a, \mu_a^*)$ is attained, provided $\mathcal{P}_{a,\mu_a^*}^{0}\not=\emptyset$.
\end{proof}

As in \cite{Albuquerque-Silva2020}, we introduce the set $\hat{\mathcal{P}}_{a,\mu}:=\{u\in \mathcal{S}_a\mid\mu<\mu(u)\}$. Then by Proposition \ref{pro2.1} and $(a)$ of Proposition \ref{pro2.2}, it is easy to see that
\begin{equation*}
\hat{\mathcal{P}}_{a,\mu}=\left\{(u)_s \mid s>0, u\in \mathcal{P}_{a,\mu}^+\cup \mathcal{P}_{a,\mu}^-\right\}.
\end{equation*}
Let $\overline{\hat{\mathcal{P}}}_{a,\mu_a^*}$ be the closure of $\hat{\mathcal{P}}_{a,\mu_a^*}$ in the $W^{1,p}\left(\mathbb{R}^N\right)$ topology.
\begin{lemma}\label{lem4.2}
Let $N\geq 2$, $1<p<N$, $p<q_1<p+\frac{p^2}{N}<q_2\leq p^*$ and $a>0$.  Then
\begin{eqnarray*}
\overline{\hat{\mathcal{P}}}_{a,\mu_a^*}&=&\hat{\mathcal{P}}_{a,\mu_a^*}\cup\left\{(u)_s \mid s>0, u\in \mathcal{P}_{a,\mu_a^*}^0\right\}\\
&=&\{u\in \mathcal{S}_a\mid \mu_a^*\leq\mu(u)\}.
\end{eqnarray*}
\end{lemma}
\begin{proof}
The conclusion follows immediately from Proposition \ref{pro2.1} and $(a)$ of Proposition \ref{pro2.2}.
\end{proof}

By the definitions of $\hat{\mathcal{P}}_{a,\mu}$ and $\overline{\hat{\mathcal{P}}}_{a,\mu_a^*}$, it follows from Proposition \ref{pro2.1} that $\hat{\mathcal{P}}_{a,\mu}=\mathcal{S}_a$ for $\mu\in (0, \mu_a^*)$ and $\overline{\hat{\mathcal{P}}}_{a,\mu_a^*}=\mathcal{S}_a$. Moreover, by Lemma \ref{lem4.2} and Proposition \ref{pro2.1}, we can define two functionals $\tau_{\mu_a^*}^{\pm}: \overline{\hat{\mathcal{P}}}_{a,\mu_a^*}\to \mathbb{R}$ by
\begin{equation}\label{e4.3}
\tau_{\mu_a^*}^{\pm}(u):=\left\{\begin{array}{ll}
t_{\mu_a^*}^{\pm}(u), & u\in \hat{\mathcal{P}}_{a,\mu_a^*},\\
t_{\mu_a^*}^{0}(u), & u\in \overline{\hat{\mathcal{P}}}_{a,\mu_a^*}\backslash\hat{\mathcal{P}}_{a,\mu_a^*},
\end{array}
\right.
\end{equation}
where $t_{\mu_a^*}^{\pm}(u)$ and $t_{\mu_a^*}^{0}(u)$ are given in Proposition \ref{pro2.1}.
\begin{lemma}\label{lem4.3}
Let $N\geq 2$, $1<p<N$, $p<q_1<p+\frac{p^2}{N}<q_2\leq p^*$ and $a>0$. Then for any $v\in \mathcal{S}_a$,
\begin{enumerate}
\item[$(a)$]\quad $t_{\mu}^{\pm}(v)$ are $C^1$ in terms of $\mu\in (0,\mu_a^*)$ with $t_{\mu}^{+}(v)$ increasing and $t_{\mu}^{-}(v)$ decreasing.
\item[$(b)$]\quad $\Psi_{\mu}\left((v)_{t_{\mu}^{\pm}(v)}\right)$ is $C^1$ and decreasing in terms of $\mu\in (0,\mu_a^*)$.
\item[$(c)$]\quad We have
\begin{equation*}
\lim_{\mu\uparrow\mu_a^*}t_{\mu}^{\pm}(v)=\tau_{\mu_a^*}^{\pm}(v)\quad \text{and}\quad
\lim_{\mu\uparrow\mu_a^*}\Psi_{\mu}\left((v)_{t_{\mu}^{\pm}(v)}\right)=\Psi_{\mu_a^*}\left((v)_{\tau_{\mu_a^*}^{\pm}(v)}\right),
\end{equation*}
where $\tau_{\mu_a^*}^{\pm}(v)$ are given by (\ref{e4.3}).
\end{enumerate}
\end{lemma}

\begin{proof}
$(a)$\quad Since $\mathcal{P}_{a,\mu}^{0}=\emptyset$ for all $0<\mu<\mu_a^*$, the proof is similar to that for Lemma~\ref{lem3.1}, so we only sketch it here.  Let
\begin{equation*}
H(\mu,t):=t^p\|\nabla v\|_{p}^{p}-\mu \gamma_{q_1}t^{q_1\gamma_{q_1}}\|v\|_{q_1}^{q_1}-\gamma_{q_2}t^{q_2\gamma_{q_2}}\|v\|_{q_2}^{q_2}.
\end{equation*}
By $(v)_{t_{\mu}^{\pm}(v)}\in \mathcal{P}_{a,\mu}^{\pm}$, it is easy to see that
\begin{eqnarray*}
H\left(\mu,t_{\mu}^{\pm}(v)\right)=0\quad\text{and}\quad\frac{\partial H(\mu,t)}{\partial t}|_{t=t_{\mu}^{\pm}(v)}\neq0.
\end{eqnarray*}
Thus, by the implicit function theorem, $t_{\mu}^{\pm}(v)$ are $C^1$ in terms of $\mu\in (0,\mu_a^*)$ with
\begin{equation*}
\frac{d\left(t_{\mu}^{\pm}(v)\right)}{d\mu}=\frac{\gamma_{q_1}\left(t_{\mu}^{\pm}(v)\right)^{q_1\gamma_{q_1}+1}\|v\|_{q_1}^{q_1}}
{p\left(t_{\mu}^{\pm}(v)\right)^p\|\nabla v\|_{p}^{p}-\mu q_1 \gamma_{q_1}^2   \left(t_{\mu}^{\pm}(v)\right))^{q_1\gamma_{q_1}}\|v\|_{q_1}^{q_1}-q_2\gamma_{q_2}^2\left(t_{\mu}^{\pm}(v)\right)^{q_2\gamma_{q_2}}\|v\|_{q_2}^{q_2}}.
\end{equation*}
In particular, $\frac{d\left(t_{\mu}^{+}(v)\right)}{d\mu}>0$ and $\frac{d\left(t_{\mu}^{-}(v)\right)}{d\mu}<0$ for all $\mu\in (0,\mu_a^*)$.

\medskip

$(b)$\quad By the conclusion $(a)$, $\Psi_{\mu}\left((v)_{t_{\mu}^{\pm}(v)}\right)$ is $C^1$ in terms of $\mu\in (0,\mu_a^*)$. Moreover, by $(v)_{t_{\mu}^{\pm}(v)}\in \mathcal{P}_{a,\mu}^{\pm}$, we also have
\begin{eqnarray*}
\frac{d \Psi_{\mu}\left((v)_{t_{\mu}^{\pm}(v)}\right)}{d\mu}&=&\frac{\partial\Psi_{\mu}\left((v)_{t_{\mu}^{\pm}(v)}\right)}{\partial\mu}+
\frac{\partial\Psi_{\mu}\left((v)_{t_{\mu}^{\pm}(v)}\right)}{\partial t_{\mu}^{\pm}(v)} \frac{d\left(t_{\mu}^{\pm}(v)\right)}{d\mu}\\
&=&\frac{\partial\Psi_{\mu}\left((v)_{t_{\mu}^{\pm}(v)}\right)}{\partial\mu}\\
&=&-\frac{\left(t_{\mu}^{\pm}(v)\right)^{q_1\gamma_{q_1}}}{q_1}\|v\|_{q_1}^{q_1}\\
&<&0.
\end{eqnarray*}
Thus, $\Psi_{\mu}\left((v)_{t_{\mu}^{\pm}(v)}\right)$ is decreasing in terms of $\mu\in (0,\mu_a^*)$.

\medskip

$(c)$\quad Since $\hat{\mathcal{P}}_{a,\mu}=\mathcal{S}_a$ for $\mu\in (0,\mu_a^*)$, for any $v\in \mathcal{S}_a$, by Proposition \ref{pro2.1}, there exist $t_{\mu}^{\pm}(v)$ such that $t_{\mu}^{+}(v)<s(v)<t_{\mu}^{+}(v)$ with $s(u)$ given in $(a)$ of Proposition~\ref{pro2.1},
\begin{eqnarray*}
\left\{\aligned
&(t_{\mu}^{+}(v))^p\|\nabla v\|_{p}^{p}-\mu \gamma_{q_1}  (t_{\mu}^{+}(v))^{q_1\gamma_{q_1}}\|v\|_{q_1}^{q_1}-\gamma_{q_2}(t_{\mu}^{+}(v))^{q_2\gamma_{q_2}}\|v\|_{q_2}^{q_2}=0,\\
&p(t_{\mu}^{+}(v))^p\|\nabla v\|_{p}^{p}-\mu q_1 \gamma_{q_1}^2   (t_{\mu}^{+}(v))^{q_1\gamma_{q_1}}\|v\|_{q_1}^{q_1}-q_2\gamma_{q_2}^2(t_{\mu}^{+}(v))^{q_2\gamma_{q_2}}\|v\|_{q_2}^{q_2}>0,
\endaligned\right.
\end{eqnarray*}
and
\begin{eqnarray*}
\left\{\aligned
&(t_{\mu}^{-}(v))^p\|\nabla v\|_{p}^{p}-\mu \gamma_{q_1}  (t_{\mu}^{-}(v))^{q_1\gamma_{q_1}}\|v\|_{q_1}^{q_1}-\gamma_{q_2}(t_{\mu}^{-}(v))^{q_2\gamma_{q_2}}\|v\|_{q_2}^{q_2}=0,\\
&p(t_{\mu}^{-}(v))^p\|\nabla v\|_{p}^{p}-\mu q_1 \gamma_{q_1}^2   (t_{\mu}^{-}(v))^{q_1\gamma_{q_1}}\|v\|_{q_1}^{q_1}-q_2\gamma_{q_2}^2(t_{\mu}^{-}(v))^{q_2\gamma_{q_2}}\|v\|_{q_2}^{q_2}<0.
\endaligned\right.
\end{eqnarray*}
By the conclusion $(a)$, we may assume that $t_{\mu}^{\pm}(v)\to \bar{t}^{\pm}(v)$ as $\mu\uparrow \mu_a^*$. Note that we also have $\overline{\hat{\mathcal{P}}}_{a,\mu_a^*}=\mathcal{S}_a$. Thus, by Lemma \ref{lem4.2}, either $v\in \hat{\mathcal{P}}_{a,\mu_a^*}$ or $v\in \left\{(u)_s \mid s>0, u\in \mathcal{P}_{a,\mu_a^*}^0\right\}$. If $v\in \hat{\mathcal{P}}_{a,\mu_a^*}$, then $\mu_a^*<\mu(v)$. It follows from Proposition \ref{pro2.1} that $\bar{t}^{\pm}(v)=t_{\mu_a^*}^{\pm}(v)$. If $v\in \left\{(u)_s \mid s>0, u\in \mathcal{P}_{a,\mu_a^*}^0\right\}$, then $\mu_a^*=\mu(v)$,
\begin{eqnarray*}
\left\{\aligned
&(\bar{t}^{+}(v))^p\|\nabla v\|_{p}^{p}-\mu_a^* \gamma_{q_1}  (\bar{t}^{+}(v))^{q_1\gamma_{q_1}}\|v\|_{q_1}^{q_1}-\gamma_{q_2}(\bar{t}^{+}(v))^{q_2\gamma_{q_2}}\|v\|_{q_2}^{q_2}=0,\\
&p(\bar{t}^{+}(v))^p\|\nabla v\|_{p}^{p}-\mu_a^* q_1 \gamma_{q_1}^2   (\bar{t}^{+}(v))^{q_1\gamma_{q_1}}\|v\|_{q_1}^{q_1}-q_2\gamma_{q_2}^2(\bar{t}^{+}(v))^{q_2\gamma_{q_2}}\|v\|_{q_2}^{q_2}\geq0,
\endaligned\right.
\end{eqnarray*}
and
\begin{eqnarray*}
\left\{\aligned
&(\bar{t}^{-}(v))^p\|\nabla v\|_{p}^{p}-\mu_a^* \gamma_{q_1}  (\bar{t}^{-}(v))^{q_1\gamma_{q_1}}\|v\|_{q_1}^{q_1}-\gamma_{q_2}(\bar{t}^{-}(v))^{q_2\gamma_{q_2}}\|v\|_{q_2}^{q_2}=0,\\
&p(\bar{t}^{-}(v))^p\|\nabla v\|_{p}^{p}-\mu_a^* q_1 \gamma_{q_1}^2   (\bar{t}^{-}(v))^{q_1\gamma_{q_1}}\|v\|_{q_1}^{q_1}-q_2\gamma_{q_2}^2(\bar{t}^{-}(v))^{q_2\gamma_{q_2}}\|v\|_{q_2}^{q_2}\leq 0.
\endaligned\right.
\end{eqnarray*}
It follows from $v=(u)_s$ for some $s>0$ and $u\in \mathcal{P}_{a,\mu_a^*}^{0}$ that $\bar{t}^{+}(v)=\bar{t}^{-}(v)=t_{\mu_a^*}^{0}(v)$.
The conclusions of $t_{\mu}^{\pm}(v)$ then follows immediately from the definitions of $\tau_{\mu_a^*}^{\pm}(v)$ given by (\ref{e4.3}).  The conclusions of $\Psi_{\mu}\left((v)_{t_{\mu}^{\pm}(v)}\right)$ follows from the continuity of $\Psi_{\mu}\left((v)_{t_{\mu}^{\pm}(v)}\right)$ in terms of $\mu$.
\end{proof}

By Lemma~\ref{lem4.2}, the variational problems~\eqref{newvp} for $\mu=\mu_a^*$ are reduced to
\begin{equation*}
\hat{\Psi}_{\mu_a^*}^{\pm}:=\inf\left\{\Psi_{\mu_a^*}(u)\mid u\in \mathcal{P}_{a,\mu_a^*}^{\pm}\cup \partial\mathcal{P}_{a,\mu_a^*}^{\pm}\right\}.
\end{equation*}
\begin{proposition}\label{pro4.4}
Let $N\geq 2$, $1<p<N$, $p<q_1<p+\frac{p^2}{N}<q_2\leq p^*$ and $a>0$. Then
\begin{enumerate}
\item[$(a)$]\quad $m^{\pm}(a,\mu)$ are decreasing in terms of $\mu\in (0,\mu_a^*)$ with
\begin{equation}\label{e4.8}
\lim_{\mu\uparrow \mu_a^*}m^{\pm}(a,\mu)=\hat{\Psi}_{\mu_a^*}^{\pm}.
\end{equation}
In particular, $\hat{\Psi}_{\mu_a^*}^{-}\geq \hat{\Psi}_{\mu_a^*}^{+}$ and $\hat{\Psi}_{\mu_a^*}^{+}<0$.
\item[$(b)$]\quad For all $u_b\in \mathcal{P}_{b,\mu_a^*}^{\pm}$ with $b<a$, we have
\begin{equation}\label{e4.9}
\Psi_{\mu_a^*}(u_b)>\Psi_{\mu_a^*}\left(\left(\frac{a}{b}u_b\right)_{\tau_{\mu_a^*}^{\pm}\left(\frac{a}{b}u_b\right)}\right).
\end{equation}
In particular, $m^{\pm}(b,\mu_a^*)> \hat{\Psi}_{\mu_a^*}^{\pm}$.
\end{enumerate}
\end{proposition}
\begin{proof}
$(a)$\quad By $(b)$ of Lemma \ref{lem4.3}, Propositions \ref{pro3.1}, \ref{pro3.21} and \ref{pro3.2}, it is easy to see that $m^{\pm}(a,\mu)$ are decreasing in terms of $\mu\in (0,\mu_a^*)$. In fact, for any $0<\mu_1<\mu_2<\mu_a^*$, there exists $u\in \mathcal{P}_{a,\mu_1}^{\pm}$ such that $\Psi_{\mu_1}(u)=m^{\pm}(a,\mu_1)$. Thus,
\begin{equation*}
m^{\pm}(a,\mu_1)=\Psi_{\mu_1}(u)=\Psi_{\mu_1}\left((u)_{t_{\mu_1}^{\pm}(u)}\right)>\Psi_{\mu_2}\left((u)_{t_{\mu_2}^{\pm}(u)}\right)\geq m^{\pm}(a,\mu_2).
\end{equation*}
By $(c)$ of Lemma \ref{lem4.3}, $\lim_{\mu\uparrow \mu_a^*}m^{\pm}(a,\mu)\geq \hat{\Psi}_{\mu_a^*}^{\pm}$. On the other hand, for any $\epsilon>0$, we can find $v_{\epsilon}\in \mathcal{P}_{a,\mu_a^*}^{\pm}\cup \mathcal{P}_{a,\mu_a^*}^{0}$ such that $\Psi_{\mu_a^*}(v_\epsilon)<\hat{\Psi}_{\mu_a^*}^{\pm}+\epsilon$. By Proposition \ref{pro2.1} and $(c)$ of Lemma \ref{lem4.3}, there exist $t_{\mu}^{\pm}(v_\epsilon)\to 1$ as $\mu\uparrow \mu_a^*$ such that $(v_\epsilon)_{t_{\mu}^{\pm}(v_\epsilon)}\in \mathcal{P}_{a,\mu}^{\pm}$. It follows that
\begin{equation*}
\lim_{\mu\uparrow \mu_a^*}m^{\pm}(a,\mu)\leq \lim_{\mu\uparrow \mu_a^*}\Psi_{\mu}\left((v_\epsilon)_{t_{\mu}^{\pm}(v_\epsilon)}\right)=\Psi_{\mu_a^*}(v_\epsilon)<\hat{\Psi}_{\mu_a^*}^{\pm}+\epsilon.
\end{equation*}
By the arbitrariness of $\epsilon>0$, we also have $\lim_{\mu\uparrow \mu_a^*}m^{\pm}(a,\mu)\leq \hat{\Psi}_{\mu_a^*}^{\pm}$. Hence,  (\ref{e4.8}) holds true. By (\ref{e4.8}) and $(a)$ of Lemma \ref{lem3.1}, we also have $\hat{\Psi}_{\mu_a^*}^{-}\geq \hat{\Psi}_{\mu_a^*}^{+}$ and $\hat{\Psi}_{\mu_a^*}^{+}<0$.

\medskip

$(b)$ Since $b<a$, $\mathcal{P}_{b,\mu_a^*}^{\pm}\neq\emptyset$ and $\mathcal{P}_{b,\mu_a^*}^{0}=\emptyset$ by Proposition \ref{pro2.1} and $(b)$ of Proposition \ref{pro2.2}.  Let $u_b\in \mathcal{P}_{b,\mu_a^*}^{\pm}$. For any $c\in (b/2,a)$, there exists $t_{\mu_a^*}^{\pm}\left(\frac{c}{b}u_b\right)$ satisfying
\begin{equation*}
t_{\mu_a^*}^{+}\left(\frac{c}{b}u_b\right)<s\left(\frac{c}{b}u_b\right)<t_{\mu_a^*}^{-}\left(\frac{c}{b}u_b\right)
\end{equation*}
 such that $\left(\frac{c}{b}u_b\right)_{t_{\mu_a^*}^{\pm}\left(\frac{c}{b}u_b\right)}\in \mathcal{P}_{c,\mu_a^*}^{\pm}$, where $s\left(\frac{c}{b}u_b\right)$ is given by $(a)$ of Proposition~\ref{pro2.1}.
Similarly to the proof  of  Lemma \ref{lem3.3}, $t_{\mu_a^*}^{+}\left(\frac{c}{b}u_b\right)$ is increasing, $t_{\mu_a^*}^{-}\left(\frac{c}{b}u_b\right)$ is decreasing  and $\Psi_{\mu_a^*}\left(\left(\frac{c}{b}u_b\right)_{t_{\mu_a^*}^{\pm}\left(\frac{c}{b}u_b\right)} \right)$ is decreasing in terms of $c\in (b/2,a)$. Then similarly to the proof of $(c)$ of Lemma \ref{lem4.3}, we can show that $t_{\mu_a^*}^{\pm}\left(\frac{c}{b}u_b\right)\to \tau_{\mu_a^*}^{\pm}\left(\frac{a}{b}u_b\right)$ as $c\uparrow a$.  Let $\{a_n\}\subset (b,a)$ such that $a_n\uparrow a$ as $n\to+\infty$. Then
\begin{equation*}
\begin{split}
\Psi_{\mu_a^*}(u_b)&>\Psi_{\mu_a^*}\left(\left(\frac{a_1}{b}u_b\right)_{t_{\mu_a^*}^{\pm}\left(\frac{a_1}{b}u_b\right)} \right)\\
&>\Psi_{\mu_a^*}\left(\left(\frac{a_n}{b}u_b\right)_{t_{\mu_a^*}^{\pm}\left(\frac{a_n}{b}u_b\right)} \right)\\
&=\Psi_{\mu_a^*}\left(\left(\frac{a}{b}u_b\right)_{\tau_{\mu_a^*}^{\pm}\left(\frac{a}{b}u_b\right)} \right)+o_n(1).
\end{split}
\end{equation*}
Thus, (\ref{e4.9}) holds true. By the arbitrariness of $u_b\in \mathcal{P}_{b,\mu_a^*}^{\pm}$, we also have $m^{\pm}(b,\mu_a^*)\geq \hat{\Psi}_{\mu_a^*}^{\pm}$. Since  $m^{\pm}(b,\mu_a^*)$ is attained for $0<b<a$ by Propositions \ref{pro2.1}, $(b)$ of  \ref{pro2.2}, \ref{pro3.1}, \ref{pro3.21} and \ref{pro3.2}, we have  $m^{\pm}(b,\mu_a^*)> \hat{\Psi}_{\mu_a^*}^{\pm}$.
\end{proof}

\subsection{The existence of ground-state solutions}
In this section, we shall prove the following result.

\begin{proposition}\label{pro4.5}
Let $N\geq 2$, $1<p<N$, $p<q_1<p+\frac{p^2}{N}<q_2\leq p^*$, $a>0$ and $\mu=\mu_{a}^*$.  Then the variational problem
\begin{equation}\label{e4.4}
\hat{\Psi}_{\mu_a^*}^{+}:=\inf\left\{\Psi_{\mu_a^*}(u)\mid u\in \mathcal{P}_{a,\mu_a^*}^{+}\cup \partial\mathcal{P}_{a,\mu_a^*}^{+}\right\}
\end{equation}
is achieved by some
$u_{a,\mu_a^*,+}\in \mathcal{P}_{a,\mu_a^*}^{+}$, which is real valued, positive, radially symmetric and radially decreasing. Moreover, $u_{a,\mu_a^*,+}$ also satisfies the equation (\ref{e1.1}) in the weak sense
for a suitable Lagrange multiplier $\lambda_{a,\mu_a^*,+}<0$.
\end{proposition}
\begin{proof}
Let $\mu_n\uparrow \mu_a^*$ as $n\to +\infty$ and $u_n$ be the solution of the variational problem (\ref{e3.1}) constructed by Proposition \ref{pro3.1} for $\mu=\mu_n$, which are real valued, positive, radially symmetric and radially decreasing. Then by (\ref{e3.5}) and (\ref{e3.56}) with $\mu=\mu_n$, we obtain that  $\|\nabla u_n\|_p^p\lesssim 1$. Hence, $\{u_n\}$ is bounded in $W^{1,p}\left(\mathbb{R}^N\right)$. By Lemma \ref{lem3.1} and $(a)$ of Proposition \ref{pro4.4}, we obtain that there exists $C>0$ such that $m^+(a,\mu_n)<-C$ for any $n$. Hence,
$\|\nabla u_n\|_p^p\gtrsim 1$, which, together with (\ref{e3.5}), implies that $\|u_n\|_{q_1}^{q_1}\thickapprox 1$.  Recall that $u_n\in \mathcal{P}_{a,\mu_n}^+$ satisfies (\ref{e1.1}) for a suitable Lagrange multiplier $\lambda_n<0$, that is,
\begin{eqnarray*}
\int_{\mathbb{R}^N}|\nabla u_n|^{p-2}\nabla u_n\cdot\nabla\varphi dx&=&\lambda_n\int_{\mathbb{R}^N}|u_n|^{p-2}u_n\varphi dx+\mu_n\int_{\mathbb{R}^N}|u_n|^{q_1-2}u_n\varphi dx\\
&&+\int_{\mathbb{R}^N}|u_n|^{q_2-2}u_n\varphi dx
\end{eqnarray*}
for any $\varphi\in W^{1,p}\left(\mathbb{R}^N\right)$.  Hence,
\begin{equation}\label{e4.5}
\lambda_n a^p=\|\nabla u_n\|_p^p-\mu_n \|u_n\|_{q_1}^{q_1}-\|u_n\|_{q_2}^{q_2},
\end{equation}
which implies $\{\lambda_n\}$ is bounded.  Without loss of generality, we may assume that $\lambda_n\to \lambda_0$ as $n\to +\infty$ for some $\lambda_0\leq0$.  Since by $u_n\in \mathcal{P}_{a,\mu_n}^+$ once more and \eqref{e4.5}, we also have
\begin{eqnarray*}
\lambda_na^p=\mu_n(\gamma_{q_1}-1)\|u_{n}\|_{q_1}^{q_1}+(\gamma_{q_2}-1)\|u_{n}\|_{q_2}^{q_2}.
\end{eqnarray*}
By $\gamma_{q_1}<\gamma_{q_2}\leq1$ and $\|u_n\|_{q_1}^{q_1}\thickapprox 1$, we must have $\lambda_0<0$.  Then
\begin{eqnarray*}
\int_{\mathbb{R}^N}|\nabla u_n|^{p-2}\nabla u_n\cdot\nabla\varphi dx&=&\lambda_0\int_{\mathbb{R}^N}|u_n|^{p-2}u_n\varphi dx+\mu_a^*\int_{\mathbb{R}^N}|u_n|^{q_1-2}u_n\varphi dx\\
&&+\int_{\mathbb{R}^N}|u_n|^{q_2-2}u_n\varphi dx +o_n(1)\|\varphi\|_{W^{1,p}},
\end{eqnarray*}
and by Proposition \ref{pro4.4}, $\Psi_{\mu_a^*}(u_n)=m^+(a,\mu_n)+o_n(1)\to \hat{\Psi}_{\mu_a^*}^+$ as $n\to +\infty$. That is, $\{u_n\}$
is a bounded $(PS)_{\hat{\Psi}_{\mu_a^*}^+}$ sequence of $\Psi_{\mu_a^*}(u)|_{\mathcal{S}_a}$. Then by Lemma \ref{lem1.3}, there exists $u_0\in W^{1,p}_{rad}\left(\mathbb{R}^N\right)$ such that
\begin{eqnarray*}\label{e4.6}
\left\{\aligned
&w_n\rightharpoonup u_0\quad \text{weakly in } W^{1,p}\left(\mathbb{R}^N\right),\\
&w_n\to u_0\quad \text{a.e. in } \mathbb{R}^N,\\
&\nabla w_n\to \nabla u_0\quad \text{a.e. in}\ \mathbb{R}^N,\\
&w_n\to u_0\quad \text{strongly in } L^t\left(\mathbb{R}^N\right)\text{ for all } t\in \left(p,p^*\right)
\endaligned\right.
\end{eqnarray*}
as $n\to+\infty$.  In this position, by Proposition~\ref{pro4.4}, we can adapt the same arguments used in the proof of Proposition~\ref{pro3.1} to show that $u_n\to u_0$ strongly in $W^{1,p}\left(\mathbb{R}^N\right)$ as $n\to+\infty$.  It remains to prove that $u_0\in \mathcal{P}^+_{a,\mu_a^*}$.  Indeed, since $u_n\in \mathcal{P}^+_{a,\mu_n}$ and $u_n\to u_0$ strongly in $W^{1,p}\left(\mathbb{R}^N\right)$ as $n\to+\infty$, we must have $u_0\in \mathcal{P}^+_{a,\mu_a^*}\cup \mathcal{P}^0_{a,\mu_a^*}$. Now suppose the contrary that $u_0\in \mathcal{P}^0_{a,\mu_a^*}$, then $u_0$ satisfies both (\ref{e1.1}) and (\ref{e4.1}) in the weak sense, which implies that
\begin{eqnarray*}
\int_{\mathbb{R}^N}\left((q_2\gamma_{q_2}-p)u_0^{q_2-1}+\mu_a^*(q_1\gamma_{q_1}-p)u_0^{q_1-1}+\lambda' u_0^{p-1}\right)\phi dx=0
\end{eqnarray*}
for all $\phi\in W^{1,p}\left(\mathbb{R}^N\right)$.  It follows from $u_0>0$ in $\mathbb{R}^N$ that
$(q_2\gamma_{q_2}-p)|u_0|^{q_2-p}-\mu_a^*(p-q_1\gamma_{q_1})|u_0|^{q_1-p}$ is a constant a.e. in $\mathbb{R}^N$, which contradicts $u_0\in W^{1,p}\left(\mathbb{R}^N\right)$ and $\|u_0\|_p=a>0$.  Thus, $u_0\in \mathcal{P}^+_{a,\mu_a^*}$.  Then $u_{a,\mu_a^*,+}:=u_0$ is the required function and $\lambda_{a,\mu_a^*,+}:=\lambda_0<0$ is its Lagrange multiplier.
\end{proof}

\subsection{The existence of mountain-pass solutions}
In this section, we shall study the variational problem
\begin{equation}\label{e4.7}
\hat{\Psi}_{\mu_a^*}^{-}:=\inf\left\{\Psi_{\mu_a^*}(u)\mid u\in \mathcal{P}_{a,\mu_a^*}^{-}\cup \partial\mathcal{P}_{a,\mu_a^*}^{-}\right\}.
\end{equation}
As for \eqref{e3.11}, a crucial point to prove the achievement of \eqref{e4.7} for the Sobolev critical case $q_2=p^*$ is the following energy estimate.
\begin{lemma}\label{lem4.6}
Let $N\geq 2$, $1<p<N$, $p<q_1<p+\frac{p^2}{N}<q_2=p^*$ and $a>0$.  Then
\begin{equation*}
\hat{\Psi}_{\mu_a^*}^{-}<\hat{\Psi}_{\mu_a^*}^{+}+\frac{1}{N}S^{\frac{N}{p}}.
\end{equation*}
\end{lemma}

\begin{proof}
Since by Proposition \ref{pro4.5}, the variational problem (\ref{e4.4}) is achieved by some $u_{a,\mu_a^*,+}\in \mathcal{P}_{a,\mu_a^*}^+$, which is  real valued, positive, radially symmetric and radially decreasing, we only need to prove $\hat{\Psi}_{\mu_a^*}^{-}<m^{+}(a,\mu_a^*)+\frac{1}{N}S^{\frac{N}{p}}$.  Similarly, define  $W_{\epsilon,\tau}\in \mathcal{S}_a$ as in the proof of Lemma \ref{lem3.6} for $p^2>N$ and as in the proof of \cite[Lemma~4.5]{Deng-Wu} for the case $p^2\leq N$.
Since $W_{\epsilon,\tau}\in \mathcal{S}_a$, by $\overline{\hat{\mathcal{P}}}_{a,\mu_a^*}=\mathcal{S}_a$ and (\ref{e4.3}), there exists
 $\eta^-_{\mu_a^*}(W_{\epsilon,\tau})>0$ such that $(W_{\epsilon,\tau})_{\eta^-_{\mu_a^*}(W_{\epsilon,\tau})}\in \mathcal{P}_{a,\mu_a^*}^{-}\cup \mathcal{P}_{a,\mu_a^*}^{0}$.  Set $\eta^-_{\epsilon,\tau}:=\eta^-_{\mu_a^*}\left(W_{\epsilon,\tau}\right)$, then
\begin{equation}\label{e5.11}
\left(\eta^-_{\epsilon,\tau}\right)^p\|\nabla W_{\epsilon,\tau}\|_p^p=\mu\gamma_{q_1}\left(\eta^-_{\epsilon,\tau}\right)^{q_1\gamma_{q_1}}\|W_{\epsilon,\tau}\|_{q_1}^{q_1}
+\left(\eta^-_{\epsilon,\tau}\right)^{p^*}\|W_{\epsilon,\tau}\|_{p^*}^{p^*}.
\end{equation}
Since $W_{\epsilon,0}=u_{a,\mu_a^*,+}\in \mathcal{P}_{a,\mu_a^*}^{+}$, by Proposition \ref{pro2.1}, $0<\mu_a^*<\mu(u_{a,\mu_a^*,+})$ and $\eta^-_{\epsilon,0}>1$. By Lemma~\ref{lemN3.7} and (\ref{e5.11}), $\eta^-_{\epsilon,\tau}\to 0$ as $\tau\to +\infty$ uniformly for $\epsilon>0$ sufficiently small. Since  $\eta^-_{\epsilon,\tau}$ is unique by Proposition \ref{pro2.1}, it is standard to show that $\eta^-_{\epsilon,\tau}$ is continuous for  $\tau\geq 0$, which implies that there exists $\tau_{\epsilon}>0$ such that  $\eta^-_{\epsilon,\tau_{\epsilon}}=1$. Consequently,
\begin{equation*}
\hat{\Psi}_{\mu_a^*}^{-}\leq  \sup_{\tau\geq 0}\Psi_{\mu_a^*}(W_{\epsilon,\tau})
\end{equation*}
for any $\epsilon$ small enough. The rest of the proof is the same as that of Lemma~\ref{lem3.6} and \cite[Lemma~4.5]{Deng-Wu}, so we omit it here.
\end{proof}

With the above estimate of $\hat{\Psi}_{\mu_a^*}^-$, we can prove the following result.

\begin{proposition}\label{pro4.7}
Let $N\geq 2$, $1<p<N$, $p<q_1<p+\frac{p^2}{N}<q_2\leq p^*$, $a>0$ and $\mu=\mu_{a}^*$. Then the variational problem (\ref{e4.7})
is achieved by some
$u_{a,\mu_a^*,-}\in \mathcal{P}_{a,\mu_a^*}^{-}$, which is real valued, positive, radially symmetric and radially decreasing. Moreover, $u_{a,\mu_a^*,-}$ also satisfies the equation (\ref{e1.1}) in the weak sense
for a suitable Lagrange multiplier $\lambda_{a,\mu_a^*,-}<0$.
\end{proposition}
\begin{proof}
The idea of the proof is to combine that of Propositions~\ref{pro3.21}, \ref{pro3.2} and \ref{pro4.5}.  Let $\mu_n\uparrow \mu_a^*$ as $n\to +\infty$ and $u_n\in \mathcal{P}_{a,\mu_n}^-$ be the solution of the variational problem (\ref{e3.11}) constructed by Propositions  \ref{pro3.21} and  \ref{pro3.2} for $\mu=\mu_n$, which are real valued, positive, radially symmetric and radially decreasing. By (\ref{e3.55}), (\ref{e3.12}) and (\ref{e3.13}) with $\mu=\mu_n$, we obtain that $\{u_n\}$ is bounded in $W^{1,p}\left(\mathbb{R}^N\right)$ with $\|\nabla u_n\|_p^p\gtrsim 1$ and $\|u_n\|_{q_2}^{q_2}\gtrsim 1$.  As in the proof of Proposition~\ref{pro4.5}, $\{u_n\}$
is a bounded $(PS)_{\hat{\Psi}_{\mu_a^*}^-}$ sequence of $\Psi_{\mu_a^*}(u)|_{\mathcal{S}_a}$.  Then by Lemma \ref{lem1.3}, there exists $u_0\in W^{1,p}_{rad}\left(\mathbb{R}^N\right)$ such that
\begin{eqnarray*}\label{e4.66}
\left\{\aligned
&w_n\rightharpoonup u_0\quad \text{weakly in } W^{1,p}\left(\mathbb{R}^N\right),\\
&w_n\to u_0\quad \text{a.e. in } \mathbb{R}^N,\\
&\nabla w_n\to \nabla u_0\quad \text{a.e. in}\ \mathbb{R}^N,\\
&w_n\to u_0\quad \text{strongly in } L^t\left(\mathbb{R}^N\right)\text{ for all } t\in \left(p,p^*\right)
\endaligned\right.
\end{eqnarray*}
as $n\to+\infty$.  Moreover, $u_0$ is real valued, positive, radially symmetric and radially decreasing, which satisfies the following equation
\begin{equation*}
-\Delta_pu=\lambda_0 |u|^{p-2}u+\mu_a^* |u|^{q_1-2}u+|u|^{q_2-2}u,\quad \text{in } \mathbb{R}^N
\end{equation*}
in the weak sense, where $\lambda_0=\lim_{n\to+\infty}\lambda_n\leq0$ and $\lambda_n<0$ are the Lagrange multipliers of $u_n$ to solve (\ref{e1.1}).  In this position, by Proposition~\ref{pro4.4} and Lemma~\ref{lem4.6}, we can adapt the same arguments used in the proofs of Propositions~\ref{pro3.21} and \ref{pro3.2} to show that $u_n\to u_0\not=0$ strongly in $W^{1,p}\left(\mathbb{R}^N\right)$ as $n\to+\infty$ and $\lambda_0<0$.  The proof of $u_0\in \mathcal{P}^-_{a,\mu_a^*}$ is the same as that in the proof of Proposition~\ref{pro4.5}, so we omit it here.  Then $u_{a,\mu_a^*,-}:=u_0$ is the required function and $\lambda_{a,\mu_a^*,-}:=\lambda_0$ is its Lagrange multiplier.
\end{proof}

\section{The existence and multiplicity theory  of (\ref{e1.1}) for $\mu>\mu_{a}^*$}
\setcounter{section}{5} \setcounter{equation}{0}
In this section, we shall construct solutions of \eqref{e1.1} for $\mu>\mu_a^*$.  To achieve this goal, we need to introduce the second extremal values.  Let
\begin{eqnarray}\label{mu1}
\hat{\mu}_{a,\pm}^{**}:=\sup\left\{\mu\geq\mu_a^*\mid m_{rad}^{\pm}(b,\mu)<m_{rad}^{0}(b,\mu)\text{ for all }0<b\leq a\right\},
\end{eqnarray}
where $m_{rad}^{0}(b,\mu)$ and $m_{rad}^{\pm}(b,\mu)$ are given by \eqref{m0} and \eqref{radial}, respectively.
\begin{lemma}\label{lem5.3}
Let $N\geq 2$, $1<p<N$, $p<q_1<p+\frac{p^2}{N}<q_2\leq p^*$ and  $a>0$.  Then
\begin{enumerate}
\item[$(a)$]\quad $m_{rad}^{+}(b,\mu_a^*)<m_{rad}^{0}(b,\mu_a^*)$ for all $0<b\leq a$.
\item[$(b)$]\quad $m_{rad}^{-}(b,\mu_a^*)<m_{rad}^{0}(b,\mu_a^*)$ for all $0<b\leq a$, provided
\begin{enumerate}
\item[$(1)$]\quad $q_2<p^*$;
\item[$(2)$]\quad $q_2=p^*$ and $q_1\in\left(q_1^*, p+\frac{p^2}{N}\right)$ where $q_1^*>p$ is given in Proposition~\ref{pro4.1}.
\end{enumerate}
\end{enumerate}
In particular, $\hat{\mu}_{a,+}^{**}$ is always well defined while, $\hat{\mu}_{a,-}^{**}$ is well defined provided $(1)$ or $(2)$ holds.
\end{lemma}
\begin{proof}
Since $\mu_a^*$ is decreasing in terms of $a>0$ by Proposition \ref{pro2.2}, we have $\mu_b^*>\mu_a^*$ for all $0<b<a$.  It follows that $\mathcal{P}_{b,\mu_a^*}^0=\emptyset$ for all $0<b<a$. By Propositions~\ref{pro3.1}, \ref{pro3.21} and \ref{pro3.2}, $m^{\pm}_{rad}(b,\mu_a^*)=m^{\pm}(b,\mu_a^*)$ for any $0<b< a$.  Thus, we always have $m^{\pm}_{rad}(b,\mu_a^*)<m^{0}_{rad}(b,\mu_a^*)$ for all $0<b< a$ by the definitions.

It remains to prove
\begin{eqnarray}\label{e5.21}
m^{\pm}_{rad}(a,\mu_a^*)<m_{rad}^{0}(a,\mu_a^*),
\end{eqnarray}
which implies that $\hat{\mu}_{a,\pm}^{**}$ are well defined.  By Proposition~\ref{pro4.1},
$m^0(a, \mu_a^*)=m^0_{rad}(a, \mu_a^*)$ is attained by some $u_{a,\mu_a^*,0}$, which is real valued, positive, radially symmetric and radially decreasing, provided $(1)$ or $(2)$ holds.  For any $\varphi \in T_{u_{a,\mu_a^*,0}}\mathcal{S}_a\cap C_0^{\infty}\left(\mathbb{R}^N\right)$, where $T_{u_{a,\mu_a^*,0}}\mathcal{S}_a$ is the tangent space of $\mathcal{S}_a$ at $u_{a,\mu_a^*,0}$, since $u_{a,\mu_a^*,0}>0$ in $\mathbb{R}^N$, by the implicit function theorem, there exists
\begin{equation}\label{eq1010}
s(\epsilon)=1-\frac{p-1}{2a^p}\left(\int_{\mathbb{R}^N}|u_{a,\mu_a^*,0}|^{p-2}\varphi^2dx\right)\epsilon^2+o(\epsilon^2)
\end{equation}
for $\epsilon$ sufficiently small such that $u_{\epsilon}:=s(\epsilon)u_{a,\mu_a^*,0}+\epsilon\varphi\in \mathcal{S}_a$ for all $1<p<N$.  It follows from \eqref{eq1010} and \cite[Theorem~1.1]{DPR1999} (see also \cite[Theorem~3]{SZ1999}) that
\begin{eqnarray}\label{e5.6}
\|\nabla u_\epsilon\|_{p}^{p}&=&\|\nabla u_{a,\mu_a^*,0}\|_{p}^{p}+\epsilon p\int_{\mathbb{R}^N}|\nabla u_{a,\mu_a^*,0}|^{p-2}\nabla u_{a,\mu_a^*,0}\cdot\nabla \varphi dx \notag\\
&&+\epsilon^2\left(\frac{p}{2}\int_{\mathbb{R}^N}|\nabla u_{a,\mu_a^*,0}|^{p-2}|\nabla \varphi|^2 dx-Ap\|\nabla u_{a,\mu_a^*,0}\|_{p}^{p}\right.\notag\\
&&\left.+p\left(\frac{p}{2}-1\right)\int_{\mathbb{R}^N}|\nabla u_{a,\mu_a^*,0}|^{p-4}|\nabla u_{a,\mu_a^*,0} \cdot\nabla \varphi|^2 dx\right)+o(\epsilon^2)
\end{eqnarray}
and
\begin{eqnarray}\label{e5.5}
\|u_\epsilon\|_{q_i}^{q_i}&=&\|u_{a,\mu_a^*,0}\|_{q_i}^{q_i}+q_i\epsilon\int_{\mathbb{R}^N}u_{a,\mu_a^*,0}^{q_i-1}\varphi dx\notag\\
&&+\left(\frac{q_i(q_i-1)}{2}\int_{\mathbb{R}^N}u_{a,\mu_a^*,0}^{q_i-2}\varphi^2 dx-Aq_i\|u_{a,\mu_a^*,0}\|_{q_i}^{q_i}\right)\epsilon^2+o(\epsilon^2)
\end{eqnarray}
for $i=1,2$ and all $1<p<N$, where
\begin{eqnarray*}
A=\frac{p-1}{2a^p}\left(\int_{\mathbb{R}^N}|u_{a,\mu_a^*,0}|^{p-2}\varphi^2dx\right).
\end{eqnarray*}
On the other hand, by Proposition \ref{pro2.1}, there exists $\tau_{\mu_a^*}^{\pm}(u_{\epsilon})$ such that $(u_{\epsilon})_{\tau_{\mu_a^*}^{\pm}(u_{\epsilon})}\in \mathcal{P}_{a,\mu_a^*}^{\pm}\cup \mathcal{P}_{a,\mu_a^*}^0$, where $\tau_{\mu_a^*}^{\pm}(u_{\epsilon})$ is defined by (\ref{e4.3}).  Since $u_{a,\mu_a^*,0}\in \mathcal{P}_{a,\mu_a^*}^0$, by the continuity of $\tau_{\mu_a^*}^{\pm}(u_{\epsilon})$ which is derived by its uniqueness, $\tau_{\mu_a^*}^{\pm}(u_{\epsilon})\to 1$ as $\epsilon\to 0$.  Thus, without loss of generality, we may write $\tau_{\mu_a^*}^{\pm}(u_{\epsilon})=1+\tau^{\pm}(\epsilon)$, where $\tau^{\pm}(\epsilon)\to 0$ as $\epsilon\to 0$.  It follows from $(u_{\epsilon})_{\tau_{\mu_a^*}^{\pm}(u_{\epsilon})}\in \mathcal{P}_{a,\mu_a^*}^{\pm}\cup \mathcal{P}_{a,\mu_a^*}^0$ that
\begin{equation}\label{e5.3}
\left(1+\tau^{\pm}(\epsilon)\right)^{p}\|\nabla u_\epsilon\|_p^p-\mu_{a}^* \gamma_{q_1}\left(1+\tau^{\pm}(\epsilon)\right)^{q_1\gamma_{q_1}}\|u_\epsilon\|_{q_1}^{q_1}
-\gamma_{q_2}\left(1+\tau^{\pm}(\epsilon)\right)^{q_2\gamma_{q_2}}\|u_\epsilon\|_{q_2}^{q_2}=0.
\end{equation}
Thus, by inserting \eqref{e5.6} and \eqref{e5.5} into \eqref{e5.3} and by $u_{a,\mu_a^*,0}\in \mathcal{P}_{a,\mu_a^*}^{0}$, Proposition \ref{pro4.1} and $\varphi \in T_{u_{a,\mu_a^*,0}}\mathcal{S}_a$, we have
\begin{eqnarray*}
0=\epsilon\tau^{\pm}(\epsilon)B_1+\epsilon^2B+(\tau^{\pm}(\epsilon))^2C+o(\epsilon^2)+o((\tau^{\pm}(\epsilon))^2),
\end{eqnarray*}
where
\begin{equation*}
B_1=p^2\int_{\mathbb{R}^N}|\nabla u_{a,\mu_a^*,0}|^{p-2}\nabla u_{a,\mu_a^*,0}\nabla\varphi dx-\mu_a^*q_1^2\gamma_{q_1}^2\int_{\mathbb{R}^N}u_{a,\mu_a^*,0}^{q_1-1}\varphi dx-q_2^2\gamma_{q_2}^2\int_{\mathbb{R}^N}u_{a,\mu_a^*,0}^{q_2-1}\varphi dx,
\end{equation*}
\begin{eqnarray*}
B&=&\frac{p}{2}\int_{\mathbb{R}^N}|\nabla u_{a,\mu_a^*,0}|^{p-2}|\nabla \varphi|^2 dx-Ap\|\nabla u_{a,\mu_a^*,0}\|_{p}^{p}\\
&&+p\left(\frac{p}{2}-1\right)\int_{\mathbb{R}^N}|\nabla u_{a,\mu_a^*,0}|^{p-4}|\nabla u_{a,\mu_a^*,0}\nabla\varphi|^2 dx\\
&&-\mu_{a}^*\gamma_{q_1}\left(\frac{q_1(q_1-1)}{2}\int_{\mathbb{R}^N}u_{a,\mu_a^*,0}^{q_1-2}\varphi^2 dx-Aq_1\|u_{a,\mu_a^*,0}\|_{q_1}^{q_1}\right)\\
&&-\gamma_{q_2}\left(\frac{q_2(q_2-1)}{2}\int_{\mathbb{R}^N}u_{a,\mu_a^*,0}^{q_2-2}\varphi^2 dx-Aq_2\|u_{a,\mu_a^*,0}\|_{q_2}^{q_2}\right)
\end{eqnarray*}
and
\begin{eqnarray*}
C&=&\frac{p(p-1)}{2}\|\nabla u_{a,\mu_a^*,0}\|_{p}^{p}-\mu_a^*\gamma_{q_1}\frac{q_1\gamma_{q_1}(q_1\gamma_{q_1}-1)}{2}\|u_{a,\mu_a^*,0}\|_{q_1}^{q_1}\\
&&-\gamma_{q_2}\frac{q_2\gamma_{q_2}(q_2\gamma_{q_2}-1)}{2}\|u_{a,\mu_a^*,0}\|_{q_2}^{q_2}.
\end{eqnarray*}
Since $u_{a,\mu_a^*,0}\in \mathcal{P}_{a,\mu_a^*}^{0}$ and $q_1\gamma_{q_1}<p<q_2\gamma_{q_2}$, we have
\begin{eqnarray*}
C=\frac{(p-q_1\gamma_{q_1})(p-q_2\gamma_{q_2})}{2}\|\nabla u_{a,\mu_a^*,0}\|_{p}^{p}<0.
\end{eqnarray*}
It follows that $|\tau^{\pm}(\epsilon)|\lesssim \epsilon$ as $\epsilon\to 0$ which, together with $u_{a,\mu_a^*,0}\in \mathcal{P}_{a,\mu_a^*}^{0}$, implies that
\begin{eqnarray*}
&&\Psi_{\mu_a^*}\left((u_\epsilon)_{\tau_{\mu_a^*}^{\pm}(u_{\epsilon})}\right)\\
&=&\frac{1}{p}\left(\tau_{\mu_a^*}^{\pm}(u_{\epsilon})\right)^{p}\|\nabla u_\epsilon\|_p^p-\frac{\mu_a^*}{q_1}\left(\tau_{\mu_a^*}^{\pm}(u_{\epsilon})\right)^{q_1\gamma_{q_1}}\|u_\epsilon\|_{q_1}^{q_1}-\frac{1}{q_2}\left(\tau_{\mu_a^*}^{\pm}(u_{\epsilon})\right)^{q_2\gamma_{q_2}}\|u_\epsilon\|_{q_2}^{q_2}\\
&=&\frac{1}{p}(1+p\tau^{\pm}(\epsilon)+o(\epsilon))\left(\|\nabla u_{a,\mu_a^*,0}\|_{p}^{p}+\epsilon p\int_{\mathbb{R}^N}|\nabla u_{a,\mu_a^*,0}|^{p-2}\nabla u_{a,\mu_a^*,0}\cdot\nabla \varphi dx+o(\epsilon)\right)\\
&&-\frac{\mu_a^*}{q_1}(1+q_1\gamma_{q_1}\tau^{\pm}(\epsilon)+o(\epsilon)) \left(\|u_{a,\mu_a^*,0}\|_{q_1}^{q_1}+q_1\epsilon\int_{\mathbb{R}^N}u_{a,\mu_a^*,0}^{q_1-1}\varphi dx+o(\epsilon)\right)\\
&&-\frac{1}{q_2}(1+q_2\gamma_{q_2}\tau^{\pm}(\epsilon)+o(\epsilon)) \left(\|u_{a,\mu_a^*,0}\|_{q_2}^{q_2}+q_2\epsilon\int_{\mathbb{R}^N}u_{a,\mu_a^*,0}^{q_2-1}\varphi dx+o(\epsilon)\right)\\
&=&\frac{1}{p}\|\nabla u_{a,\mu_a^*,0}\|_{p}^{p}-\frac{\mu_a^*}{q_1}\|u_{a,\mu_a^*,0}\|_{q_1}^{q_1}-\frac{1}{q_2}\|u_{a,\mu_a^*,0}\|_{q_2}^{q_2}\\
&&+\epsilon\left( \int_{\mathbb{R}^N}|\nabla u_{a,\mu_a^*,0}|^{p-2}\nabla u_{a,\mu_a^*,0}\cdot\nabla \varphi dx- \mu_a^*\int_{\mathbb{R}^N}u_{a,\mu_a^*,0}^{q_1-1}\varphi dx-\int_{\mathbb{R}^N}u_{a,\mu_a^*,0}^{q_2-1}\varphi dx \right)\\
&&+o(\epsilon).
\end{eqnarray*}
We claim that there exists $\varphi \in T_{u_{a,\mu_a^*,0}}\mathcal{S}_a\cap C_0^{\infty}\left(\mathbb{R}^N\right)$ such that
\begin{equation}\label{e5.7}
\int_{\mathbb{R}^N}|\nabla u_{a,\mu_a^*,0}|^{p-2}\nabla u_{a,\mu_a^*,0}\cdot\nabla \varphi dx- \mu_a^*\int_{\mathbb{R}^N}u_{a,\mu_a^*,0}^{q_1-1}\varphi dx-\int_{\mathbb{R}^N}u_{a,\mu_a^*,0}^{q_2-1}\varphi dx\neq 0.
\end{equation}
Suppose the contrary that for any $\varphi \in T_{u_{a,\mu_a^*,0}}\mathcal{S}_a\cap C_0^{\infty}\left(\mathbb{R}^N\right)$,
\begin{equation*}%\label{e5.7}
\int_{\mathbb{R}^N}|\nabla u_{a,\mu_a^*,0}|^{p-2}\nabla u_{a,\mu_a^*,0}\cdot\nabla \varphi dx- \mu_a^*\int_{\mathbb{R}^N}u_{a,\mu_a^*,0}^{q_1-1}\varphi dx-\int_{\mathbb{R}^N}u_{a,\mu_a^*,0}^{q_2-1}\varphi dx=0,
\end{equation*}
then $u_{a,\mu_a^*,0}$ satisfies both (\ref{e1.1}) and (\ref{e4.1}) in the weak sense.  Thus, by similar arguments used in the proof of Proposition~\ref{pro4.5}, we can get a contradiction.  Now, by \eqref{e5.7}, we can choose $\varphi \in T_{u_{a,\mu_a^*,0}}\mathcal{S}_a\cap C_0^{\infty}\left(\mathbb{R}^N\right)$ such that
\begin{equation*}
\Psi_{\mu_a^*}\left((u_\epsilon)_{\tau_{\mu_a^*}^{\pm}(u_{\epsilon})}\right)<\Psi_{\mu_a^*}(u_{a,\mu_a^*,0})=m^0(a,\mu_a^*),
\end{equation*}
which, together with $(u_{\epsilon})_{\tau_{\mu_a^*}^{\pm}(u_{\epsilon})}\in \mathcal{P}_{a,\mu_a^*}^{\pm}\cup \mathcal{P}_{a,\mu_a^*}^0$, implies that
$(u_{\epsilon})_{\tau_{\mu_a^*}^{\pm}(u_{\epsilon})}\in \mathcal{P}_{a,\mu_a^*}^{\pm}$.  It follows that (\ref{e5.21}) holds true, provided $(1)$ or $(2)$ holds.

We complete the proof by showing that
\begin{eqnarray*}
m^{+}_{rad}(a,\mu_a^*)<m^{-}_{rad}(a,\mu_a^*)\leq m_{rad}^{0}(a,\mu_a^*)
\end{eqnarray*}
always holds true for $q_2=p^*$.  By Propositions~\ref{pro2.1} and \ref{pro4.7}, it only need to show that $m^{-}_{rad}(a,\mu_a^*)\leq m_{rad}^{0}(a,\mu_a^*)$ always holds true for $q_2=p^*$.  If $\mathcal{P}_{a,\mu_a^*,rad}^{0}=\emptyset$, then the conclusion immediately holds true by the definitions.  Thus, without loss of generality, we assume that $\mathcal{P}_{a,\mu_a^*,rad}^{0}\neq\emptyset$.  Suppose by contradiction that $m^{-}_{rad}(a,\mu_a^*)> m^{0}_{rad}(a,\mu_a^*)$. For any $\epsilon\in (0,m^{-}_{rad}(a,\mu_a^*)-m^{0}_{rad}(a,\mu_a^*))$, we choose $u_{\epsilon}\in \mathcal{P}_{a,\mu_a^*,rad}^{0}$  such that
$\Psi_{\mu_a^*}(u_{\epsilon})<m^{0}_{rad}(a,\mu_a^*)+\epsilon$. Direct calculations give that
\begin{equation*}
\Psi_{\mu_a^*}(u_{\epsilon})=-\frac{(q_2\gamma_{q_2}-p)(p-q_1\gamma_{q_1})}{pq_1\gamma_{q_1}q_2\gamma_{q_2}}\|\nabla u_{\epsilon}\|_p^p.
\end{equation*}
Let $t>1$, $b(t)>0$ satisfy  $b(t)=\frac{q_2\gamma_{q_2}-q_1\gamma_{q_1}}{q_2\gamma_{q_2}-p}t^{p-q_1\gamma_{q_1}}-\frac{p-q_1\gamma_{q_1}}{q_2\gamma_{q_2}-p}t^{q_2\gamma_{q_2}-q_1\gamma_{q_1}}$.
For any $u\in \mathcal{P}_{a,\mu_a^*,rad}^{0}$, we define the map $\mathcal{T}: u\to w_t=(u)_{t}$. Then $w_t$ satisfies
\begin{equation}\label{e7.3}
\begin{cases}
&\|\nabla w_t\|_p^p=\mu(t)\gamma_{q_1} \|w_t\|_{q_1}^{q_1}+\|w_t\|_{p^*}^{p^*},\\
&p\|\nabla w_t\|_p^p<\mu(t)q_1\gamma_{q_1}^2 \|w_t\|_{q_1}^{q_1}+p^*\|w_t\|_{p^*}^{p^*},\\
&\|w_t\|_{p^*}^{p^*}=\frac{\mu_a^*\gamma_{q_1}(p-q_1\gamma_{q_1})}{p^*-p}t^{p^*-q_1\gamma_{q_1}}\|w_t\|_{q_1}^{q_1},
\end{cases}
\end{equation}
where $\mu(t)=b(t)\mu_a^*$. Thus, by \eqref{e7.3}, the map $\mathcal{T}$ projects $\mathcal{P}_{a,\mu_a^*,rad}^0$ into $\mathcal{P}_{a,\mu(t),rad}^-$. Direct calculations give that
\begin{equation*}
\Psi_{\mu(t)}(w_t)=\frac{p-q_1\gamma_{q_1}}{q_1\gamma_{q_1}}\left(\frac{1}{q_2\gamma_{q_2}}t^{q_2\gamma_{q_2}}-\frac{1}{p}t^{p}\right)\|\nabla u\|_p^p
\end{equation*}
and
\begin{equation*}
\min_{t\geq 1}\Psi_{\mu(t)}(w_t)=\Psi_{\mu(1)}(w_1)=-\frac{(q_2\gamma_{q_2}-p)(p-q_1\gamma_{q_1})}{pq_1\gamma_{q_1}q_2\gamma_{q_2}}\|\nabla u\|_p^p,
\end{equation*}
with $\max_{t\geq 1}b(t)=b(1)=1$.  For any $t>1$ such that $b(t)>0$, we let $w_{\epsilon,t}:=\mathcal{T}(u_{\epsilon})$. Then $w_{\epsilon,t}\in \mathcal{P}_{a,\mu(t),rad}^-$ and by Proposition \ref{pro4.7}, for any $t>1$,
\begin{equation*}
\frac{p-q_1\gamma_{q_1}}{q_1\gamma_{q_1}}\left(\frac{1}{q_2\gamma_{q_2}}t^{q_2\gamma_{q_2}}-\frac{1}{p}t^{p}\right)\|\nabla u_{\epsilon}\|_p^p=\Psi_{\mu(t)}(w_{\epsilon,t})\geq m^-_{rad}(a,\mu(t))\geq m^-_{rad}(a,\mu_a^*).
\end{equation*}
Letting $t\downarrow 1$, we obtain that
\begin{equation*}
m^-_{rad}(a,\mu_a^*)>m^{0}_{rad}(a,\mu_a^*)+\epsilon>\Psi_{\mu_a^*}(u_\epsilon)\geq m^-_{rad}(a,\mu_a^*),
\end{equation*}
which is a contradiction. Thus, $m^{0}_{rad}(a,\mu_a^*)\geq m^-_{rad}(a,\mu_a^*)$ always holds true for $q_2=p^*$.
\end{proof}

With Lemma~\ref{lem5.3} in hands, we can obtain the following result.

\begin{lemma}\label{lem5.5}
Let $N\geq 2$, $1<p<N$, $p<q_1<p+\frac{p^2}{N}<q_2\leq p^*$ and  $a>0$. Then
\begin{enumerate}
\item[$(1)$]\quad $\hat{\mu}_{a,+}^{**}\geq\hat{\mu}_{a,-}^{**}>\mu_a^*$, provided
\begin{enumerate}
\item[$(a)$]\quad $q_2<p^*$;
\item[$(b)$]\quad $q_2=p^*$ and $q_1\in\left(q_1^*, p+\frac{p^2}{N}\right)$ where $q_1^*>p$ is given in Proposition~\ref{pro4.1}.
\end{enumerate}
\item[$(2)$]\quad $\hat{\mu}_{a,+}^{**}>\mu_a^*$ for all $q_1\in \left(p, p+\frac{p^2}{N}\right)$ if $q_2=p^*$.
\end{enumerate}
\end{lemma}

\begin{proof}
By Lemma \ref{lem5.4} and Proposition \ref{pro2.1}, $\mathcal{P}_{c,\mu,rad}^{\pm}\neq \emptyset$ and $m^{+}_{rad}(c,\mu)\leq m^{-}_{rad}(c,\mu)$ for any $c,\mu>0$.  Thus, we always have $\hat{\mu}_{a,+}^{**}\geq\hat{\mu}_{a,-}^{**}$, provided $\hat{\mu}_{a,\pm}^{**}$ are well defined which is guaranteed by Lemma~\ref{lem5.3}, provided $(a)$ or $(b)$ holds. Moreover, by Lemma~\ref{lem5.3} again, $\hat{\mu}_{a,+}^{**}$ is also well defined for all $q_1\in \left(p, p+\frac{p^2}{N}\right)$ if $q_2=p^*$.

It remains to prove that $\hat{\mu}_{a,\pm}^{**}>\mu_a^*$.  Suppose the contrary that $\hat{\mu}_{a,\pm}^{**}=\mu_a^*$.  Then there exists $\mu_{n}\downarrow \mu_a^*$ as $n\to +\infty$ and $0<c_n\leq a$ such that
\begin{equation}\label{e5.12}
m_{rad}^{\pm}(c_n,\mu_n)\geq m_{rad}^0(c_n,\mu_n)\ \text{for\ any\ }n.
\end{equation}
Without loss of generality, we assume that $c_0=\lim_{n\to +\infty}c_n$.  If $c_0<a$,
then by Proposition \ref{pro2.2} and $\lim_{n\to +\infty}\mu_{n}=\mu_{a}^*$, $\mu_{n}<\mu_{c_n}^*$ for $n$ sufficiently large. Thus, by Proposition \ref{pro2.1}, $\mathcal{P}^0_{c_n,\mu_n}=\emptyset$ and $m_{rad}^0(c_n,\mu_n)=+\infty$ by the definitions.  On the other hand, by Lemma \ref{lem5.4} and Proposition \ref{pro2.1}, $\mathcal{P}_{c_n,\mu_n,rad}^{\pm}\neq \emptyset$, which implies that $m_{rad}^{\pm}(c_n,\mu_n)<+\infty$.  It follows that $m_{rad}^{\pm}(c_n,\mu_n)< m_{rad}^0(c_n,\mu_n)$ for $n$ sufficiently large, which contradicts (\ref{e5.12}).  Thus, we must have $c_0=a$.  Moreover, since by Lemma \ref{lem5.4} and Proposition \ref{pro2.1}, $\mathcal{P}_{c_n,\mu_n,rad}^{\pm}\neq \emptyset$, which implies that $m_{rad}^{\pm}(c_n,\mu_n)<+\infty$, by \eqref{e5.12} and the definition of $m_{rad}^0(c_n,\mu_n)$, we always have $\mathcal{P}^0_{c_n,\mu_n,rad}\neq\emptyset$ for any $n$ and $m_{rad}^0(c_n,\mu_n)\gtrsim -1$ by similar arguments used in the proofs of Propositions~\ref{pro3.1} and \ref{pro3.21}.
 We claim that
\begin{equation}\label{e5.13}
\limsup_{n\to+\infty}m_{rad}^{\pm}(c_n,\mu_n)\leq m_{rad}^{\pm}(a,\mu_a^*)
\end{equation}
and
\begin{equation}\label{e5.14}
\liminf_{n\to+\infty}m_{rad}^0(c_n,\mu_n)\geq m_{rad}^0(a,\mu_a^*)\ \text{provided}\ (a) \ \text{or}\  (b)\ \text{holds},
\end{equation}
which, together with Lemma~\ref{lem5.3}, contradicts (\ref{e5.12}) for $n$ sufficiently large provided $(a)$ or $(b)$ holds.  It remains to  prove that $\hat{\mu}_{a,+}^{**}>\mu_a^*$ for all $q_1\in \left(p, p+\frac{p^2}{N}\right)$ if $q_2=p^*$. In this case, let $\tilde{u}_{n}\in\mathcal{P}_{c_{n},\mu_n,rad}^0$ such that $\Psi_{\mu_n}\left(\tilde{u}_{n}\right)\leq m_{rad}^{0}(c_{n},\mu_n)+o_{n}(1)$, where $o_{n}(1)\to0$ as $n\to+\infty$.  It follows from  $\tilde{u}_{n}\in\mathcal{P}_{c_{n},\mu_n,rad}^0$ and $\mu_n>\mu_a^*$ that
\begin{eqnarray*}
\|\nabla \tilde{u}_{n}\|_p^p=\mu_n\gamma_{q_1}\|\tilde{u}_{n}\|_{q_1}^{q_1}+\gamma_{q_2}\|\tilde{u}_{n}\|_{q_2}^{q_2}
>\mu_a^*\gamma_{q_1}\|\tilde{u}_{n}\|_{q_1}^{q_1}+\gamma_{q_2}\|\tilde{u}_{n}\|_{q_2}^{q_2},
\end{eqnarray*}
which, together with $\mu_a^*<\mu_n=\mu(\tilde{u}_n)$ and $\|\nabla \tilde{u}_n\|_p^p\thickapprox\|\tilde{u}_n\|_{q_1}^{q_1}\thickapprox\|\tilde{u}_n\|_{q_2}^{q_2}\thickapprox1$ obtained by similar arguments used in the proofs of Propositions~\ref{pro3.1} and \ref{pro3.21}, implies that there exists $1\lesssim t^+_{\mu_{a}^*}(\tilde{u}_{n}) < 1<t^-_{\mu_{a}^*}(\tilde{u}_{n})\lesssim1$ such that $\left(\tilde{u}_{n}\right)_{t^{\pm}_{\mu_a^*}\left(\tilde{u}_{n}\right)}\in \mathcal{P}_{c_{n},\mu_a^*,rad}^{\pm}$.  It follows that
\begin{eqnarray*}
m_{rad}^{+}(a,\mu_a^*)&\geq&\limsup_{n\to+\infty}m_{rad}^{+}(c_{n},\mu_n)\\
&\geq&\limsup_{n\to+\infty}m_{rad}^{0}(c_{n},\mu_n)=\limsup_{n\to+\infty}\Psi_{\mu_n}\left(\tilde{u}_{n}\right)\\
&\geq&\limsup_{n\to+\infty}\Psi_{\mu_n}\left(\left(\tilde{u}_{n}\right)_{t^{-}_{\mu_a^*}\left(\tilde{u}_{n}\right)}\right)\\
&=&\limsup_{n\to+\infty}\Psi_{\mu_a^*}\left(\left(\tilde{u}_{n}\right)_{t^{-}_{\mu_a^*}\left(\tilde{u}_{n}\right)}\right)\\
&\geq&\limsup_{n\to+\infty}m_{rad}^{-}(c_{n},\mu_a^*)\\
&\geq&m_{rad}^{-}(a,\mu_a^*)\\
&>&m_{rad}^{+}(a,\mu_a^*).
\end{eqnarray*}
That is a contradiction, which implies that  $\hat{\mu}_{a,+}^{**}>\mu_{a}^{*}$ for all $q_1\in \left(p, p+\frac{p^2}{N}\right)$ if $q_2=p^*$.

It remains to prove (\ref{e5.13}) and (\ref{e5.14}).  We first prove  (\ref{e5.13}).  By the definition of $m_{rad}^{\pm}(a,\mu_a^*)$, for any $\epsilon>0$ there exists $u\in \mathcal{P}^{\pm}_{a,\mu_a^*,rad}$ such that
\begin{equation*}
m_{rad}^{\pm}(a,\mu_a^*)\leq \Psi_{\mu_a^*}(u)<m_{rad}^{\pm}(a,\mu_a^*)+\epsilon.
\end{equation*}
Set $w_n:=\frac{c_n}{a}u$. Then $w_n\in \mathcal{S}_{c_n}$.  Since $u\in \mathcal{P}^{\pm}_{a,\mu_a^*,rad}$, by Proposition~\ref{pro2.1}, $\mu_a^*<\mu(u)$, which together with $\lim_{n\to+\infty}\mu_n=\mu_a^*$, $\lim_{n\to+\infty}c_n=a$ and \eqref{muu}, implies that $\mu_n<\mu(w_n)$ for $n$ sufficiently large.  It follows from Proposition \ref{pro2.1} that there exists $t^{\pm}_{\mu_n}(w_n)$ with $t^{\pm}_{\mu_n}(w_n)\to 1$ as $n\to +\infty$ such that $(w_n)_{t^{\pm}_{\mu_n}(w_n)}\in \mathcal{P}^{\pm}_{c_n,\mu_n,rad}$ and
\begin{equation*}
\limsup_{n\to+\infty}m_{rad}^{\pm}(c_n,\mu_n)\leq  \limsup_{n\to+\infty}\Psi_{\mu_n}\left((w_n)_{t^{\pm}_{\mu_n}(w_n)}\right)=\Psi_{\mu_a^*}(u)<m_{rad}^{\pm}(a,\mu_a^*)+\epsilon.
\end{equation*}
By the arbitrariness of $\epsilon$, we obtain (\ref{e5.13}).  We next prove  (\ref{e5.14}).  By the definition of $m_{rad}^0(c_n,\mu_n)$, for any $\epsilon>0$ there exists $u_n\in \mathcal{P}^{0}_{c_n,\mu_n,rad}$ such that
\begin{equation*}
m_{rad}^{0}(c_n,\mu_n)\leq \Psi_{\mu_n}(u_n)<m_{rad}^{0}(c_n,\mu_n)+\epsilon.
\end{equation*}
As before, we also have $\|\nabla u_n\|_p^p\thickapprox\|u_n\|_{q_1}^{q_1}\thickapprox\|u_n\|_{q_2}^{q_2}\thickapprox1$.
Set $v_n:=\frac{a}{c_n}u_n$.  Then $v_n\in \mathcal{S}_{a}$.  It follows from $\mu_n=\mu(u_n)$, $\lim_{n\to+\infty}\mu_n=\mu_a^*$ and $\lim_{n\to+\infty}c_n=a$  that $\mu_a^*=\lim_{n\to+\infty}\mu(v_n)$.  Thus, $\{v_n\}$ is a bounded radial minimizing sequence of $\mu_a^*$ such that
\begin{eqnarray*}
\left\{\aligned
&\|\nabla v_n\|_p^p=\mu\gamma_{q_1}\|v_n\|_{q_1}^{q_1}+\gamma_{q_2}\|v_n\|_{q_2}^{q_2}+o_n(1),\\
&p\|\nabla v_n\|_p^p=\mu q_1\gamma_{q_1}^2\|v_n\|_{q_1}^{q_1}+q_2\gamma_{q_2}^2\|v_n\|_{q_2}^{q_2}+o_n(1).
\endaligned\right.
\end{eqnarray*}
By the same arguments used for Proposition~\ref{pro4.1}, we can show that there exists $u_0\in W_{rad}^{1,p}\left(\mathbb{R}^N\right)\backslash\{0\}$ such that $v_n\to u_0$ strongly in $W^{1,p}_{rad}\left(\mathbb{R}^N\right)$ as $n\to+\infty$, provided $(a)$ or $(b)$ holds.  It follows that $u_0\in \mathcal{P}^0_{a,\mu_a^*,rad}$ and
\begin{equation*}
\liminf_{n\to+\infty}m_{rad}^{0}(c_n,\mu_n)+\epsilon\geq  \liminf_{n\to+\infty}\Psi_{\mu_n}(u_n)=\Psi_{\mu_a^*}(u_0)\geq m_{rad}^{0}(a,\mu_a^*).
\end{equation*}
By the arbitrariness of $\epsilon$, we obtain (\ref{e5.14}).  It completes the proof.
\end{proof}

We also define
\begin{eqnarray}\label{mu2}
\tilde{\mu}_{a,\pm}^{**}=\sup\left\{\mu\geq\mu_a^*\mid m^{\pm}_{rad}(c,\mu)\geq m^{\pm}_{rad}(a,\mu)\text{ for all }0<c< a\right\}.
\end{eqnarray}
Then by Proposition~\ref{pro4.4} and Lemma~\ref{lem5.3}, $\tilde{\mu}_{a,+}^{**}$ is always well defined while, $\tilde{\mu}_{a,-}^{**}$ is well defined provided
\begin{enumerate}
\item[$(1)$]\quad $q_2<p^*$;
\item[$(2)$]\quad $q_2=p^*$ and $q_1\in\left(q_1^*, p+\frac{p^2}{N}\right)$ where $q_1^*>p$ is given in Proposition~\ref{pro4.1}.
\end{enumerate}
To show that $\tilde{\mu}_{a,\pm}^{**}>\mu_a^*$, we need the following two results.
\begin{lemma}\label{lem5.7}
Assume that $N\geq 2$, $1<p<N$, $p<q_1<p+\frac{p^2}{N}<q_2\leq  p^*$ and  $a>0$.  Then
\begin{equation*}
\lim_{\mu\downarrow \mu_a^*}m_{rad}^{\pm}(b,\mu)= m_{rad}^{\pm}(b,\mu_a^*)\ \text{for\ any\ }0<b\leq a.
\end{equation*}
\end{lemma}
\begin{proof}
It suffices to prove that
\begin{equation}\label{e5.16}
\limsup_{n\to+\infty}m_{rad}^{\pm}(b,\mu_n)\leq m_{rad}^{\pm}(b,\mu_a^*)
\end{equation}
and
\begin{equation}\label{e5.17}
\liminf_{n\to+\infty}m_{rad}^{\pm}(b,\mu_n)\geq m_{rad}^{\pm}(b,\mu_a^*)
\end{equation}
for any $\mu_n\downarrow \mu_a^*$ as $n\to+\infty$.  The proof of (\ref{e5.16}) is similar to that of (\ref{e5.13}), so we omit it.  We next prove (\ref{e5.17}).  By the definition of $m_{rad}^{\pm}(b,\mu_n)$, for any $\epsilon>0$ there exists $u_n\in \mathcal{P}^{\pm}_{b,\mu_n,rad}$ such that
\begin{equation*}
m_{rad}^{\pm}(b,\mu_n)\leq \Psi_{\mu_n}(u_n)<m_{rad}^{\pm}(b,\mu_n)+\epsilon,
\end{equation*}
which, together with \eqref{e5.16} and similar estimates in the proofs of Propositions~\ref{pro3.1}, \ref{pro3.21} and \ref{pro3.2}, implies that $\{u_n\}$ is bounded in $W^{1,p}\left(\mathbb{R}^N\right)$ with $\|u_n\|_{q_1}^{q_1}\gtrsim 1$ for $u_n\in \mathcal{P}^{+}_{b,\mu_n,rad}$ and $\|u_n\|_{q_2}^{q_2}\gtrsim 1$ for $u_n\in \mathcal{P}^{-}_{b,\mu_n,rad}$.
Since $\mu_n<\mu(u_n)$ and $\mu_{n}\downarrow \mu_a^*$ as $n\to +\infty$, we obtain that $\mu_a^*<\mu(u_n)$ for any $n$.  Thus, there exists $t^{\pm}_{\mu_a^*}(u_n)$ such that $(u_n)_{t^{\pm}_{\mu_a^*}(u_n)}\in \mathcal{P}^{\pm}_{a,\mu_a^*,rad}$.  Note that
\begin{equation*}
\|\nabla u_n\|_p^p=\mu_n\gamma_{q_1}\|u_n\|_{q_1}^{q_1}+\gamma_{q_2}\|u_n\|_{q_2}^{q_2}>\mu_a^*\gamma_{q_1}\|u_n\|_{q_1}^{q_1}+\gamma_{q_2}\|u_n\|_{q_2}^{q_2},
\end{equation*}
by Proposition \ref{pro2.1}, $t^{+}_{\mu_a^*}(u_n)<1<t^{-}_{\mu_a^*}(u_n)$ for all $n$.  It follows from $\|u_n\|_{q_1}^{q_1}\gtrsim 1$ for $u_n\in \mathcal{P}^{+}_{b,\mu_n,rad}$, $\|u_n\|_{q_2}^{q_2}\gtrsim 1$ for $u_n\in \mathcal{P}^{-}_{b,\mu_n,rad}$ and
\begin{equation*}
(t^{\pm}_{\mu_a^*}(u_n))^p\|\nabla u_n\|_p^p=\mu_a^*\gamma_{q_1}(t^{\pm}_{\mu_a^*}(u_n))^{q_1\gamma_{q_1}}\|u_n\|_{q_1}^{q_1}+\gamma_{q_2}(t^{\pm}_{\mu_a^*}(u_n))^{q_2\gamma_{q_2}}\|u_n\|_{q_2}^{q_2}
\end{equation*}
that $1\lesssim t^{\pm}_{\mu_a^*}(u_n)\lesssim 1$.  Thus, by Proposition~\ref{pro2.1},
\begin{eqnarray*}
\liminf_{n\to+\infty}m_{rad}^{+}(b,\mu_n)+\epsilon&\geq&\liminf_{n\to+\infty}\Psi_{\mu_n}(u_n)\\
&=&\liminf_{n\to+\infty}\Psi_{\mu_a^*}(u_n)\\
&\geq&\liminf_{n\to+\infty}\Psi_{\mu_a^*}\left((u_n)_{t^{+}_{\mu_a^*}(u_n)}\right)\\
&\geq&m_{rad}^{+}(b,\mu_a^*)
\end{eqnarray*}
for $u_n\in \mathcal{P}^{+}_{b,\mu_n,rad}$ and
\begin{eqnarray*}
\liminf_{n\to+\infty}m_{rad}^{-}(b,\mu_n)+\epsilon&\geq&\liminf_{n\to+\infty}\Psi_{\mu_n}(u_n)\\
&\geq&\liminf_{n\to+\infty}\Psi_{\mu_n}\left((u_n)_{t^{-}_{\mu_a^*}(u_n)}\right)\\
&=&\liminf_{n\to+\infty}\Psi_{\mu_a^*}\left((u_n)_{t^{-}_{\mu_a^*}(u_n)}\right)\\
&\geq&m_{rad}^{-}(b,\mu_a^*)
\end{eqnarray*}
for $u_n\in \mathcal{P}^{-}_{b,\mu_n,rad}$.  By the arbitrariness of $\epsilon$, we obtain (\ref{e5.17}).  It completes the proof.
\end{proof}

\begin{lemma}\label{lem2.2}
Let $N\geq 2$, $1<p<N$, $p<q_1<p+\frac{p^2}{N}<q_2\leq p^*$ and $a>0$. Let $\{u_n\}\subset \mathcal{P}_{a,\mu_{a}^*,rad}^+$ be a minimizing sequence of $m^{+}(a,\mu_{a}^*)$, then up to a subsequence $u_n\to u_0$ strongly in $W^{1,p}(\mathbb{R}^N)$ for some $u_0\in \mathcal{P}_{a,\mu_{a}^*,rad}^+$.
\end{lemma}

\begin{proof}
We use the arguments in the proof of Proposition \ref{pro3.1}. First we have
$\|\nabla u_n\|_p^p\thickapprox\|u_n\|_{q_1}^{q_1}\thickapprox\|u_n\|_{q_2}^{q_2}\thickapprox1$.
Next we claim that there exists $\delta>0$ sufficiently small such that
\begin{equation}\label{e3.77}
\inf_{\mathcal{P}_{a,\mu_a^*}^{\delta,+}}\Psi_{\mu_a^*}(u)=\inf_{\mathcal{P}_{a,\mu_a^*}^{+}}\Psi_{\mu_a^*}(u)=m^{+}(a,\mu_a^*).
\end{equation}
Suppose the contrary.  Then by Lemma~\ref{lem2.1}, there exists $\delta_n\to0$ as $n\to+\infty$,  $\varphi_n\in \mathcal{S}_a$ and $\phi_n\in \mathcal{P}_{a,\mu_a^*}^+$ such that $\varphi_n-\phi_n\to0$ strongly in $W^{1,p}\left(\mathbb{R}^N\right)$ as $n\to+\infty$ and
\begin{eqnarray*}
\Psi_{\mu_a^*}(\varphi_n)<m^+(a,\mu_a^*)
\end{eqnarray*}
for all $n$.  Since $\mu^{*}_{a}\leq \mu(u)$ for any $u\in \mathcal{S}_a$, by Proposition~\ref{pro2.1}, there exists $\tau_{\mu_a^*}^{\pm}(\varphi_n)>0$ such that $\left(\varphi_n\right)_{\tau^{\pm}_{\mu_a^*}(\varphi_n)}\in\mathcal{P}_{a,\mu_a^*}^{\pm}\cup \mathcal{P}_{a,\mu_a^*}^{0}$.  It follows from $\phi_n\in \mathcal{P}_{a,\mu_a^*}^+$ and $\varphi_n-\phi_n\to0$ strongly in $W^{1,p}\left(\mathbb{R}^N\right)$ as $n\to+\infty$ that
\begin{equation}\label{0266}
\left\{\aligned
&\left(\tau_{\mu_a^*}^{\pm}\left(\varphi_n\right)\right)^{p-q_1\gamma_{q_1}}\|\nabla \varphi_n\|_p^p=\mu_a^*\gamma_{q_1}\|\varphi_n\|_{q_1}^{q_1}+\left(\tau_{\mu_a^*}^{\pm}\left(\varphi_n\right)\right)^{q_2\gamma_{q_2}-q_1\gamma_{q_1}}\gamma_{q_2}\|\varphi_n\|_{q_2}^{q_2},
\\
&\|\nabla \varphi_n\|_p^p=\mu_a^*\gamma_{q_1}\|\varphi_n\|_{q_1}^{q_1}+\gamma_{q_2}\|\varphi_n\|_{q_2}^{q_2}+o(1).
\endarray\right.
\end{equation}
Since $\{\phi_n\}$ is a minimizing sequence of $m^+(a,\mu_a^*)$ and $\varphi_n-\phi_n\to0$ strongly in $W^{1,p}\left(\mathbb{R}^N\right)$ as $n\to+\infty$,
as that for $\{u_n\}$, we know that $\|\varphi_n\|_{q_1}\thickapprox\|\varphi_n\|_{q_2}\thickapprox\|\nabla \varphi_n\|_p\thickapprox1$, which, together with \eqref{0266}, implies that $\tau_{\mu_a^*}^{\pm}\left(\varphi_n\right)\thickapprox1$.  Thus, up to a subsequence, we may assume that $\tau_{\mu_a^*}^{\pm}\left(\varphi_n\right)\to \tau_{\mu_a^*}^{\pm}$ as $n\to+\infty$.  We also denote $A=\lim_{n\to+\infty}\|\nabla \phi_n\|_p^p$, $B=\|\phi_n\|_{q_1}^{q_1}$ and $C=\|\phi_n\|_{q_2}^{q_2}$.  Then by \eqref{0266}, $\|\varphi_n\|_{q_1}\thickapprox\|\varphi_n\|_{q_2}\thickapprox\|\nabla \varphi_n\|_p\thickapprox1$ and $\varphi_n-\phi_n\to0$ strongly in $W^{1,p}\left(\mathbb{R}^N\right)$ as $n\to+\infty$, we have
\begin{eqnarray*}
\left(\tau_{\mu_a^*}^{\pm}\right)^{p}A=\left(\tau_{\mu_a^*}^{\pm}\right)^{q_1\gamma_{q_1}}\mu_a^*\gamma_{q_1}B+\left(\tau_{\mu_a^*}^{\pm}\right)^{q_2\gamma_{q_2}}\gamma_{q_2}C.
\end{eqnarray*}
Since $\phi_n\in \mathcal{P}_{a,\mu_a^*}^+$, we also have
\begin{eqnarray*}
\left\{\aligned
&A=\mu_a^*\gamma_{q_1}B+\gamma_{q_2}C,\\
&pA\geq\mu_a^* q_1\gamma_{q_1}^2B+q_2\gamma_{q_2}^2C.
\endaligned\right.
\end{eqnarray*}
By similar computations in the proof of Proposition~\ref{pro2.1}, we know that $\tau_{\mu_a^*}^+=1$.  If $\tau_{\mu_a^*}^->1$, then $\left(\varphi_n\right)_{\tau^{\pm}_{\mu_a^*}(\varphi_n)}\in\mathcal{P}_{a,\mu_a^*}^{\pm}$ and  by Proposition~\ref{pro2.1} again,
\begin{equation*}
m^+(a,\mu_a^*)\le\Psi_{\mu_a^*}\left((\varphi_n)_{\tau^+_{\mu_a^*}(\varphi_n)}\right)\leq\Psi_{\mu_a^*}(\varphi_n)<m^+(a,\mu_a^*)
\end{equation*}
for $n$ sufficiently large, which is impossible.  If $\tau_{\mu_a^*}^-=1$, then
\begin{equation*}
m^+(a,\mu_a^*)\geq \lim_{n\to+\infty}\Psi_{\mu_a^*}(\varphi_n)=\lim_{n\to+\infty}\left((\varphi_n)_{\tau^-_{\mu_a^*}(\varphi_n)}\right)\geq \min\{m^-(a,\mu_a^*),m^0(a,\mu_a^*)\}>m^+(a,\mu_a^*),
\end{equation*}
which is also impossible. Thus, \eqref{e3.77} holds true for $\delta>0$ sufficiently small. In this position, we can adapt the same arguments used in the proofs of Propositions \ref{pro3.1} and \ref{pro4.5} to show that $u_n\to u_0$ strongly in $W^{1,p}(\mathbb{R}^N)$ for some $u_0\in \mathcal{P}_{a,\mu_{a}^*,rad}^+$.
\end{proof}

With Lemmas~\ref{lem5.7} and \ref{lem2.2} in hands, we can obtain a positive gap between $\tilde{\mu}_{a,\pm}^{**}$ and $\mu_a^*$.
\begin{lemma}\label{lem5.6}
Let $N\geq 2$, $1<p<N$, $p<q_1<p+\frac{p^2}{N}<q_2\leq p^*$ and  $a>0$. Then
\begin{enumerate}
\item[$(1)$]\quad  $\tilde{\mu}_{a,\pm}^{**}>\mu_a^*$ holds true, provided
\begin{enumerate}
\item[$(a)$]\quad $q_2<p^*$;
\item[$(b)$]\quad $q_2=p^*$ and $q_1\in\left(q_1^*, p+\frac{p^2}{N}\right)$ where $q_1^*>p$ is given in Proposition~\ref{pro4.1}.
\end{enumerate}
\item[$(2)$]\quad  $\tilde{\mu}_{a,+}^{**}>\mu_a^*$  for all $q_1\in \left(p, p+\frac{p^2}{N}\right)$ if $q_2=p^*$.
\end{enumerate}
\end{lemma}
\begin{proof}
Assume by contradiction that $\tilde{\mu}_{a,\pm}^{**}=\mu_a^*$.  Then there exists $\mu_{n}^{\pm}\downarrow \mu_a^*$ as $n\to +\infty$ and $0<c_n^{\pm}< a$ such that
\begin{equation}\label{e5.15}
m_{rad}^{\pm}\left(c_n^{\pm},\mu_n^{\pm}\right)< m_{rad}^{\pm}\left(a,\mu_n^{\pm}\right)\ \text{for\ any\ }n.
\end{equation}
Up to a subsequence, we assume $c_0^{\pm}=\lim_{n\to +\infty}c_n^{\pm}$.  If $c_0^{\pm}<a$ then
by Proposition \ref{pro2.2}, there exists $c_0^{\pm}<b<a$ such that $c_n<b<a$,  $\mu_n^{\pm}<\mu_{c_n^{\pm}}^*$ and $\mu_n^{\pm}<\mu_b^*$ for $n$ sufficiently large. By Lemmas~\ref{lem3.1} and \ref{lem5.7} and Propositions~\ref{pro4.5} and \ref{pro4.7}, we obtain that
\begin{equation*}
\liminf_{n\to+\infty}m_{rad}^{\pm}\left(c_n^\pm,\mu_n^\pm\right)\geq\liminf_{n\to+\infty}m_{rad}^{\pm}\left(b,\mu_n^\pm\right)=m_{rad}^{\pm}(b,\mu_a^*)>m_{rad}^{\pm}(a,\mu_a^*),
\end{equation*}
which contradicts (\ref{e5.15}) and Lemma \ref{lem5.7} for $n$ sufficiently large.  Thus, we must have $c_0=a$. We may assume that $c_n^\pm\uparrow a$ as $n\to+\infty$ up to a subsequence.  By the definition of $m_{rad}^{\pm}\left(c_n^\pm,\mu_n^\pm\right)$, for any $\epsilon_n\to 0$ as $n\to +\infty$ with $\epsilon_n< m_{rad}^{\pm}\left(a,\mu_n^\pm\right)-m_{rad}^{\pm}\left(c_n^\pm,\mu_n^\pm\right)$, there exists $u_n^\pm\in \mathcal{P}^{\pm}_{c_n^\pm,\mu_n^\pm,rad}$ such that
\begin{equation*}
m_{rad}^{\pm}\left(c_n^\pm,\mu_n^\pm\right)\leq \Psi_{\mu_n^\pm}(u_n)<m_{rad}^{\pm}\left(c_n^\pm,\mu_n^\pm\right)+\epsilon_n.
\end{equation*}
By (\ref{e5.13}) and \eqref{e5.15}, we can adapt similar arguments used in the proof of Lemma~\ref{lem5.7} to show that $\{u_n^\pm\}$ is bounded in $W^{1,p}\left(\mathbb{R}^N\right)$ with $\|u_n^+\|_{q_1}^{q_1}\gtrsim \|\nabla u_n^+\|_{p}^{p} \gtrsim 1$ and $\|u_n^-\|_{q_2}^{q_2}\gtrsim \|\nabla u_n^+\|_{p}^{p}\gtrsim 1$.  Let $v_n^\pm=\frac{a}{c_n^\pm}u_n^\pm$.  Then $v_n^\pm\in \mathcal{S}_{a}$. It follows from $\mu_n^\pm<\mu\left(u_n^\pm\right)$, $\lim_{n\to +\infty}\mu_n^\pm=\mu_a^*$ and $\lim_{n\to +\infty}c_n^\pm=a$ that $\mu_a^*\leq \liminf_{n\to +\infty}\mu\left(u_n^\pm\right)=\liminf_{n\to +\infty}\mu\left(v_n^\pm\right)$.
If
\begin{equation*}
\mu_a^*<\liminf_{n\to +\infty}\mu\left(u_n^\pm\right)=\liminf_{n\to +\infty}\mu\left(v_n^\pm\right),
\end{equation*}
then $\mu_a^*<\mu\left(v_n^\pm\right)$,  $\mu_n<\mu\left(v_n^\pm\right)$  and $\mu_n<\mu\left(u_n^\pm\right)$ for $n$ large enough. Thus, by $\|u_n^\pm\|_p=c_n^\pm$, $\|v_n^\pm\|_p=a$, $\mu_n<\mu\left(v_n^\pm\right)$, $\mu_n<\mu\left(u_n^\pm\right)$ and similar arguments used for Lemma~\ref{lem3.1}, we obtain that
\begin{equation*}
m^{\pm}_{rad}\left(c_n^\pm,\mu_n^\pm\right)+\epsilon_n>\Psi_{\mu_n}\left(u_n^\pm\right)>\Psi_{\mu_n^\pm}\left(\left(\frac{a}{c_n^\pm}u_n\right)_{t_{\mu_n^\pm}^{\pm}\left(\frac{a}{c_n^\pm}u_n^\pm\right)}\right)\geq m^{\pm}_{rad}\left(a,\mu_n^\pm\right),
\end{equation*}
which, together with the choice of $\epsilon_n$, implies that
\begin{equation*}
m^{\pm}_{rad}\left(a,\mu_n^\pm\right)>m^{\pm}_{rad}\left(c_n^\pm,\mu_n^\pm\right)+\epsilon_n> m^{\pm}_{rad}\left(a,\mu_n^\pm\right).
\end{equation*}
It is a  contradiction.  If
\begin{equation*}
\mu_a^*=\liminf_{n\to +\infty}\mu\left(u_n^\pm\right)=\liminf_{n\to +\infty}\mu\left(v_n^\pm\right),
\end{equation*}
then the proof need to be divided into two cases.  For the first case that $(a)$ or $(b)$ holds, then $u_n^\pm\in \mathcal{P}^{\pm}_{c_n^{\pm},\mu_n^{\pm},rad}$ with
$\|u_n^\pm\|_{q_1}^{q_1}\thickapprox\|u_n^\pm\|_{q_2}^{q_2}\thickapprox\|\nabla u_n^\pm\|_{p}^{p}\thickapprox 1$. It follows that
$v_n^\pm\in \mathcal{S}_a$ is a radial minimizing sequence of $\mu_a^*$ with $\|v_n^\pm\|_{q_1}^{q_1}\thickapprox\|v_n^\pm\|_{q_2}^{q_2}\thickapprox\|\nabla v_n^\pm\|_{p}^{p}\thickapprox 1$.  Thus, by similar arguments used for Proposition~\ref{pro4.1}, we can show that there exists $v_0^\pm\in W^{1,p}_{rad}\left(\mathbb{R}^N\right)\backslash \{0\}$ such that $v_n^\pm\to v_0^\pm$ strongly in $W^{1,p}\left(\mathbb{R}^N\right)$ and $\mu_a^*=\mu(v_0^\pm)$, provided $(a)$ or $(b)$ holds.  By the definition of $v_n^\pm$ and $u_n^\pm\in \mathcal{P}^{\pm}_{c_n^{\pm},\mu_n^{\pm},rad}$, we have $v_0^\pm\in \mathcal{P}^{0}_{a,\mu_a^*,rad}$ and by Lemma~\ref{lem5.7},
\begin{equation*}
m_{rad}^{0}(a,\mu_a^*)\leq \Psi_{\mu_a^*}(v_0^\pm)=\limsup_{n\to+\infty}\Psi_{\mu_n^\pm}(u_n^\pm)\leq\limsup_{n\to+\infty} \left( m_{rad}^{\pm}\left(c_n^\pm,\mu_n^\pm\right)+\epsilon_n\right)\leq m_{rad}^{\pm}(a,\mu_a^*),
\end{equation*}
which implies that
\begin{equation*}
m_{rad}^{0}(a,\mu_a^*)\leq m_{rad}^{\pm}(a,\mu_a^*).
\end{equation*}
It contradicts Lemma~\ref{lem5.3}. For the second case $q_2=p^*$ and $q_1\in\left(q_1^*, p+\frac{p^2}{N}\right)$, we denote
$A:=\lim_{n\to+\infty}\|\nabla u_{n}^+\|_p^p$, $B:=\lim_{n\to+\infty}\|u_{n}^+\|_{q_1}^{q_1}$ and $C:=\lim_{n\to+\infty}\|u_{n}^+\|_{p^*}^{p^*}$ up to a subsequence. Then
\begin{equation}\label{e7.10}
\mu_a^*=\frac{(p^*-p)(p-q_1\gamma_{q_1})^{\frac{p-q_1\gamma_{q_1}}{p^*-p}}}
{\gamma_{q_1}(p^*-q_1\gamma_{q_1})^{\frac{p^*-q_1\gamma_{q_1}}{p^*-p}}}
\frac{A^{\frac{p^*-q_1\gamma_{q_1}}{p^*-p}}}{B\,C^{\frac{p-q_1\gamma_{q_1}}{p^*-p}}}.
\end{equation}
By  (\ref{e5.13}), we have
\begin{eqnarray}\label{e7.8}
m_{rad}^{\pm}(a,\mu_a^*)\geq \limsup_{n\to+\infty}m_{rad}^{\pm}(a,\mu_n^+).
\end{eqnarray}					
By  $u_{n}^+\in\mathcal{P}_{c_n^+,\mu_n^+,rad}^{+}$ and $\mu_n^+>\mu_a^*$, we have
\begin{eqnarray*}
\|\nabla u_{n}^+\|_p^p=\mu_n^+\gamma_{q_1}\|u_{n}^+\|_{q_1}^{q_1}+\|u_{n}^+\|_{p^*}^{p^*}>\mu_a^*\gamma_{q_1}\|u_{n}^+\|_{q_1}^{q_1}+\|u_{n}^+\|_{p^*}^{p^*},
\end{eqnarray*}
which, together with $\mu_a^*<\mu_n^+<\mu(u_n^+)$, implies that there exists $1\lesssim t^+_{\mu_{a}^*}(u_{n}^+)<1<t^-_{\mu_{a}^*}(u_{n}^+)\lesssim1$ such that $\left(u_{n}^+\right)_{t^{\pm}_{\mu_a^*}\left(u_{n}^+\right)}\in \mathcal{P}_{c_n,\mu_a^*,rad}^{\pm}$.  It follows that
\begin{eqnarray*}
\limsup_{n\to+\infty}m_{rad}^{+}(a,\mu_n^+)&\geq&\limsup_{n\to+\infty}m_{rad}^{+}(c_{n}^+,\mu_n^+)\\
&\geq&\limsup_{n\to+\infty}\Psi_{\mu_n^+}\left(u_{n}^+\right)\\
&=&\limsup_{n\to+\infty}\Psi_{\mu_a^*}\left(u_{n}^+\right)\\
&\geq&\limsup_{n\to+\infty}\Psi_{\mu_a^*}\left(\left(u_{n}^+\right)_{t^{+}_{\mu_a^*}\left(u_{n}^+\right)}\right)\\
&\geq&\limsup_{n\to+\infty}m_{rad}^{+}(c_{n}^+,\mu_a^*)\\
&\geq&m_{rad}^{+}(a,\mu_a^*),
\end{eqnarray*}
which, together with (\ref{e7.8}), implies that $\limsup_{n\to+\infty}\Psi_{\mu_a^*}\left(u_{n}^+\right)=m_{rad}^{+}(a,\mu_a^*)$ up to a subsequence.
By the definition of $\mu_a^*$, we know $\mu_a^*\leq \mu(v_n^+)$.
There exists $\tau^{+}_{\mu_a^*}(v_n^+)>0$ such that $(v_n^+)_{\tau^{+}_{\mu_a^*}(v_n^+)}\in \mathcal{P}^0_{a,\mu_a^*,rad}\cup \mathcal{P}^{+}_{a,\mu_a^*,rad}$. Hence,
\begin{equation}\label{e*}
\left\{\aligned
&\left(\tau^{+}_{\mu_a^*}\left(v_{n}^+\right)\right)^p\|\nabla v_{n}^+\|_p^p=\left(\tau^{+}_{\mu_a^*}\left(v_{n}^+\right)\right)^{q_1\gamma_{q_1}}\mu_a^*\gamma_{q_1}\|v_{n}^+\|_{q_1}^{q_1}
+\left(\tau^{+}_{\mu_a^*}\left(v_{n}^+\right)\right)^{p^*}\|v_{n}^+\|_{p^*}^{p^*},\\
&\left(\tau^{+}_{\mu_a^*}\left(v_{n}^+\right)\right)^pp\|\nabla v_{n}^+\|_p^p\geq \left(\tau^{+}_{\mu_a^*}\left(v_{n}^+\right)\right)^{q_1\gamma_{q_1}}\mu_a^* q_1\gamma_{q_1}^2\|v_{n}^+\|_{q_1}^{q_1}+\left(\tau^{+}_{\mu_a^*}\left(v_{n}^+\right)\right)^{p^*}p^*\|v_{n}^+\|_{p^*}^{p^*}.
\endaligned\right.
\end{equation}
Then $\tau^{+}_{\mu_a^*}\left(v_{n}^+\right)\sim1$ as $n\to+\infty$.  Let $\tau^{+}_{\mu_a^*}:=\lim_{n\to+\infty}\tau^{+}_{\mu_a^*}\left(v_{n}^+\right)$. Then by $u_n^+\in \mathcal{P}_{c_n^+,\mu_n^+,rad}^+$ and (\ref{e*}), we have
\begin{eqnarray*}
\left\{\aligned
&A=\mu_a^*\gamma_{q_1}B+C,\\
&pA\geq \mu_a^* q_1\gamma_{q_1}^2B+p^*C
\endaligned\right.
\end{eqnarray*}
and
\begin{equation*}
\left\{\aligned
&\left(\tau^{+}_{\mu_a^*}\right)^pA=\left(\tau^{+}_{\mu_a^*}\right)^{q_1\gamma_{q_1}}\mu_a^*\gamma_{q_1}B
+\left(\tau^{+}_{\mu_a^*}\right)^{p^*}C,\\
&\left(\tau^{+}_{\mu_a^*}\right)^ppA\geq \left(\tau^{+}_{\mu_a^*}\right)^{q_1\gamma_{q_1}}\mu_a^* q_1\gamma_{q_1}^2B+\left(\tau^{+}_{\mu_a^*}\right)^{p^*}p^*C.
\endaligned\right.
\end{equation*}
Consider the function
\begin{eqnarray*}
g(s)=\frac{s^p}{p}A-\frac{\mu_a^*s^{q_1\gamma_{q_1}}}{q_1}B-\frac{s^{p^*}}{p^*}C.
\end{eqnarray*}
In view of (\ref{e7.10}), by similar arguments for Proposition \ref{pro2.1}, we know that $g'(1)=g''(1)=0$ and $g(s)$ is strictly decreasing in $(0, +\infty)$, which implies that $1$ is the unique critical point of $g(s)$ in $(0, +\infty)$.  It follows  that $\tau^{+}_{\mu_a^*}=1$.
If $\mu_a^*=\mu(v_n^+)$, then   $(v_n^+)_{\tau^{+}_{\mu_a^*}(v_n^+)}\in \mathcal{P}^0_{a,\mu_a^*,rad}$ and
\begin{equation*}
m_{rad}^{\pm}(a,\mu_a^*)=\lim_{n\to+\infty}\Psi_{\mu_a^*}\left(u_{n}^+\right)=
\lim_{n\to+\infty}\Psi_{\mu_a^*}\left( (v_n^+)_{\tau^{+}_{\mu_a^*}(v_n^+)}\right)
\geq m_{rad}^{0}(a,\mu_a^*)>m_{rad}^{+}(a,\mu_a^*).
\end{equation*}
That is a contradiction. If $\mu_a^*<\mu(v_n^+)$, then $(v_n^+)_{\tau^{+}_{\mu_a^*}(v_n^+)}\in \mathcal{P}^{+}_{a,\mu_a^*,rad}$ and
\begin{eqnarray*}
\lim_{n\to+\infty}\Psi_{\mu_a^*}\left( (v_n^+)_{\tau^{+}_{\mu_a^*}(v_n^+)}   \right)
=\lim_{n\to+\infty}\Psi_{\mu_a^*}\left(u_{n}^+\right)=m_{rad}^{\pm}(a,\mu_a^*).
\end{eqnarray*}
Thus, $\left(v_{n}^+\right)_{\tau^{+}_{\mu_a^*}\left(v_{n}^+\right)}$ is a minimizing sequence of $m_{rad}^{\pm}(a,\mu_a^*)$ and by Lemma \ref{lem2.2}, we  know that $\left(v_{n}^+\right)_{\tau^{+}_{\mu_a^*}\left(v_{n}^+\right)}\to u_{a,\mu_a^*}^{+}\in\mathcal{P}_{a,\mu_a^*,rad}^{+}$ strongly in $W^{1,p}\left(\mathbb{R}^N\right)$ as $n\to+\infty$, which implies that $v_n^+,\ u_{n}^+\to u_{a,\mu_a^*}^{+}\in\mathcal{P}_{a,\mu_a^*,rad}^{+}$ strongly in $W^{1,p}\left(\mathbb{R}^N\right)$ as $n\to+\infty$. By $\mu_a^*=\lim_{n\to +\infty}\mu(v_n^+)$, we obtain that
$\mu_a^*=\mu\left(u_{a,\mu_a^*}^{+}\right)$, which is  a contradiction.  It completes the proof.
\end{proof}

\subsection{The existence of ground-state solutions.}
Let
\begin{eqnarray*}
\mu_a^{**}=\sup\left\{\mu\geq\mu_a^*\mid m^+_{rad}(a,\mu)<m^-_{rad}(a,\mu)\right\}.
\end{eqnarray*}
Then by Propositions~\ref{pro2.1}, \ref{pro3.1}, \ref{pro3.21}, \ref{pro3.2}, \ref{pro4.5} and \ref{pro4.7} and Lemma~\ref{lem5.7}, $\mu_a^{**}>\mu_a^*$ is well defined.  We recall that
\begin{eqnarray*}
\mu_{a,+}^{**}=\min\left\{\mu_a^{**}, \hat{\mu}_{a,+}^{**}, \tilde{\mu}_{a,+}^{**}\right\}
\end{eqnarray*}
is the second extremal value of $\mathcal{P}_{a,\mu}^{+}$, where $\hat{\mu}_{a,+}^{**}$ and $\tilde{\mu}_{a,+}^{**}$ are given by \eqref{mu1} and \eqref{mu2}, respectively.
\begin{proposition}\label{pro5.1}
Let $N\geq 2$, $1<p<N$, $p<q_1<p+\frac{p^2}{N}<q_2\leq p^*$ and $a>0$.  Then the variational problem
\begin{equation}\label{e5.54}
\hat{\Psi}_{\mu,rad}^{+}:=\inf\left\{\Psi_{\mu}(u) \mid u\in \mathcal{P}_{a,\mu,rad}^{+}\cup\mathcal{P}_{a,\mu,rad}^{0}\right\}
\end{equation}
is achieved by some
$u_{a,\mu,+}\in \mathcal{P}_{a,\mu,rad}^{+}$ for $\mu_{a}^*<\mu<\mu_{a,+}^{**}$.
Moreover, $u_{a,\mu,+}$ also satisfies the equation (\ref{e1.1}) in the weak sense
for a suitable Lagrange multiplier $\lambda_{a,\mu,+}<0$.
\end{proposition}
\begin{proof}
Since this proof is similar to that of Proposition~\ref{pro3.1}, we only sketch it and point out the differences.  First of all, as that in the proof of Proposition~\ref{pro3.1} and by $\mu<\hat{\mu}_{a,+}^{**}$, we can show that any minimizing sequence of $\hat{\Psi}_{\mu,rad}^{+}$, say $\{u_n\}\subset \mathcal{P}_{a,\mu,rad}^{+}$, is bounded in $W^{1,p}\left(\mathbb{R}^N\right)$ with
$\|\nabla u_n\|_{p}\thickapprox\|u_n\|_{q_1}\thickapprox\|u_n\|_{q_2}\thickapprox1$.  We then also claim that
there exists $\delta>0$ sufficiently small such that
\begin{equation*}
\inf_{\left(\mathcal{P}_{a,\mu,rad}^{+}\cup \mathcal{P}_{a,\mu,rad}^{0}\right)^{\delta}}\Psi_{\mu}(u)=\inf_{\mathcal{P}_{a,\mu,rad}^{+}\cup \mathcal{P}_{a,\mu,rad}^{0}}\Psi_{\mu}(u)=m^{+}_{rad}(a,\mu),
\end{equation*}
where $\left(\mathcal{P}_{a,\mu,rad}^{+}\cup \mathcal{P}_{a,\mu,rad}^{0}\right)^{\delta}:=\left\{u\in \mathcal{S}_a\mid \text{dist}_{W^{1,p}_{rad}}(u,\mathcal{P}_{a,\mu,rad}^{+}\cup \mathcal{P}_{a,\mu,rad}^{0})\leq \delta\right\}$ with
\begin{equation*}
\text{dist}_{W^{1,p}_{rad}}\left(u,\mathcal{P}_{a,\mu,rad}^{+}\cup \mathcal{P}_{a,\mu,rad}^{0}\right)=\inf_{v\in \left(\mathcal{P}_{a,\mu,rad}^{+}\cup \mathcal{P}_{a,\mu,rad}^{0}\right)}\|u-v\|_{W^{1,p}_{rad}}.
\end{equation*}
This can be achieved by applying the same arguments in the proof of Proposition~\ref{pro3.1}, provided we can show that $\mu(u_n)-\mu\gtrsim1$ as $n\to+\infty$, where $\{u_n\}\subset\mathcal{P}_{a,\mu,rad}^{+}\cup \mathcal{P}_{a,\mu,rad}^{0}$ is a minimizing sequence.  Since $\mu<\hat{\mu}_{a,+}^{**}$, by the definition of $\hat{\mu}_{a,+}^{**}$, we must have $\{u_n\}\subset\mathcal{P}_{a,\mu,rad}^{+}$, which implies that $\mu(u_n)-\mu>0$.  If
\begin{eqnarray*}
\liminf_{n\to +\infty}\left(\mu\left(u_n\right)-\mu\right)=0,
\end{eqnarray*}
then by $u_n\in \mathcal{P}_{a,\mu,rad}^{+}$, Proposition~\ref{pro2.1} and $\|\nabla u_n\|_{p}\thickapprox\|u_n\|_{q_1}\thickapprox\|u_n\|_{q_2}\thickapprox1$, we must have
\begin{eqnarray*}
\left\{\aligned
&\|\nabla u_n\|_{p}^{p}-\mu \gamma_{q_1}\|u_n\|_{q_1}^{q_1}-\gamma_{q_2}\|u_n\|_{q_2}^{q_2}=0,\\
&p\|\nabla u_n\|_{p}^{p}-\mu q_1 \gamma_{q_1}^2\|u_n\|_{q_1}^{q_1}-q_2\gamma_{q_2}^2\|u_n\|_{q_2}^{q_2}=o_n(1).
\endaligned\right.
\end{eqnarray*}
Moreover, by Proposition~\ref{pro2.1} once more, there exists $t^-_{\mu}(u_n)>1$ such that $\left(u_n\right)_{t^-_{\mu}(u_n)}\in\mathcal{P}^{-}_{a,\mu,rad}$, that is,
\begin{eqnarray*}
\left\{\aligned
&\left(t^-_{\mu}(u_n)\right)^p\|\nabla u_n\|_{p}^{p}-\mu\gamma_{q_1}  \left(t^-_{\mu}(u_n)\right)^{q_1\gamma_{q_1}}\|u_n\|_{q_1}^{q_1}-\gamma_{q_2}\left(t^-_{\mu}(u_n)\right)^{q_2\gamma_{q_2}}\|u_n\|_{q_2}^{q_2}=0,\\
&p\left(t^-_{\mu}(u_n)\right)^p\|\nabla u_n\|_{p}^{p}-\mu q_1 \gamma_{q_1}^2   \left(t^-_{\mu}(u_n)\right)^{q_1\gamma_{q_1}}\|u_n\|_{q_1}^{q_1}-q_2\gamma_{q_2}^2\left(t^-_{\mu}(u_n)\right)^{q_2\gamma_{q_2}}\|u_n\|_{q_2}^{q_2}<0.
\endaligned\right.
\end{eqnarray*}
We denote $A=\liminf_{n\to+\infty}\|\nabla u_n\|_{p}^{p}$, $B=\liminf_{n\to+\infty}\|u_n\|_{q_1}^{q_1}$ and $C=\liminf_{n\to+\infty}\|u_n\|_{q_2}^{q_2}$.  Since $\{u_n\}$ is bounded in $W^{1,p}\left(\mathbb{R}^N\right)$ with $\|\nabla u_n\|_{p}\thickapprox\|u_n\|_{q_1}\thickapprox\|u_n\|_{q_2}\thickapprox1$, we can also see that $t^-_{\mu}\left(u_n\right)\thickapprox1$.  Thus, passing to the limit, we have
\begin{eqnarray*}
\left\{\aligned
&A-\mu \gamma_{q_1}  B-\gamma_{q_2}C=0,\\
&pA-\mu q_1 \gamma_{q_1}^2   B-q_2\gamma_{q_2}^2C=0
\endaligned\right.
\end{eqnarray*}
and
\begin{eqnarray*}
\left\{\aligned
&\left(t^-\right)^pA-\mu \gamma_{q_1}  \left(t^-\right)^{q_1\gamma_{q_1}}B-\gamma_{q_2}\left(t^-\right)^{q_2\gamma_{q_2}}C=0,\\
&p\left(t^-\right)^pA-\mu q_1 \gamma_{q_1}^2   \left(t^-\right)^{q_1\gamma_{q_1}}B-q_2\gamma_{q_2}^2\left(t^-\right)^{q_2\gamma_{q_2}}C\leq0,
\endaligned\right.
\end{eqnarray*}
where $t^-=\liminf_{n\to+\infty}t^-_{\mu}\left(u_n\right)$.
By the same computations in the proof of Propositioin~\ref{pro2.1}, we must have $t^-=1$.  Thus,
\begin{eqnarray*}
m^-_{rad}\left(a,\mu\right)\leq\Psi_{\mu}\left(\left(u_n\right)_{t^-_{\mu}(u_n)}\right)=\Psi_{\mu}\left(u_n\right)+o_n(1)\leq m_{rad}^{+}\left(a,\mu\right)+o_n(1),
\end{eqnarray*}
which contradicts $\mu<\mu_a^{**}$ and the definition of $\mu_a^{**}$.  Thus, we must have $\mu(u_n)-\mu\gtrsim 1$ as $n\to +\infty$ and also have $\liminf_{n\to+\infty}t^-_{\mu}\left(u_n\right)=t^->1$.  In this position, we can go through the remaining proof of Proposition~\ref{pro3.1} to show that the variational problem~\eqref{e5.54} is achieved by some
$u_{a,\mu,+}\in \mathcal{P}_{a,\mu,rad}^{+}$.  Moreover, $u_{a,\mu,+}$ also satisfies the equation (\ref{e1.1})
for a suitable Lagrange multiplier $\lambda_{a,\mu,+}<0$.
\end{proof}

\subsection{The existence of mountain-pass solutions.}
Let
\begin{eqnarray*}
\tilde{\Phi}(a,\mu)=\left\{u\in W^{1,p}_{rad}\left(\mathbb{R}^N\right)\backslash\{0\}\mid
\begin{array}{l}
\text{There\ exists\ a\ minimizing\ sequence}\ \{u_n\}\subset \mathcal{P}_{a,\mu,rad}^{-}\ \\
 \text{of}\  m_{rad}^-\left(a,\mu\right)\ \text{such\ that}\ u_n\rightharpoonup u\ \text{weakly\ in\ }  W^{1,p}_{rad}\left(\mathbb{R}^N\right)
\end{array}
\right\}
\end{eqnarray*}
and
\begin{eqnarray*}
\overline{\mu}_{a,-}^{**}=\sup\left\{\mu\geq\mu_a^*\mid \mu\leq\mu(u) \ \text{for\ all}\ u\in\tilde{\Phi}(a,\mu)\right\},
\end{eqnarray*}
where $\hat{\mu}_{a,-}^{**}$ is given by \eqref{mu1}.
\begin{lemma}\label{lem7.5}
Let $N\geq 2$, $1<p<N$, $p<q_1<p+\frac{p^2}{N}<q_2<p^*$ and  $a>0$. Then $\overline{\mu}_{a,-}^{**}>\mu_a^*$ is well defined.
\end{lemma}

\begin{proof}
For any $a,\mu>0$, since $m_{rad}^-(a,\mu)>-\infty$ by Lemmas \ref{lem3.1} and \ref{lem2.1}, we can choose $\{u_n\}\subset \mathcal{P}^{-}_{a,\mu,rad}$ such that it is a minimizing sequence of $m_{rad}^-(a,\mu)$. By similar arguments in the proof of Proposition \ref{pro3.21}, we have $\{u_n\}$ is bounded in $W^{1,p}_{rad}(\mathbb{R}^N)$ and $\|\nabla u_n\|_p^p\thickapprox \| u_n\|_{q_2}^{q_2}\thickapprox 1$.  Thus, there exists $u_0\in W^{1,p}_{rad}(\mathbb{R}^N)$ such that $u_n\to u_0$ weakly in $W^{1,p}_{rad}(\mathbb{R}^N)$ and strongly in $L^q(\mathbb{R}^N)$ for all $q\in (p,p^*)$.  In the subcritical case $q_2<p^*$, it is easy to see that $u_0\not=0$, which implies that $\tilde{\Phi}(a,\mu)\neq \emptyset$ for any $a,\mu>0$.  By Proposition \ref{pro2.2} and the proof of Proposition \ref{pro4.7}, $\overline{\mu}_{a,-}^{**}$ is well defined.

It remains to prove that $\overline{\mu}_{a,-}^{**}>\mu_a^*$.  Suppose the contrary that $\overline{\mu}_{a,-}^{**}=\mu_a^*$. Then there exists $\mu_n\downarrow \mu_a^*$ as $n\to +\infty$ and  $u_{n}\in\tilde{\Phi}(a,\mu_n)$ such that $\mu_n\geq \mu(u_n)$ for all $n$.  Set $b_{n}:=\|u_n\|_p\in(0, a]$.  Thus, by $(b)$ of Proposition~\ref{pro2.2}, we must have $\lim_{n\to+\infty}b_n=a$.  It follows form the definition of $\mu_a^*$ that $v_n:=\frac{a}{b_n}u_{n}\in\mathcal{S}_{a}$ is a radial minimizing sequence of $\mu_a^*$ as $n\to+\infty$. 	By $u_{n}\in\tilde{\Phi}(a,\mu_n)$ and the definition of $\tilde{\Phi}(a,\mu_n)$, there exists a minimizing sequence $\left\{w_k^{(n)}\right\}\subset \mathcal{P}_{a,\mu_n,rad}^-$ of $m_{rad}^-(a,\mu_n)$ such that $w_k^{(n)}\rightharpoonup u_n$ weakly in $W^{1,p}(\mathbb{R}^N)$ as $k\to+\infty$. By Lemma \ref{lem5.7} and similar arguments in the proof of Proposition \ref{pro3.21}, we obtain that $\left\{w_k^{(n)}\right\}$ is bounded in $W^{1,p}(\mathbb{R}^N)$ and $\left\|\nabla w_k^{(n)}\right\|_p^p\thickapprox \left\| w_k^{(n)}\right\|_{q_2}^{q_2}\thickapprox 1$ for all $k$ and $n$. Thus $\{u_n\}$
is  bounded in $W^{1,p}(\mathbb{R}^N)$ and $\|\nabla u_n\|_p^p\thickapprox \|u_n\|_{q_2}^{q_2}\thickapprox 1$, which implies that $\{v_n\}$
is  bounded in $W^{1,p}(\mathbb{R}^N)$ and $\|\nabla v_n\|_p^p\thickapprox \|v_n\|_{q_2}^{q_2}\thickapprox 1$.  Thus, $v_n\to v_0$ weakly in $W^{1,p}_{rad}(\mathbb{R}^N)$ and strongly in $L^q(\mathbb{R}^N)$ for all $q\in (p,p^*)$ for some $v_0\in W^{1,p}_{rad}(\mathbb{R}^N)$.  As above, we can show that $v_0\not=0$.  Thus, we also have $\left\| w_k^{(n)}\right\|_{q_1}^{q_1}\thickapprox\|u_n\|_{q_1}^{q_1}\thickapprox 1$.
Since the functional $\mu(u)$, given by \eqref{muu}, is $0$-homogenous by Proposition~\ref{pro2.2}, by similar arguments in the proof of Proposition \ref{pro4.1}, we can show that $v_n\to v_0$ strongly in $W^{1,p}_{rad}(\mathbb{R}^N)$ as $n\to+\infty$ and $\mu_a^*=\mu(v_0)$.  It follows that $u_n\to v_0$ strongly in $W^{1,p}_{rad}(\mathbb{R}^N)$ $n\to+\infty$.  By Proposition \ref{pro2.1}, there exists $t_{\mu_a^*}^0(v_0)>0$ such that $(v_0)_{t_{\mu_a^*}^0(v_0)}\in \mathcal{P}_{a,\mu_a^*,rad}^{0}$.  Let us consider the following function
 \begin{equation*}
 f(t):=\left(\frac{1}{p}-\frac{1}{q_2\gamma_{q_2}}\right)t^p\|\nabla v_0\|_p^p-\mu_a^*\gamma_{q_1}\left(\frac{1}{q_1\gamma_{q_1}}-\frac{1}{q_2\gamma_{q_2}}\right)t^{q_1\gamma_{q_1}}\|v_0\|_{q_1}^{q_1}.
 \end{equation*}
Direct calculations show that $f(t)$ has a unique critical point, which is a globally minimal point, $\tilde{t}\in[0,+\infty)$, where
 \begin{equation*}
 \tilde{t}^{p-q_1\gamma_{q_1}}=\frac{\mu_a^*\gamma_{q_1}(q_2\gamma_{q_2}-q_1\gamma_{q_1})\|v_0\|_{q_1}^{q_1}}{(q_2\gamma_{q_2}-p)\|\nabla v_0\|_p^p}.
 \end{equation*}
 On the other hand, it follows from $(v_0)_{t_{\mu_a^*}^0(v_0)}\in \mathcal{P}_{a,\mu_a^*,rad}^{0}$ that
  \begin{equation*}
(t_{\mu_a^*}^0(v_0))^p(q_2\gamma_{q_2}-p)\|\nabla v_0\|_p^p=(t_{\mu_a^*}^0(v_0))^{q_1\gamma_{q_1}}\mu_a^*\gamma_{q_1}(q_2\gamma_{q_2}-q_1\gamma_{q_1})\|v_0\|_{q_1}^{q_1},
 \end{equation*}
 which implies that
  \begin{equation*}
(t_{\mu_a^*}^0(v_0))^{p-q_1\gamma_{q_1}}=\frac{\mu_a^*\gamma_{q_1}(q_2\gamma_{q_2}-q_1\gamma_{q_1})\|v_0\|_{q_1}^{q_1}}{(q_2\gamma_{q_2}-p)\|\nabla v_0\|_p^p}.
 \end{equation*}
Hence $\tilde{t}=t_{\mu_a^*}^0(v_0)$, which, together with (\ref{e5.13}), $\{w_k^{(n)}\}\subset \mathcal{P}_{a,\mu_n,rad}^-$, the Fatou lemma and $u_n\to v_0$ strongly in $W^{1,p}_{rad}(\mathbb{R}^N)$, implies that
\begin{eqnarray*}
m_{rad}^-(a,\mu_a^*)&\geq& \lim_{n\to+\infty}m_{rad}^-(a,\mu_n)\\
&=&\lim_{n\to+\infty}\lim_{k\to+\infty}\Psi_{\mu_n}\left(w_k^{(n)}\right)\\
&=&\lim_{n\to+\infty}\lim_{k\to+\infty}\left( \left(\frac{1}{p}-\frac{1}{q_2\gamma_{q_2}}\right)\left\|\nabla w_k^{(n)}\right\|_p^p-\mu_n\gamma_{q_1}\left(\frac{1}{q_1\gamma_{q_1}}-\frac{1}{q_2\gamma_{q_2}}\right)\left\|w_k^{(n)}\right\|_{q_1}^{q_1}  \right)\\
&\geq& \lim_{n\to+\infty}\left( \left(\frac{1}{p}-\frac{1}{q_2\gamma_{q_2}}\right)\|\nabla u_n\|_p^p-\mu_n\gamma_{q_1}\left(\frac{1}{q_1\gamma_{q_1}}-\frac{1}{q_2\gamma_{q_2}}\right)\|u_n\|_{q_1}^{q_1}  \right)\\
&=&\left(\frac{1}{p}-\frac{1}{q_2\gamma_{q_2}}\right)\|\nabla v_0\|_p^p-\mu_a^*\gamma_{q_1}\left(\frac{1}{q_1\gamma_{q_1}}-\frac{1}{q_2\gamma_{q_2}}\right)\|v_0\|_{q_1}^{q_1}\\
&\geq& \left(\frac{1}{p}-\frac{1}{q_2\gamma_{q_2}}\right)\tilde{t}^p\|\nabla v_0\|_p^p-\mu_a^*\gamma_{q_1}\left(\frac{1}{q_1\gamma_{q_1}}-\frac{1}{q_2\gamma_{q_2}}\right)\tilde{t}^{q_1\gamma_{q_1}}\|v_0\|_{q_1}^{q_1}\\
&=&\Psi_{\mu_a^*}\left((v_0)_{\tilde{t}}\right)\\
&\geq& m_{rad}^0(a,\mu_a^*).
\end{eqnarray*}
It contradicts Lemma \ref{lem5.3}.  Thus, we must have $\overline{\mu}_{a,-}^{**}>\mu_a^*$, which completes the proof.
  \end{proof}

We recall that
\begin{eqnarray}\label{mu-}
\mu_{a,-}^{**}=\min\left\{\hat{\mu}_{a,-}^{**}, \tilde{\mu}_{a,-}^{**}, \overline{\mu}_{a,-}^{**}\right\}
\end{eqnarray}
is the second extremal value of $\mathcal{P}_{a,\mu,rad}^{-}$ for $q_2<p^*$, where $\hat{\mu}_{a,+}^{**}$ and $\tilde{\mu}_{a,+}^{**}$ are given by \eqref{mu1} and \eqref{mu2}, respectively.

\begin{proposition}\label{pro5.2}
Let $N\geq 2$, $1<p<N$, $p<q_1<p+\frac{p^2}{N}<q_2<p^*$ and $a>0$.  Then the variational problem
\begin{equation}\label{e5.55}
\hat{\Psi}_{\mu,rad}^{-}:=\inf\left\{\Psi_{\mu}(u)\mid u\in \mathcal{P}_{a,\mu,rad}^{-}\cup \mathcal{P}_{a,\mu,rad}^{0}\right\}
\end{equation}
is achieved by some
$u_{a,\mu,-}\in \mathcal{P}_{a,\mu,rad}^{-}$ for $\mu_{a}^*<\mu<\mu_{a,-}^{**}$.
Moreover, $u_{a,\mu,-}$ also satisfies the equation (\ref{e1.1}) in the weak sense
for a suitable Lagrange multiplier $\lambda_{a,\mu,-}<0$.
\end{proposition}

\begin{proof}
Since this proof is similar to that of Proposition~\ref{pro3.21}, we only sketch it and point out the differences.  First of all, as that in the proof of Proposition~\ref{pro3.21} and by $\mu<\hat{\mu}_{a,-}^{**}$, we can show that any minimizing sequence of $\hat{\Psi}_{\mu,rad}^{-}$, say $\{u_n\}\subset \mathcal{P}_{a,\mu,rad}^{-}$, is bounded in $W_{rad}^{1,p}\left(\mathbb{R}^N\right)$ with
$\|\nabla u_n\|_{p}\thickapprox\|u_n\|_{q_2}\thickapprox1$.  It follows that there exists $u_0\in W_{rad}^{1,p}\left(\mathbb{R}^N\right)$ such that
\begin{eqnarray*}\label{e5.100}
\left\{\aligned
&u_n\rightharpoonup u_0\quad \text{weakly in } W^{1,p}\left(\mathbb{R}^N\right),\\  &u_n\to u_0\quad \text{a.e. in } \mathbb{R}^N,\\
&u_n\to u_0\quad \text{strongly in } L^t\left(\mathbb{R}^N\right) \text{ for all } t\in (p,p^*)
\endaligned\right.
\end{eqnarray*}
as $n\to+\infty$.  As that in the proof of Proposition~\ref{pro3.21}, we can show that $u_0\not=0$.  It follows from the Fatou lemma that $u_0\in \tilde{\Phi}(a,\mu)$, which implies that $\mu\leq\mu(u_0)$ by $\mu<\overline{\mu}_{a,-}^{**}$ and the definition of $\overline{\mu}_{a,-}^{**}$.  By Proposition \ref{pro2.1}, there  exists $\tau^-_{\mu}(u_0)>0$ such that $(u_0)_{\tau^-_{\mu}(u_0)}\in \mathcal{P}_{a_1,\mu,rad}^-\cup \mathcal{P}_{a_1,\mu,rad}^0$, where $a_1:=\|u_0\|_p\in (0,a]$. As that in the proof of Proposition~\ref{pro3.21},
\begin{eqnarray*}
1\lesssim t^+_{\mu}(u_n)+o_n(1)\leq \tau_{\mu}^-(u_0)\leq1+o_n(1).
\end{eqnarray*}
If $(u_0)_{\tau^-_{\mu}(u_0)}\in \mathcal{P}_{a_1,\mu,rad}^0$, then by $\mu<\hat{\mu}_{a,-}^{**}$ and $\mu<\tilde{\mu}_{a,-}^{**}$, we have
\begin{eqnarray*}
m^-_{rad}(a,\mu)&=&\lim_{n\to +\infty}\Psi_{\mu}(u_n)\\
&\geq&\lim_{n\to +\infty}\Psi_{\mu}\left(\left(u_n\right)_{\tau^-_\mu\left(u_0\right)}\right)\\
&\geq&\Psi_{\mu}\left(\left(u_0\right)_{\tau^-_\mu\left(u_0\right)}\right)\\
&\geq& m^0_{rad}(a_1,\mu)\\
&>& m^-_{rad}(a_1,\mu)\\
&\geq& m^-_{rad}(a,\mu),
\end{eqnarray*}
which is a contradiction. Thus we must have  $(u_0)_{\tau^-_{\mu}(u_0)}\in \mathcal{P}_{a_1,\mu,rad}^-$.  In this position, we can go through the remaining proof of Proposition~\ref{pro3.21} to show that the variational problem~\eqref{e5.55} is achieved by some
$u_{a,\mu,-}\in \mathcal{P}_{a,\mu,rad}^{-}$.  Moreover, $u_{a,\mu,-}$ also satisfies the equation (\ref{e1.1})
for a suitable Lagrange multiplier $\lambda_{a,\mu,-}<0$.
\end{proof}

As for $\mu<\mu_a^*$, we next consider \eqref{e3.11} for $q_2=p^*$ and $\mu>\mu_a^*$ by using the subcritical perturbation argument.  For this purpose, we re-denote $\mu_{a}^*$, $m^-_{rad}(a,\mu)$, $\Psi_{\mu}(u)$, $\mathcal{P}_{a,\mu}^-$, $t^-_\mu\left(u\right)$, $\lambda_{a, \mu,-}<0$, $\mu_{a,-}^{**}$ and $u_{a,\mu,-}$ in Proposition~\ref{pro5.2} by $\mu_{a,q_2}^*$, $m^-_{rad,q_2}(a,\mu)$, $\Psi_{\mu,q_2}(u)$, $t^-_{\mu,q_2}\left(u\right)$, $\mathcal{P}_{a,\mu,q_2}^-$, $\lambda_{a, \mu,q_2,-}<0$, $\mu_{a,-,q_2}^{**}$ and $u_{a,\mu,q_2,-}$, respectively.  We also define
\begin{eqnarray*}
\mu_{a,-}^{***}=\sup\left\{\mu\geq\mu_a^*\mid m^{-}_{rad}(a,\mu)<m^{+}_{rad}(a,\mu)+\frac{1}{N}S^{\frac{N}{p}}\right\}.
\end{eqnarray*}
By Lemmas~\ref{lem4.6} and \ref{lem5.7}, $\mu_{a,-}^{***}>\mu_a^*$ is well defined.
We recall that
\begin{eqnarray*}
\mu_{a,-}^{**}=\inf\left\{\mu_{a,-}^{***}, \mu_{a,-,0}^{**}\right\}
\end{eqnarray*}
is the second extremal value of $\mathcal{P}_{a,\mu}^-$ for $q_2=p^*$, where $\mu_{a,-,0}^{**}=\limsup_{q_2\to p^*}\mu_{a,-,q_2}^{**}$.
\begin{lemma}\label{lem7.7}
Let $N\geq 2$, $1<p<N$, $p<q_1<p+\frac{p^2}{N}<q_2=p^*$ and $a>0$.  Then $\mu_{a,-}^{**}>\mu_a^*$, provided $q_1\in\left(q_1^*, p+\frac{p^2}{N}\right)$ where $q_1^*>p$ is given in Proposition~\ref{pro4.1}.
\end{lemma}
\begin{proof}
Since $\mu_{a,-}^{***}>\mu_a^*$, we only need to show that $\mu_{a,-,0}^{**}>\mu_a^*$.  Suppose the contrary that $\mu_{a,-,0}^{**}=\mu_a^*$, then by \eqref{mu-}, one of the following three cases must happen:
\begin{enumerate}
\item[$(1)$]\quad $\hat{\mu}_{a,q_2,-}^{**}\to \mu_a^*$ as $q_2\to p^*$ up to a subsequence;
\item[$(2)$]\quad $\tilde{\mu}_{a,q_2,-}^{**}\to \mu_a^*$ as $q_2\to p^*$ up to a subsequence;
\item[$(3)$]\quad $\overline{\mu}_{a,q_2,-}^{**}\to \mu_a^*$ as $q_2\to p^*$ up to a subsequence.
\end{enumerate}
We can exclude the case~$(1)$ by similar arguments used for Lemma~\ref{lem5.5}.  The only additional argument lies in proving \eqref{e5.14}, where in this case, we have $\mu_{q_{2,n}}(v_{n})\to \mu_a^*$ as $q_{2,n}\to p^*$ up to a subsequence, which together with the H\"older inequality that $\|v_{q_{2,n}}\|_{q_{2,n}}^{q_{2,n}}-\|v_{q_{2,n}}\|_{p^*}^{p^*}\to0$ as $q_{2,n}\to p^*$ up to a subsequence.  Thus, we can still go through the proof of Lemma~\ref{lem5.5} to exclude the case~$(1)$.  

We can also exclude the case~$(2)$ by similar arguments used for Lemma~\ref{lem5.6}.  The only additional argument is that in this case, if $c_0^-<a$ then we have $\mu_a^*<\mu_b^*<\mu_{c_0^-}^*$ by choosing $c_0^-<b<a$.  In this situation, we need to apply the perturbation argument used for Proposition~\ref{pro3.2} to replace the usage of Lemma~\ref{lem5.7} to obtain the continuity of $m_{rad}^-(b,\mu,q_2)$ as a function of $(\mu, q_2)$ to exclude that $c_0^-<a$, since $\mu\to\mu_a^*<\mu_b^*$.  Thus, we can still go through the proof of Lemma~\ref{lem5.6} to exclude the case~$(2)$ by further noticing that $\|v_{q_{2,n}}\|_{q_{2,n}}^{q_{2,n}}-\|v_{q_{2,n}}\|_{p^*}^{p^*}\to0$ as $q_{2,n}\to p^*$ up to a subsequence by the H\"older inequality, as in excluding the case~$(1)$.  

We finally exclude the case~$(3)$ by taking $u_n=w_{k}^{(n)}=u_{q_{2,n}}$, where $u_{q_{2,n}}$ is the minimizer constructed in Proposition~\ref{pro5.2}, in the proof of Lemma~\ref{lem7.5}, by further adapting the arguments for Proposition~\ref{pro3.2}, and by noticing that $\|u_{q_{2,n}}\|_{q_{2,n}}^{q_{2,n}}-\|u_{q_{2,n}}\|_{p^*}^{p^*}\to0$ as $q_{2,n}\to p^*$ up to a subsequence by the H\"older inequality, as in excluding the case~$(1)$.  Thus, we must have $\mu_{a,-,0}^{**}>\mu_a^*$, which completes the proof.
\end{proof}

With Lemma~\ref{lem7.7} in hands, we can obtain the following result.
\begin{proposition}\label{pro5.3}
Let $N\geq 2$, $1<p<N$, $p<q_1<p+\frac{p^2}{N}<q_2=p^*$ and $a>0$.  Then the variational problem~\eqref{e5.55}
is achieved by some
$u_{a,\mu,-}\in \mathcal{P}_{a,\mu,rad}^{-}$ for $\mu_{a}^*<\mu<\mu_{a,-}^{**}$, which is real valued, positive, radially symmetric and radially decreasing.
Moreover, $u_{a,\mu,-}$ also satisfies the equation (\ref{e1.1})
for a suitable Lagrange multiplier $\lambda_{a,\mu,-}<0$.
\end{proposition}
\begin{proof}
By Lemma~\ref{lem7.7}, we can assume that $u_{a,\mu,q_2,-}$ is a mountain-pass solution constructed in Proposition~\ref{pro5.2} for $\mu_a^*<\mu<\mu_{a,-}^{**}$ as $q_2\to p^*$.  The rest of the proof is the same as that of Proposition~\ref{pro3.2}, so we omit it here.
\end{proof}

We close this section by the proofs of Theorems~\ref{thm1.1} and \ref{thm1.2}.

\vskip0.12in

\noindent\textbf{Proof of Theorem~\ref{thm1.1}:} \quad
It follows immediately from Propositions~\ref{pro3.1}, \ref{pro3.21}, \ref{pro4.5}, \ref{pro4.7}, \ref{pro5.1} and \ref{pro5.2}.
\hfill$\Box$

\noindent\textbf{Proof of Theorem~\ref{thm1.2}:} \quad
It follows immediately from Propositions~\ref{pro3.1}, \ref{pro3.2}, \ref{pro4.5}, \ref{pro4.7}, \ref{pro5.1} and \ref{pro5.3}.
\hfill$\Box$

\section{The computation of Morse index}
As in the introduction, we denote $u_{\pm}=u_{a,\mu,\pm}$, $\lambda_{\pm}=\lambda_{a,\mu,\pm}$ and $\mathcal{P}^{\pm}=\mathcal{P}_{a,\mu}^{\pm}$ for the sake of simplicity in this section.  We also recall that
\begin{eqnarray*}
\|v\|_{\mathbb{H}_{\pm}}=\left(\int_{\mathbb{R}^N}\left|u'_{\pm}\right|^{p-2}|\nabla v|^2+(p-2)\left|u'_{\pm}\right|^{p-2}(\partial_rv)^2-\lambda_{\pm}(p-1)\left| u_{\pm}\right|^{p-2}| v|^2dx\right)^{\frac{1}{2}}
\end{eqnarray*}
and
\begin{eqnarray*}
\left\langle w, v\right\rangle_{\mathbb{H}_{\pm}}=\int_{\mathbb{R}^N}\left|u'_{\pm}\right|^{p-2}\nabla w\nabla v+(p-2)\left|u'_{\pm}\right|^{p-2}\partial_rw\partial_rv-\lambda_{\pm}(p-1)\left| u_{\pm}\right|^{p-2}wvdx,
\end{eqnarray*}
where $r=|x|$.
Let
\begin{eqnarray*}
P(u)=\Phi_{\mu, u}'\left(1\right)=\|\nabla u\|_p^p-\mu \gamma_{q_1}\|u\|_{q_1}^{q_1}-\gamma_{q_2}\|u\|_{q_2}^{q_2}.
\end{eqnarray*}
Then we rewrite
\begin{eqnarray*}
\mathcal{P}=\left\{u\in W^{1,p}\left(\mathbb{R}^N\right) \mid P(u)=0\quad\text{and}\quad\|u\|_p^p=a^p\right\}
\end{eqnarray*}
and
\begin{eqnarray*}
\mathcal{P}_{rad}=\left\{u\in W_{rad}^{1,p}\left(\mathbb{R}^N\right) \mid P(u)=0\quad\text{and}\quad\|u\|_p^p=a^p\right\}.
\end{eqnarray*}
The tangent space of $\mathcal{P}$ at $u_{\pm}$ in $\mathbb{H}_{\pm}$ and in $\mathbb{H}_{\pm,rad}$ are given by
\begin{eqnarray*}
\mathcal{T}_{u_{\pm}}\mathcal{P}=\left\{v\in \mathbb{H}_{\pm}\cap \mathcal{S}_a \mid \left\langle u_{\pm}, v\right\rangle_{\mathbb{H}_{\pm}}-\left\langle u_{\pm}, v\right\rangle_{L^2; \pmb{w}_{\pm,*}}=\left\langle u_{\pm}, v\right\rangle_{L^2; u_{\pm}^{p-2}}=0\right\}
\end{eqnarray*}
and
\begin{eqnarray*}
\left(\mathcal{T}_{u_{\pm}}\mathcal{P}\right)_{rad}=\left\{v\in \mathbb{H}_{\pm,rad} \cap \mathcal{S}_a\mid \left\langle u_{\pm}, v\right\rangle_{\mathbb{H}_{\pm,rad}}-\left\langle u_{\pm}, v\right\rangle_{L^2; \pmb{w}_{\pm,*},rad}=\left\langle u_{\pm}, v\right\rangle_{L^2; u_{\pm}^{p-2}, rad}=0\right\},
\end{eqnarray*}
respectively, where $\pmb{w}_{\pm,*}=\frac{p-1}{p}\left(\mu q_1\gamma_{q_1}u_{\pm}^{q_1-2}+q_2\gamma_{q_2}u_{\pm}^{q_2-2}\right)$.
\begin{lemma}\label{Morse1}
Let $N\geq 3$ and $2\leq p<N$.  Suppose that $u_{\pm}$ are the solutions constructed in Theorems~\ref{thm1.1} and \ref{thm1.2}, then $\mathcal{T}_{u_{\pm}}\mathcal{P}$ and $\left(\mathcal{T}_{u_{\pm}}\mathcal{P}\right)_{rad}$ are of codimension $2$ in $\mathbb{H}_{\pm}$ and $\mathbb{H}_{\pm,rad}$, respectively.
\end{lemma}
\begin{proof}
Since by \cite[Theorem~6.1]{Figalli-Neumayer} and \eqref{decay}, $\mathbb{H}_{\pm}$ compactly embeds into $L^2\left(\mathbb{R}^N; \pmb{w}_{\pm,*}\right)$, 
by the Reisz representation in Hilbert spaces $\mathbb{H}_{\pm}$ and $\mathbb{H}_{\pm,rad}$, we can denote
\begin{eqnarray}\label{Wu01}
\left\langle \pmb{u}_{\pm}, v\right\rangle_{\mathbb{H}_{\pm}}=\left\langle u_{\pm}, v\right\rangle_{\mathbb{H}_{\pm}}-\left\langle u_{\pm}, v\right\rangle_{L^2; \pmb{w}_{\pm,*}}
\end{eqnarray}
for all $v\in \mathbb{H}_{\pm}$ with $0<\mu\leq \mu_a^*$ and
\begin{eqnarray}\label{Wu02}
\left\langle \pmb{u}_{\pm,rad}, v\right\rangle_{\mathbb{H}_{\pm}}=\left\langle u_{\pm}, v\right\rangle_{\mathbb{H}_{\pm,rad}}-\left\langle u_{\pm}, v\right\rangle_{L^2; \pmb{w}_{\pm,*},rad}
\end{eqnarray}
for all $v\in \mathbb{H}_{\pm,rad}$ with $\mu>\mu_a^*$.  Let 
\begin{eqnarray*}
\mathcal{H}_{\pm}=\text{span}\{u_{\pm}, \pmb{u}_{\pm}\}\quad\text{and}\quad \mathcal{H}_{\pm,rad}=\text{span}\{u_{\pm,rad}, \pmb{u}_{\pm,rad}\}.
\end{eqnarray*}
We claim that dim$(\mathcal{H}_{\pm})=2$ and dim$(\mathcal{H}_{\pm,rad})=2$.
Suppose the contrary that dim$(\mathcal{H}_{\pm})=1$ and dim$(\mathcal{H}_{\pm,rad})=1$, which implies that there exists $\widetilde{\lambda}_{\pm}\in\mathbb{R}$ such that $\pmb{u}_{\pm}=\widetilde{\lambda}_{\pm}u_{\pm}$ in $\mathbb{H}_{\pm}$ for $0<\mu\leq \mu_a^*$ and $\pmb{u}_{\pm,rad}=\widetilde{\lambda}_{\pm,rad}u_{\pm,rad}$ in $\mathbb{H}_{\pm,rad}$ for $\mu>\mu_a^*$.
It follows \eqref{Wu01} and \eqref{Wu02} that 
\begin{eqnarray*}
-p \Delta_p u_{\pm}=\widetilde{\lambda}_{\pm}u_{\pm}^{p-1}+\mu (q_1-1)\gamma_{q_1}u_{\pm}^{q_1-1}+(q_2-1)\gamma_{q_2}u_{\pm}^{q_2-1}
\end{eqnarray*}
in $\mathbb{R}^N$ in the weak sense for $0<\mu\leq \mu_a^*$ and
\begin{eqnarray*}
-p\Delta_p u_{\pm,rad}=\widetilde{\lambda}_{\pm,rad}u_{\pm,rad}^{p-1}+\mu q_1\gamma_{q_1}u_{\pm,rad}^{q_1-1}+q_2\gamma_{q_2}u_{\pm,rad}^{q_2-1}
\end{eqnarray*}
in $\mathbb{R}^N$ in the weak sense for $\mu>\mu_a^*$.  Note that $u_{\pm}$ are also solutions of \eqref{e1.1} in the weak sense.  Thus, by similar arguments used in the proof of Proposition~\ref{pro4.5}, we can get a contradiction.  Thus, we must have dim$(\mathcal{H}_{\pm})=2$ and dim$(\mathcal{H}_{\pm,rad})=2$.  Since $u_{\pm}\in \mathcal{P}^{\pm}$ are the solutions of \eqref{e1.1} and $p<q_1<q_2$, we also have that $u_{\pm}\not\in\mathcal{T}_{u_{\pm}}\mathcal{P}$ for $0<\mu\leq \mu_a^*$ and $u_{\pm}\not\in\left(\mathcal{T}_{u_{\pm}}\mathcal{P}\right)_{rad}$ for $\mu>\mu_a^*$.
Thus, by dim$(\mathcal{H}_{\pm})=2$ and dim$(\mathcal{H}_{\pm,rad})=2$, $\mathcal{T}_{u_{\pm}}\mathcal{P}$ and $\left(\mathcal{T}_{u_{\pm}}\mathcal{P}\right)_{rad}$ are of codimension $2$ in $\mathbb{H}_{\pm}$ and $\mathbb{H}_{\pm,rad}$, respectively, which completes the proof.
\end{proof}

With Lemma~\ref{Morse1} in hands, we can finally prove Theorem~\ref{thm1.3}.

\vskip0.12in

\noindent\textbf{Proof of Theorem~\ref{thm1.3}:} \quad We recall that the linear operator
\begin{eqnarray*}
\mathcal{L}_{\pm}(w)=-\text{div}\left(\left|u'_{\pm}\right|^{p-2}\nabla w\right)-\partial_r\left((p-2)\left|u'_{\pm}\right|^{p-2}\partial_rw\right)-\lambda_{\pm}(p-1)\left| u_{\pm}\right|^{p-2}w
\end{eqnarray*}
has a discrete spectrum $\{\sigma_m\}$ in the weighted Lebesgue space $L^2\left(\mathbb{R}^N; \pmb{w}_{\pm}\right)$.  We denote $\{\sigma_{m,rad}\}$ the discrete spectrum of $\mathcal{L}_{\pm}(w)$ in the radial setting.  First of all, since $u_{\pm}$ are radial solutions of \eqref{e1.1},
\begin{eqnarray}\label{RepWuNew1020}
\left\langle\mathcal{L}_{\pm}(u_{\pm}), u_{\pm}\right\rangle-\left\| u_{\pm}\right\|^2_{L^2; \pmb{w}_{\pm}}&=&\left\| u_{\pm}\right\|^2_{\mathbb{H}_{\pm}}-\left\| u_{\pm}\right\|^2_{L^2; \pmb{w}_{\pm}}\notag\\
&=&-(q_1-p)\mu\|u_{a,\mu,\pm}\|_{q_1}^{q_1}-(q_2-p)\|u_{a,\mu,\pm}\|_{q_2}^{q_2}\notag\\
&<&0
\end{eqnarray}
for $0<\mu\leq \mu_a^*$ and
\begin{eqnarray}\label{RepWuNew10201}
\left\langle\mathcal{L}_{\pm}(u_{\pm}), u_{\pm}\right\rangle-\left\| u_{\pm}\right\|^2_{L^2; \pmb{w}_{\pm}, rad}&=&\left\| u_{\pm}\right\|^2_{\mathbb{H}_{\pm,rad}}-\left\| u_{\pm}\right\|^2_{L^2; \pmb{w}_{\pm}, rad}\notag\\
&=&-(q_1-p)\mu\|u_{\pm}\|_{q_1}^{q_1}-(q_2-p)\|u_{\pm}\|_{q_2}^{q_2}\notag\\
&<&0
\end{eqnarray}
for $\mu>\mu_a^*$.  Clearly, by \eqref{RepWuNew1020} and \eqref{RepWuNew10201}, $u_{\pm}\in\mathcal{H}_{\pm}$ for $0<\mu\leq \mu_a^*$ and $u_{\pm}\in\mathcal{H}_{\pm,rad}$ for $\mu>\mu_a^*$.  Next, we introduce the function
\begin{eqnarray*}
\phi_{\pm}=\frac{N}{2}u_{\pm}+\nabla u_{\pm}\cdot x=\frac{\partial (u_{\pm})_{t}}{\partial t}|_{t=1}.
\end{eqnarray*}
Clearly, $\phi_{\pm}$ are radial since $u_{\pm}$ are radial.  Moreover, by \eqref{decay}, we also know that
\begin{eqnarray*}
\left\{\aligned
&\langle u_{\pm}, \phi_{\pm}\rangle_{L^2; u_{\pm}^{p-2}}=0,\quad\text{for }0<\mu\leq \mu_a^*,\\
&\langle u_{\pm}, \phi_{\pm}\rangle_{L^2; u_{\pm}^{p-2}, rad}=0,\quad\text{for }\mu>\mu_a^*,
\endaligned\right.
\end{eqnarray*}
according to $\|(u_{\pm})_{t}\|_p^p=a^p$ for all $t>0$.  Since $u_{\pm}\in \mathcal{P}^{\pm}$, by Proposition~\ref{pro2.1}, we must have
\begin{eqnarray*}
\frac{d^2\Phi_{\mu, u_{+}}}{d t^2}(1)>0\quad\text{and}\quad \frac{d^2\Phi_{\mu, u_{-}}}{d t^2}(1)<0,
\end{eqnarray*}
which is equivalent to
\begin{eqnarray}\label{Wu03}
\left\{\aligned
0<\Psi_{\mu}''((u_{+})_t)|_{t=1}\left(\frac{\partial (u_{+})_{t}}{\partial t}|_{t=1}\right)^2+\Psi_{\mu}'((u_{+})_t)|_{t=1}\frac{\partial^2 (u_{+})_{t}}{\partial t^2}|_{t=1},\\
0>\Psi_{\mu}''((u_{-})_t)|_{t=1}\left(\frac{\partial (u_{-})_{t}}{\partial t}|_{t=1}\right)^2+\Psi_{\mu}'((u_{-})_t)|_{t=1}\frac{\partial^2 (u_{-})_{t}}{\partial t^2}|_{t=1},
\endaligned\right.
\end{eqnarray}
where $\Phi_{\mu, u^\pm}(s)$ are the fibering maps given by \eqref{e1.9}.
By $\|(u_{\pm})_{t}\|_p^p=a^p$ for all $t>0$, we have
\begin{eqnarray*}
0&=&p\left\langle \left((u_{\pm})_t\right)^{p-1}, \frac{\partial_t\left(u_{\pm}\right)_t}{\partial t}\right\rangle\\
&=&p(p-1)\left\langle \left((u_{\pm})_t\right)^{p-2}, \left(\frac{\partial_t\left(u_{\pm}\right)_t}{\partial t}\right)^2\right\rangle+p\left\langle \left((u_{\pm})_t\right)^{p-1}, \frac{\partial^2_t\left(u_{\pm}\right)_t}{\partial t^2}\right\rangle
\end{eqnarray*}
for all $t>0$.  Since $u_{\pm}$ are radial solutions of \eqref{e1.1}, by \eqref{Wu03}, we have
\begin{eqnarray*}\label{Wu04}
\left\{\aligned
&0>\left\langle\Psi_\mu''\left(u_{-}\right), \phi_{-}^2\right\rangle+\lambda_{-}\left\langle u_{-}^{p-1}, \frac{\partial^2 (u_{-})_{t}}{\partial t^2}|_{t=1}\right\rangle=\left\|\phi_{-}\right\|^2_{\mathbb{H}_{-}}-\left\|\phi_{-}\right\|^2_{L^2; \pmb{w}_{-}},\\
&0<\left\langle\Psi_\mu''\left(u_{+}\right), \phi_{+}^2\right\rangle+\lambda_{a,\mu,+}\left\langle u_{+}^{p-1}, \frac{\partial^2 (u_{+})_{t}}{\partial t^2}|_{t=1}\right\rangle
=\left\|\phi_{+}\right\|^2_{\mathbb{H}_{+}}-\left\|\phi_{+}\right\|^2_{L^2; \pmb{w}_{+}}
\endaligned\right.
\end{eqnarray*}
for $0<\mu\leq \mu_a^*$ and
\begin{eqnarray*}\label{Wu05}
\left\{\aligned
&0>\left\langle\Psi_\mu''\left(u_{-}\right), \phi_{-}^2\right\rangle+\lambda_{-}\left\langle u_{-}^{p-1}, \frac{\partial^2 (u_{-})_{t}}{\partial t^2}|_{t=1}\right\rangle=\left\|\phi_{-}\right\|^2_{\mathbb{H}_{-,rad}}-\left\|\phi_{-}\right\|^2_{L^2; \pmb{w}_{-}, rad},\\
&0<\left\langle\Psi_\mu''\left(u_{+}\right), \phi_{+}^2\right\rangle+\lambda_{a,\mu,+}\left\langle u_{+}^{p-1}, \frac{\partial^2 (u_{+})_{t}}{\partial t^2}|_{t=1}\right\rangle
=\left\|\phi_{+}\right\|^2_{\mathbb{H}_{+,rad}}-\left\|\phi_{+}\right\|^2_{L^2; \pmb{w}_{+}, rad}
\endaligned\right.
\end{eqnarray*}
for $\mu>\mu_a^*$, where $\left\langle\cdot, \cdot\right\rangle$ is the dual product.  It follows that $\phi_{\pm}\in\mathcal{H}_{\pm}$ for $0<\mu\leq \mu_a^*$ and $\phi_{\pm}\in\mathcal{H}_{\pm,rad}$ for $\mu>\mu_a^*$.  Finally,
for every $\varphi\in \mathcal{T}_{u_{\pm}}\mathcal{P}$ with $0<\mu\leq \mu_a^*$ and $\varphi\in \left(\mathcal{T}_{u_{\pm}}\mathcal{P}\right)_{rad}$ with $\mu> \mu_a^*$, since $p\geq2$, by the implicit function theorem, there exists $s_{\pm}(\epsilon)$ of class $C^2$ for $|\epsilon|$ sufficiently small such that
\begin{eqnarray*}
a^p=\left\|(1+s(\epsilon))u_{\pm}+\epsilon\varphi\right\|_p^p\quad\text{with}\quad\left\{\aligned
&s_{\pm}'(0)=0,\\
&s''(0)=-\frac{p-1}{2a^p}\left\|\varphi\right\|^2_{L^2; u_{\pm}^{p-2}}.
\endaligned\right.
\end{eqnarray*}
We denote $v_{\pm,\epsilon}=(1+s(\epsilon))u_{\pm}+\epsilon\varphi$ for the sake of simplicity.  Since $u_{\pm}\in \mathcal{P}^{\pm}$, by Proposition~\ref{pro2.1}, we have $\mu(u_{\pm})>\mu_a^*$.  Thus, by the continuity of $s(\epsilon)$, we know that $\mu(v_{\pm, \epsilon})>\mu_a^*$ for $|\epsilon|$ small.  It follows that there exists $t^{\pm}\left(v_{\pm, \epsilon}\right)$ such that $w_{\epsilon}^{\pm}:=\left(v_{\pm, \epsilon}\right)_{t^{\pm}\left(v_{\pm, \epsilon}\right)}\in \mathcal{P}^{\pm}$, which together with $\Psi_{\mu}(u_{\pm})=m^{\pm}(a,\mu)$ for  $0<\mu\leq\mu_a^*$ and $\Psi_{\mu}(u_{\pm})=m_{rad}^{\pm}(a,\mu)$ for $\mu>\mu_a^*$ by Theorems~\ref{thm1.1} and \ref{thm1.2}, implies that $g_{\pm}(\epsilon)\geq g_{\pm}(0)$ where
\begin{eqnarray*}
g_{\pm}(\epsilon)&=&\Psi_\mu\left(w_{\epsilon}^{\pm}\right)\\
&=&\Phi_{\mu, v_{\pm,\epsilon}^{\pm}}\left(t^{\pm}\left(v_{\pm,\epsilon}\right)\right)\\
&=&\frac{\left(t^{\pm}\left(v_{\pm,\epsilon}\right)\right)^p}{p}\|\nabla v_{\pm,\epsilon}\|_2^2-\frac{\mu \left(t^{\pm}\left(v_{\pm,\epsilon}\right)\right)^{q_1\gamma_{q_1}}}{q}\| v_{\pm,\epsilon}\|_{q_1}^{q_1}-\frac{\left(t^{\pm}\left(v_{\pm,\epsilon}\right)\right)^{q_2\gamma_{q_2}}}{q_2}\| v_{\pm,\epsilon}\|_{q_2}^{q_2}\\
&:=&\widetilde{\Psi}_{\mu}\left(t^{\pm}\left(v_{\pm,\epsilon}\right), v_{\pm,\epsilon}\right).
\end{eqnarray*}
Moreover, since $s_{\pm}(\epsilon)$ of class $C^2$ for $|\epsilon|$ sufficiently small, by the implicit function theorem, we can also show that $t_{\mu}^{\pm}\left(v_{\pm, \epsilon}\right)$ are of class $C^2$ in terms of $\epsilon$, as that in the proof of Lemma~\ref{lem3.3}.  Since $0$ is a locally minimal point of $g_{\pm}(\epsilon)$, we have $g_{\pm}''(0)\geq0$
which, together with $\|\varphi\|_p^p=a^a$, implies that
\begin{eqnarray*}
0&\leq&\frac{\partial^2\widetilde{\Psi}_{\mu}\left(t^{\pm}\left(v_{\pm,\epsilon}\right), v_{\pm,\epsilon}\right)}{\partial t^{\pm}\left(v_{\pm,\epsilon}\right)^2}|_{\epsilon=0}\left(\frac{\partial\left(t^{\pm}\left(v_{\pm,\epsilon}\right)\right)}{\partial\epsilon}|_{\epsilon=0}\right)^2+\frac{\partial\widetilde{\Psi}_{\mu}\left(t^{\pm}\left(v_{\pm,\epsilon}\right), v_{\pm,\epsilon}\right)}{\partial t^{\pm}\left(v_{\pm,\epsilon}\right)}|_{\epsilon=0}\frac{\partial^2\left(t^{\pm}\left(v_{\pm,\epsilon}\right)\right)}{\partial\epsilon^2}|_{\epsilon=0}\\
&&+\frac{\partial^2\widetilde{\Psi}_{\mu}\left(t^{\pm}\left(v_{\pm,\epsilon}\right), v_{\pm,\epsilon}\right)}{\partial v_{\pm,\epsilon}^2}|_{\epsilon=0}\left(\frac{\partial v_{\pm,\epsilon}}{\partial\epsilon}|_{\epsilon=0}\right)^2+\frac{\partial\widetilde{\Psi}_{\mu}\left(t^{\pm}\left(v_{\pm,\epsilon}\right), v_{\pm,\epsilon}\right)}{\partial v_{\pm,\epsilon}}|_{\epsilon=0}\frac{\partial^2 v_{\pm,\epsilon}}{\partial\epsilon^2}|_{\epsilon=0}\\
&=&\Phi_{\mu, v_{\pm, \epsilon}}''\left(1\right)\left(\frac{\partial\left(t^{\pm}\left(v_{\pm, \epsilon}\right)\right)}{\partial\epsilon}|_{\epsilon=0}\right)^2+\Phi_{\mu, v_{\pm, \epsilon}}'\left(1\right)\frac{\partial^2\left(t^{\pm}\left(v_{\pm, \epsilon}\right)\right)}{\partial\epsilon^2}|_{\epsilon=0}\\
&&+\left\|\varphi\right\|^2_{\mathbb{H}_{\pm}}-\left\|\varphi\right\|^2_{L^2; \pmb{w}_{\pm}}
\end{eqnarray*}
for $0<\mu\leq\mu_a^*$ and similarly,
\begin{eqnarray*}
0&\leq&\Psi_{\mu, v_{\pm, \epsilon}}''\left(1\right)\left(\frac{\partial\left(t^{\pm}\left(v_{\pm, \epsilon}\right)\right)}{\partial\epsilon}|_{\epsilon=0}\right)^2+\Psi_{\mu, v_{\pm, \epsilon}}'\left(1\right)\frac{\partial^2\left(t^{\pm}\left(v_{\pm, \epsilon}\right)\right)}{\partial\epsilon^2}|_{\epsilon=0}\\
&&+\left\|\varphi\right\|^2_{\mathbb{H}_{\pm,rad}}-\left\|\varphi\right\|^2_{L^2; \pmb{w}_{\pm},rad}
\end{eqnarray*}
for $\mu>\mu_a^*$.
By $w_{\epsilon}^{\pm}\in \mathcal{P}^{\pm}$ and $\varphi\in\mathcal{T}_{u_{\pm}}\mathcal{P}$ for $0<\mu\leq\mu_a^*$ and $w_{\epsilon}^{\pm}\in \mathcal{P}_{rad}^{\pm}$ and $\varphi\in\left(\mathcal{T}_{u_{\pm}}\mathcal{P}\right)_{rad}$ for $\mu>\mu_a^*$, we have
\begin{eqnarray*}
\frac{\partial_\epsilon\left(t^{\pm}\left(v_{\pm,\epsilon}\right)\right)}{\partial\epsilon}|_{\epsilon=0}
=-\frac{\left\langle u_{\pm}, \varphi\right\rangle_{\mathbb{H}_{\pm}}-\left\langle u_{\pm}, \varphi\right\rangle_{L^2; \pmb{w}_{\pm}}+(p-1)\lambda_{\pm}\left\langle u_{\pm}, \varphi\right\rangle_{L^2; u_{\pm}^{p-2}}}{\left\|u_{\pm}\right\|^2_{\mathbb{H}_{\pm}}-\left\|u_{\pm}\right\|^2_{L^2; \pmb{w}_{\pm}}+(p-1)\lambda_{\pm}\left\|u_{\pm}\right\|^2_{L^2; u_{\pm}^{p-2}}}=0
\end{eqnarray*}
for $0<\mu\leq\mu_a^*$ and
\begin{eqnarray*}
\frac{\partial_\epsilon\left(t^{\pm}\left(v_{\pm,\epsilon}\right)\right)}{\partial\epsilon}|_{\epsilon=0}
=-\frac{\left\langle u_{\pm}, \varphi\right\rangle_{\mathbb{H}_{\pm,rad}}-\left\langle u_{\pm}, \varphi\right\rangle_{L^2; \pmb{w}_{\pm},rad}+(p-1)\lambda_{\pm}\left\langle u_{\pm}, \varphi\right\rangle_{L^2; u_{\pm}^{p-2},rad}}{\left\|u_{\pm}\right\|^2_{\mathbb{H}_{\pm,rad}}-\left\|u_{\pm}\right\|^2_{L^2; \pmb{w}_{\pm},rad}+(p-1)\lambda_{\pm}\left\|u_{\pm}\right\|^2_{L^2; u_{\pm}^{p-2}, rad}}=0
\end{eqnarray*}
for $\mu>\mu_a^*$.  It follows that
\begin{eqnarray}\label{Wu06}
\left\|\varphi\right\|^2_{\mathbb{H}_{\pm}}-\left\|\varphi\right\|^2_{L^2; \pmb{w}_{\pm}}\geq0
\end{eqnarray}
for every $\varphi\in \mathcal{T}_{u_{\pm}}\mathcal{P}$ with $0<\mu\leq \mu_a^*$ and
\begin{eqnarray}\label{Wu07}
\left\|\varphi\right\|^2_{\mathbb{H}_{\pm,rad}}-\left\|\varphi\right\|^2_{L^2; \pmb{w}_{\pm},rad}\geq0
\end{eqnarray}
for every $\varphi\in \left(\mathcal{T}_{u_{\pm}}\mathcal{P}\right)_{rad}$ with $\mu>\mu_a^*$.  The conclusions then follows immediately from Lemma~\ref{Morse1} and \eqref{RepWuNew1020}, \eqref{RepWuNew10201} and \eqref{Wu03}.
\hfill$\Box$

\bigskip

\section{Acknowledgments}
The research is supported by National Natural Science Foundation of China (No. 12171470, No. 12161093, No. 12001403), National Key R\&D Program of China
(No. 2022YFA1005602), Hong Kong General Research Fund ``On critical and supercritical Fujita equation'', Yunnan Key Laboratory of Modern Analytical Mathematics and Applications (No. 202302AN360007) and Yunnan Provincial Innovation Team on the Interdisciplinary Integration of Modern Applied Mathematics and Life Sciences (No. 202405AS350003).

\bigskip

\section{Apendix~A: Some useful results}
In this section, we list some useful lemmas, which are frequently used in this paper.  We begin with the well known Brez\'is-Lieb lemma.
\begin{lemma}\label{lem1.1}(\cite[Theorem~1]{Brezis-Lieb 1983})
Let $\{u_n\}$ be a bounded function sequence in $L^q\left(\mathbb{R}^N\right)$ for some $q\in (0,+\infty)$. If $u_n\to u$ a.e. in $\mathbb{R}^N$, then
\begin{equation*}
\lim_{n\to +\infty}(\|u_n\|_q^q-\|u_n-u\|_q^q)=\|u\|_q^q.
\end{equation*}
\end{lemma}

The following estimate is also well known and useful in considering the $p$-Laplacian equation.
\begin{lemma}\label{lem1.2}
(\cite[(2.2)]{Simon1978})  There exists a constant $C(q)>0$ such that for all $x, y\in \mathbb{R}^
N$ with $|x|+|y|\neq 0$,
\begin{equation*}
\langle|x|^{q-2}x-|y|^{q-2}y,x-y\rangle\geq C(q)
\left\{\begin{array}{ll}
\frac{|x-y|^2}{(|x|+|y|)^{2-q}},&\ \text{if}\ 1\leq q<2,\\
|x-y|^{q},&\ \text{if}\ q\geq 2.\\
\end{array}
\right.
\end{equation*}
\end{lemma}

We next move to a convergent lemma, which is essential obtained by Boccardo and Murat \cite[Theorem 2.1]{Boccardo-Murat1992}.
\begin{lemma}\label{lem1.3}
Let $N\geq 2$, $1<p<N$, $p<q_1<q_2\leq p^*$ and $a, \mu>0$. Let $\{u_n\}\subset \mathcal{S}_a$ be a bounded $(PS)_{c}$ sequence of $\Psi_\mu(u)|_{\mathcal{S}_a}$, where $\Psi_\mu(u)|_{\mathcal{S}_a}$ is the restriction of $\Psi_\mu(u)$ on $\mathcal{S}_a$ with $\Psi_\mu(u)$ and $\mathcal{S}_a$ given by \eqref{e1.2} and  \eqref{e1.4}, respectively.  Then, up to a subsequence, $\nabla u_n\to \nabla u$ in $\mathbb{R}^N$
for some $u\in W^{1,p}\left(\mathbb{R}^N\right)$ as $n\to+\infty$.
\end{lemma}
\begin{proof}
Since $\{u_n\}\subset \mathcal{S}_a$ is a bounded $(PS)_{c}$ sequence of $\Psi_\mu(u)|_{\mathcal{S}_a}$ and $\mathcal{S}_a$ is nondegenerate and smooth,
by the method of Lagrange multiplier, there exists $\lambda_n\subset \mathbb{R}$ such that for any $v\in W^{1,p}\left(\mathbb{R}^N\right)$,
\begin{eqnarray}\label{e1.13}
o_n(1)&=&\langle\Psi'_{\mu}(u_n),v\rangle+\lambda_n\int_{\mathbb{R}^N}|u_n|^{p-2}u_nvdx\notag\\
      &=&\int_{\mathbb{R}^N}\left(|\nabla u_n|^{p-2}\nabla u_n\nabla v-\mu|u_n|^{q_1-2}u_n v-|u_n|^{q_2-2}u_n v\right)dx\notag\\
      &&+\lambda_n\int_{\mathbb{R}^N}|u_n|^{p-2}u_nvdx.
\end{eqnarray}
Thus, the boundedness of $\{u_n\}$ in $W^{1,p}\left(\mathbb{R}^N\right)$ implies that $\lambda_n$ is bounded. Without loss of generality, we may assume that $u_n\to u$ weakly in $W^{1,p}\left(\mathbb{R}^N\right)$, a.e. in  $\mathbb{R}^N$ and strongly in $L^q_{loc}\left(\mathbb{R}^N\right)$ for all $q\in [1,p^*)$ and some $u\in W^{1,p}\left(\mathbb{R}^N\right)$.  Moreover, we may also assume that $\lambda_n\to \lambda$ as $n\to +\infty$ for some $\lambda\in \mathbb{R}$. Then, for any $v\in W^{1,p}\left(\mathbb{R}^N\right)$,
\begin{eqnarray}\label{e1.14}
\langle\Psi'_{\mu}(u),v\rangle+\lambda\int_{\mathbb{R}^N}|u|^{p-2}uvdx
&=&\lim_{n\to+\infty}\left(\langle\Psi'_{\mu}(u_n),v\rangle+\lambda_n\int_{\mathbb{R}^N}|u_n|^{p-2}u_nvdx\right)\notag\\
&=&0.
\end{eqnarray}
For $k>0$, let $\tau_k:\mathbb{R}\to \mathbb{R}$ be the truncation given by
\begin{equation*}
\tau_k(s):=
\left\{\begin{array}{ll}
s,&\ \text{if}\ |s|\leq k,\\
ks/|s|,&\ \text{if}\ |s|>k.\\
\end{array}
\right.
\end{equation*}
Then it is easy to see that $\{\tau_k(u_n-u)\}_n$ is bounded in $W^{1,p}\left(\mathbb{R}^N\right)$. Let $\psi\in C_0^\infty\left(\mathbb{R}^N\right)$ be a cut-off function such that $0\leq \psi\leq 1$ in $\mathbb{R}^N$, $\psi(x)=1$ for $x\in B_1(0)$ and $\psi(x) = 0$ for $x\in \mathbb{R}^N \backslash B_2(0)$.  Now, take $R > 1$ and define $\psi_R(x):= \psi(x/R)$
for $x\in \mathbb{R}^N$. By (\ref{e1.13}) and (\ref{e1.14}),
\begin{eqnarray}\label{e1.15}
\begin{split}
o_n(1)&=\langle\Psi'_{\mu}(u_n),\tau_k(u_n-u)\psi_R\rangle+\lambda_n\int_{\mathbb{R}^N}|u_n|^{p-2}u_n\tau_k(u_n-u)\psi_Rdx\\
      &=\langle\Psi'_{\mu}(u_n)-\Psi'_{\mu}(u),\tau_k(u_n-u)\psi_R\rangle+\lambda_n\int_{\mathbb{R}^N}|u_n|^{p-2}u_n\tau_k(u_n-u)\psi_Rdx\\
      &\quad-\lambda\int_{\mathbb{R}^N}|u|^{p-2}u\tau_k(u_n-u)\psi_Rdx\\
      &=\int_{\mathbb{R}^N}\left(|\nabla u_n|^{p-2}\nabla u_n-|\nabla u|^{p-2}\nabla u\right)\nabla (\tau_k(u_n-u)\psi_R)dx\\
      &\quad +\lambda_n\int_{\mathbb{R}^N}|u_n|^{p-2}u_n\tau_k(u_n-u)\psi_Rdx-\lambda\int_{\mathbb{R}^N}|u|^{p-2}u\tau_k(u_n-u)\psi_Rdx\\
      &\quad-\mu\int_{\mathbb{R}^N}\left(|u_n|^{q_1-2}u_n-|u|^{q_1-2}u\right)\tau_k(u_n-u)\psi_Rdx\\
      &\quad-\int_{\mathbb{R}^N}\left(|u_n|^{q_2-2}u_n-|u|^{q_2-2}u\right)\tau_k(u_n-u)\psi_Rdx.
\end{split}
\end{eqnarray}
Since $|\tau_k(u_n-u)|\leq |u_n-u|$ and $u_n\to u$ strongly in $L^q_{loc}\left(\mathbb{R}^N\right)$ for $q\in [1,p^*)$ as $n\to+\infty$, by the H\"{o}lder inequality,
\begin{eqnarray}\label{e1.171}
\begin{split}
&\left|\int_{\mathbb{R}^N}\left(|u_n|^{q_1-2}u_n-|u|^{q_1-2}u\right)\tau_k(u_n-u)\psi_Rdx\right|\\
\leq &\int_{B_{2R}(0)}\left||u_n|^{q_1-2}u_n-|u|^{q_1-2}u\right||u_n-u|dx\\
=&o_n(1)
\end{split}
\end{eqnarray}
and
\begin{eqnarray}\label{e1.172}
\left|\int_{\mathbb{R}^N}|u_n|^{p-2}u_n\tau_k(u_n-u)\psi_Rdx\right|+\left|\int_{\mathbb{R}^N}|u|^{p-2}u\tau_k(u_n-u)\psi_Rdx\right|=o_n(1).
\end{eqnarray}
Since we also have $|\tau_k(u_n-u)|\leq k$, by the boundedness of $\{u_n\}$ in $W^{1,p}\left(\mathbb{R}^N\right)$ and the H\"{o}lder inequality,
\begin{equation}\label{e1.18}
\begin{split}
\left|\int_{\mathbb{R}^N}\left(|u_n|^{q_2-2}u_n-|u|^{q_2-2}u\right)\tau_k(u_n-u)\psi_Rdx\right|
&\leq k\int_{\mathbb{R}^N}\left||u_n|^{q_2-2}u_n-|u|^{q_2-2}u\right|dx\\
   &\lesssim  k.
\end{split}
\end{equation}
Define
\begin{equation*}
e_{k,n}(x):=\left(|\nabla u_n|^{p-2}\nabla u_n-|\nabla u|^{p-2}\nabla u\right)\nabla (\tau_k(u_n-u)\psi_R).
\end{equation*}
Then by (\ref{e1.15}), (\ref{e1.171}), (\ref{e1.172}) and (\ref{e1.18}), we have
\begin{equation}\label{e1.20}
\int_{\mathbb{R}^N}e_{k,n}(x)dx\lesssim k+o_n(1).
\end{equation}
Since $|\tau_k(u_n-u)|\leq |u_n-u|$, $|\nabla \psi_R|\lesssim\frac{1}{R}$ and $u_n\to u$ in $L^q_{loc}\left(\mathbb{R}^N\right)$ for $q\in [1,p^*)$, by the H\"{o}lder inequality,
\begin{equation}\label{e1.21}
\begin{split}
\left|\int_{B_{2R}(0)\backslash B_{R}(0)}|\nabla u_n|^{p-2}\nabla u_n\tau_k(u_n-u)\nabla \psi_Rdx\right|
      &\lesssim \frac{1}{R}\int_{B_{2R}(0)\backslash B_{R}(0)}|\nabla u_n|^{p-1}|u_n-u|dx\\
      &=o_n(1)
\end{split}
\end{equation}
and
\begin{equation}\label{e1.21111111}
\begin{split}
      \left|\int_{B_{2R}(0)\backslash B_{R}(0)}|\nabla u|^{p-2}\nabla u\tau_k(u_n-u)\nabla \psi_Rdx\right|=o_n(1).
\end{split}
\end{equation}
By Lemma \ref{lem1.2}, (\ref{e1.21}) and \eqref{e1.21111111},
\begin{eqnarray*}\label{e1.22}
      &&\int_{B_{2R}(0)\backslash B_{R}(0)}e_{k,n}(x)dx\\
      &=&\int_{(B_{2R}(0)\backslash B_{R}(0))\cap \{x:|u_n-u|\leq k\}}\left(|\nabla u_n|^{p-2}\nabla u_n-|\nabla u|^{p-2}\nabla u\right) (\nabla u_n-\nabla u)\psi_Rdx\\
      &&+\int_{B_{2R}(0)\backslash B_{R}(0)}\left(|\nabla u_n|^{p-2}\nabla u_n-|\nabla u|^{p-2}\nabla u\right) \tau_k(u_n-u)\nabla\psi_Rdx\\
      &\geq& o_n(1),
\end{eqnarray*}
which, together with (\ref{e1.20}), implies that
\begin{equation}\label{e1.23}
\int_{B_{R}(0)}e_{k,n}(x)dx\lesssim k+o_n(1).
\end{equation}
Define
\begin{equation*}
e_{n}(x):=\left(|\nabla u_n|^{p-2}\nabla u_n-|\nabla u|^{p-2}\nabla u\right)\nabla \left((u_n-u)\psi_R\right).
\end{equation*}
Then by the  H\"{o}lder inequality,
\begin{equation}\label{e1.24}
\int_{B_{R}(0)}e_{n}(x)dx=\int_{B_{R}(0)}\left(|\nabla u_n|^{p-2}\nabla u_n-|\nabla u|^{p-2}\nabla u\right)\nabla (u_n-u)dx\lesssim 1.
\end{equation}
We take $\theta\in(0,1)$ and split $B_{R}(0)$ into
\begin{equation*}
S_n^k:=\{x\in B_{R}(0)\mid|u_n-u|\leq k\} \ \text{and}\ G_n^k:=\{x\in B_{R}(0)\mid|u_n-u|> k\}.
\end{equation*}
By Lemma \ref{lem1.2}, $e_n(x)\geq 0$ and $e_{k,n}(x)\geq 0$ in $B_{R}(0)$. Thus, by the H\"{o}lder inequality,
\begin{eqnarray*}\label{e1.25}
\int_{B_{R}(0)}e_{n}^{\theta}(x)dx
      &=&\int_{S_n^k}e_{n}^{\theta}(x)dx+\int_{G_n^k}e_{n}^{\theta}(x)dx\\
      &\leq& \left(\int_{S_n^k}e_{n}(x)dx\right)^{\theta}|S_n^k|^{1-\theta}+\left(\int_{G_n^k}e_{n}(x)dx\right)^{\theta}|G_n^k|^{1-\theta}\\
      &=&\left(\int_{S_n^k}e_{k,n}(x)dx\right)^{\theta}|S_n^k|^{1-\theta}+\left(\int_{G_n^k}e_{n}(x)dx\right)^{\theta}|G_n^k|^{1-\theta}.
\end{eqnarray*}
Clearly, $|G_n^k|\to 0$ as $n\to +\infty$ for any fixed $k>0$.  Thus, by (\ref{e1.23}) and (\ref{e1.24}),
\begin{equation*}
\int_{B_{R}(0)}e_{n}^{\theta}(x)dx\lesssim (k+o_n(1))^{\theta}|B_R(0)|^{1-\theta}+o_n(1).
\end{equation*}
Let $k\to 0$ in the above estimate, we obtain that $e_n^\theta\to 0$ strongly in  $L^1(B_R(0))$ as $n\to +\infty$. By Lemma \ref{lem1.2}, we have $\nabla u_n\to \nabla u$ a.e. in $B_R(0)$. Since $R>0$ is arbitrary, we have $\nabla u_n\to \nabla u$ a.e. in $\mathbb{R}^N$ by further applying a diagonal argument.
\end{proof}

We close this section by introducing the Pohozaev identity for $p$-Laplacian in $\mathbb{R}^N$, which was obtained in \cite{Berestycki-Lions 83} for the semilinear case $p=2$ and \cite{Jeanjean-Squassina09} (see also \cite{Zhang-Zhang}) for the quasilinear case $1<p<N$.
\begin{lemma}\label{lem1.4}
Assume that $N\geq 2$, $1<p<N$, $f\in C(\mathbb{R},\mathbb{R})$ such that $f(0)=0$.  Let
$F(t)=\int_0^t f(s)ds$ and $u$ be a local weak solution of
\begin{equation*}
-\Delta_p u=f(u)\ \text{in}\ \mathcal{D}'\left(\mathbb{R}^N\right),
\end{equation*}
where $\mathcal{D}\left(\mathbb{R}^N\right)=C_{0}^{\infty}\left(\mathbb{R}^N\right)$ and $\mathcal{D}'\left(\mathbb{R}^N\right)$ is its dual space. If
\begin{equation*}
u\in L^{\infty}_{loc}\left(\mathbb{R}^N\right),\ |\nabla u|\in  L^p\left(\mathbb{R}^N\right),\ F(u)\in  L^1\left(\mathbb{R}^N\right).
\end{equation*}
Then $u$ satisfies the following Pohozaev identity
\begin{equation*}
(N-p)\int_{\mathbb{R}^N}|\nabla u|^p dx=pN\int_{\mathbb{R}^N}F(u) dx.
\end{equation*}
\end{lemma}

\section{Apendix~B: The energy estimate}
For any $\epsilon>0$, we define
\begin{equation}\label{e8.4}
u_{\epsilon}(x):=\varphi_{\epsilon}(x)U_{\epsilon}(x),
\end{equation}
where $U_{\epsilon}(x):=U_{\epsilon,0}$ with $U_{\epsilon,0}$ defined in  (\ref{e2.1}) and $\varphi_{\epsilon}(x)\in C_0^{\infty}\left(\mathbb{R}^N\right)$ is a radial cut-off function such that $0\leq \varphi_{\epsilon}(x)\leq 1$ for any $x\in \mathbb{R}^N$, $\varphi_{\epsilon}(x)\equiv 1$ in $B_{\epsilon^{\alpha}}(0)$, $\varphi_{\epsilon}(x)\equiv 0$ in $\mathbb{R}^N \backslash B_{2\epsilon^{\alpha}}(0)$ and $|\nabla \varphi_{\epsilon}(x)|\lesssim \epsilon^{-\alpha}$ for any $x\in \mathbb{R}^N$ with $\alpha\in [0,1)$ specified later.

\begin{lemma}\label{lemN3.7}
Let $N\geq 2$ and $1<p<N$, then for any $\alpha\in[0, 1)$, $\|u_\epsilon\|_{p^*}^{p^*}=S^{\frac{N}{p}}+O\left(\epsilon^{\frac{N(1-\alpha)}{p-1}}\right)$,
\begin{equation}\label{e2.111}
\begin{split}
\|u_\epsilon\|_{q}^{q}&\thickapprox
\left\{
\begin{array}{ll}
\epsilon^{N-\frac{(N-p)q}{p}}, & \  \frac{N(p-1)}{N-p}<q<p^*,\\
\epsilon^{N-\frac{(N-p)q}{p}}\log(\epsilon^{\alpha-1}), & \ q=\frac{N(p-1)}{N-p},\\
\epsilon^{\frac{(N-p)q}{p(p-1)}+\alpha\left(N-\frac{(N-p)q}{p-1}\right)}, & \ 0< q<\frac{N(p-1)}{N-p}
\end{array}
\right.
\end{split}
\end{equation}
and
\begin{equation}\label{e2.64}
\begin{split}
\|\nabla u_{\epsilon}\|_{q}^{q}\leq
 \left\{
\begin{array}{ll}
S^{\frac{N}{p}}+O\left(\epsilon^{(1-\alpha)\frac{N-p}{p-1}}\right), & \  q=p,\\
O\left(\epsilon^{N-\frac{Nq}{p}}\right), & \  \frac{N(p-1)}{N-1}<q<p,\\
O\left(\epsilon^{N-\frac{Nq}{p}}\left|\log\epsilon\right|\right), & \ q=\frac{N(p-1)}{N-1},\\
O\left(\epsilon^{\frac{(N-p)q}{p(p-1)}+\alpha\left(N-\frac{(N-1)q}{p-1}\right)}\right), & \ 0< q<\frac{N(p-1)}{N-1}.
\end{array}
\right.
\end{split}
\end{equation}
\end{lemma}

\begin{proof}
The estimate of $\|u_\epsilon\|_{p^*}^{p^*}$ follows from direct calculations and
\begin{eqnarray*}
\|u_\epsilon\|_{p^*}^{p^*}=\|U_\epsilon\|_{p^*}^{p^*}-\int_{\mathbb{R}^N}\left(1-\varphi_\epsilon^{p^*}\right)\left(U_\epsilon\right)^{p^*}dx,
\end{eqnarray*}
the estimate of $\|u_\epsilon\|_{q}^{q}$ follows from direct calculations and
\begin{eqnarray*}
\int_{B_{\epsilon^\alpha}(0)}\left(U_\epsilon\right)^{q}dx\lesssim\|u_\epsilon\|_{q}^{q}\lesssim \int_{B_{2\epsilon^\alpha}(0)}\left(U_\epsilon\right)^{q}dx,
\end{eqnarray*}
and the estimate of $\|\nabla u_\epsilon\|_{q}^{q}$ follows from direct calculations and
\begin{equation}\label{e2.65}
(a+b)^q\leq a^q+b^q+C(a^{q-1}b+ab^{q-1}),\quad \forall a,b\geq 0\text{ and }\forall q\geq 1,
\end{equation}
where $C>0$ is a constant.
\end{proof}

With the fundamental estimates of $u_{\epsilon}$ in Lemma~\ref{lemN3.7}, we can prove the crucial energy estimate in Lemma~\ref{lem3.6}.

\vskip0.12in

\noindent\textbf{Proof of Lemma~\ref{lem3.6}:}\quad We only prove the conclusion for $N<p^2$ since the case $p^2\leq N$ has already been proved in \cite[Lemma~4.5]{Deng-Wu}.  Let $u_{a,\mu}^+$ be a positive and radially symmetric non-increasing ground-state solution of (\ref{e1.1}) obtained in Proposition \ref{pro3.1}.  For $\tau\geq 0$ and $\epsilon>0$, we define
\begin{equation}\label{e2.12}
V_{\epsilon,\tau}(x):=u_{a,\mu}^+(x)+\tau u_{\epsilon}(x)\ \text{and}
\ W_{\epsilon,\tau}(x):=\left(a^{-1}\|V_{\epsilon,\tau}\|_p\right)^{\frac{N-p}{p}}V_{\epsilon,\tau}\left(a^{-1}\|V_{\epsilon,\tau}\|_px\right).
\end{equation}
Then
\begin{equation*}
\|W_{\epsilon,\tau}\|_p^p=a^p,\ \|\nabla W_{\epsilon,\tau}\|_p^p=\|\nabla V_{\epsilon,\tau}\|_p^p,\ \|W_{\epsilon,\tau}\|_{p^*}^{p^*}=\|V_{\epsilon,\tau}\|_{p^*}^{p^*},
\end{equation*}
and
\begin{equation*}
\|W_{\epsilon,\tau}\|_{q_1}^{q_1}=\left(a\|V_{\epsilon,\tau}\|_p^{-1}\right)^{{q_1}(1-\gamma_{q_1})}\|V_{\epsilon,\tau}\|_{q_1}^{q_1}.
\end{equation*}
Since $W_{\epsilon,\tau}\in \mathcal{S}_a$,  by Proposition \ref{pro2.1}, there exists $t^-_{\mu}\left(W_{\epsilon,\tau}\right)>0$ such that $(W_{\epsilon,\tau})_{t^-_{\mu}(W_{\epsilon,\tau})}\in \mathcal{P}_{a,\mu}^{-}$.  Set $t^-_{\epsilon,\tau}:=t^-_{\mu}(W_{\epsilon,\tau})$ for the sake of simplicity, then
\begin{equation}\label{e2.13}
(t^-_{\epsilon,\tau})^p\|\nabla W_{\epsilon,\tau}\|_p^p=\mu\gamma_{q_1}(t^-_{\epsilon,\tau})^{q_1\gamma_{q_1}}\|W_{\epsilon,\tau}\|_{q_1}^{q_1}
+(t^-_{\epsilon,\tau})^{p^*}\|W_{\epsilon,\tau}\|_{p^*}^{p^*}.
\end{equation}
Since $W_{\epsilon,0}=u_{a,\mu}^+\in \mathcal{P}_{a,\mu}^{+}$, by Proposition \ref{pro2.1}, $t^-_{\epsilon,0}>1$.  On the other hand, by Lemma~\ref{lemN3.7} and (\ref{e2.13}), $t^-_{\epsilon,\tau}\to 0$ as $\tau\to +\infty$ uniformly for $\epsilon>0$ sufficiently small. Since  $t^-_{\epsilon,\tau}$ is unique by Proposition \ref{pro2.1}, it is standard to show that $t^-_{\epsilon,\tau}$ is continuous for  $\tau\geq 0$, which implies that there exists $\tau_{\epsilon}>0$ such that  $t^-_{\epsilon,\tau_{\epsilon}}=1$. Consequently,
\begin{equation*}
m^{-}(a,\mu)\leq  \sup_{\tau\geq 0}\Psi_{\mu}(W_{\epsilon,\tau})
\end{equation*}
for any $\epsilon$ small enough.  Since
\begin{equation}\label{e2.14}
\begin{split}
\Psi_{\mu}(W_{\epsilon,\tau})&=\frac{1}{p}\|\nabla W_{\epsilon,\tau}\|_p^p-\frac{\mu}{q_1}\|W_{\epsilon,\tau}\|_{q_1}^{q_1}-\frac{1}{p^*}\|W_{\epsilon,\tau}\|_{p^*}^{p^*}\\
&=\frac{1}{p}\|\nabla V_{\epsilon,\tau}\|_p^p-\frac{\mu}{q_1}(a\|V_{\epsilon,\tau}\|_p^{-1})^{{q_1}(1-\gamma_{q_1})}\|V_{\epsilon,\tau}\|_{q_1}^{q_1}
-\frac{1}{p^*}\|V_{\epsilon,\tau}\|_{p^*}^{p^*},
\end{split}
\end{equation}
by Lemma~\ref{lemN3.7} and the inequality
\begin{equation*}\label{e2.66}
(a+b)^q\geq a^q+b^q,\quad\forall a,b\geq 0\text{ and } \forall q\geq 1,
\end{equation*}
it is easy to see that there exist $\tau_0>0$ large enough and $\epsilon_0>0$ small enough such that
\begin{equation*}
m^{-}(a,\mu)<m^{+}(a,\mu)+\frac{1}{N}S^{\frac{N}{p}}
\end{equation*}
for $\tau<\frac{1}{\tau_0}$ and $\tau>\tau_0$ uniformly for $0<\epsilon<\epsilon_0$.
To estimate $\Psi_{\mu}(W_{\epsilon,\tau})$ for $\frac{1}{\tau_0}\leq \tau\leq \tau_0$, we note that $p^*\geq 3$ for $p\in (\sqrt{N},N)$.  Thus by using the  inequalities
\begin{equation*}\label{e2.20}
(a+b)^{r}\geq a^r+ra^{r-1}b+rab^{r-1}+b^r,\quad\forall a,b\geq 0\text{ and } \forall r\geq 3
\end{equation*}
and
\begin{equation*}\label{e2.71}
(a+b)^{r}\geq a^r+ra^{r-1}b,\quad\forall a,b\geq 0\text{ and } \forall r\geq 1,
\end{equation*}
we obtain that
\begin{equation}\label{e2.21}
\|V_{\epsilon,\tau}\|_{q_1}^{q_1}\geq \|u_{a,\mu}^+\|_{q_1}^{q_1}+q_1\tau\int_{\mathbb{R}^N}(u_{a,\mu}^+)^{q_1-1}u_{\epsilon}dx
\end{equation}
and
\begin{equation}\label{e2.22}
\begin{split}
\|V_{\epsilon,\tau}\|_{p^*}^{p^*}&\geq \|u_{a,\mu}^+\|_{p^*}^{p^*}+\tau^{p^*}\|u_{\epsilon}\|_{p^*}^{p^*}+p^*\tau\int_{\mathbb{R}^N}(u_{a,\mu}^+)^{p^*-1}u_{\epsilon}dx\\
&\quad+p^*\tau^{p^*-1}\int_{\mathbb{R}^N}u_{a,\mu}^+u_{\epsilon}^{p^*-1}dx.
\end{split}
\end{equation}
By using $u_{a,\mu}^+\thickapprox 1$ in $B_{2\epsilon^{\alpha}}(0)$ and (\ref{e2.111}), we obtain that
\begin{equation}\label{e2.30}
\int_{\mathbb{R}^N}u_{a,\mu}^+u_{\epsilon}^{p^*-1}dx \thickapprox \int_{\mathbb{R}^N}u_{\epsilon}^{p^*-1}dx \thickapprox\epsilon^{\frac{N-p}{p}}.
\end{equation}
For $\|\nabla V_{\epsilon,\tau}\|_{p}^{p}$ and $\|V_{\epsilon,\tau}\|_{p}^{p}$, we need to divide their estimates into three cases.

\vskip0.12in

\textbf{The case }$1<p\leq2$.

In this case, for any $\eta\in(1, p)$, there exists a constant $C>0$ such that
\begin{eqnarray}\label{e2.72}
(1+t^2+2t\cos \beta)^{\frac{p}{2}}\leq 1+t^p+pt\cos\beta+Ct^\eta
\end{eqnarray}
for all $t\geq0$ and $\beta\in\mathbb{R}$.  Thus, by \eqref{e2.72},
\begin{eqnarray}\label{e2.74}
\|\nabla V_{\epsilon,\tau}\|_{p}^{p}&\leq& \|\nabla u_{a,\mu}^+\|_{p}^{p}+\tau^{p}\|\nabla u_{\epsilon}\|_{p}^{p}+p\tau\int_{\mathbb{R}^N}|\nabla u_{a,\mu}^+|^{p-2}\nabla u_{a,\mu}^+\cdot \nabla u_{\epsilon}dx\notag\\
&&+C\tau^{\eta_1} \int_{\mathbb{R}^N}|\nabla u_{a,\mu}^+|^{p-\eta_1}|\nabla u_{\epsilon}|^{\eta_1}dx
\end{eqnarray}
and
\begin{eqnarray}\label{e2.75}
\|V_{\epsilon,\tau}\|_{p}^{p}&\leq& \|u_{a,\mu}^+\|_{p}^{p}+\tau^{p}\| u_{\epsilon}\|_{p}^{p}+p\tau\int_{\mathbb{R}^N}(u_{a,\mu}^+)^{p-1}u_{\epsilon}dx\notag\\
&&+C\tau^{\eta_2} \int_{\mathbb{R}^N}(u_{a,\mu}^+)^{p-\eta_2}u_{\epsilon}^{\eta_2}dx
\end{eqnarray}
with $\eta_1,\eta_2\in (1,p)$ specified later.  By the classical regularity theorem of the $p$-Laplacian equation (cf. \cite{Tolksdorf84}), $u_{a,\mu}^+\in C^{1, \gamma}_{loc}\left(\mathbb{R}^N\right)$ for some $\gamma\in(0, 1)$.  Thus, $u_{a,\mu}^+\thickapprox 1$ in $B_{2\epsilon^{\alpha}}(0)$ and $p<\frac{N(p-1)}{N-p}$ for $N<p^2$. By (\ref{e2.111}), we obtain that
\begin{equation}\label{e2.80}
\left\{\aligned
&\int_{\mathbb{R}^N}\left(u_{a,\mu}^+\right)^{p-\eta_2}u_{\epsilon}^{\eta_2}dx\thickapprox \epsilon^{\frac{(N-p)\eta_2}{p(p-1)}+\alpha\left(N-\frac{\eta_2(N-p)}{p-1}\right)},\\
&\|u_\epsilon\|_p^p\thickapprox\epsilon^{\frac{N-p}{p-1}+\frac{\alpha(p^2-N)}{p-1}},\\
&\int_{\mathbb{R}^N}\left(u_{a,\mu}^+\right)^{p-1}u_{\epsilon}dx\thickapprox\epsilon^{\frac{N-p}{p(p-1)}+\alpha\left(\frac{p-N}{p-1}+N\right)}
 \endaligned\right.
\end{equation}
By (\ref{e2.75})-(\ref{e2.80}), $\|u_{a,\mu}^+\|_{p}^{p}=a^p$ and $0\leq \alpha<1$, we obtain that
\begin{equation}\label{e2.83}
\|V_{\epsilon,\tau}\|_{p}^{p}\leq a^{p}+p\tau\int_{\mathbb{R}^N}\left(u_{a,\mu}^+\right)^{p-1}u_{\epsilon}dx+M_1(\epsilon),
\end{equation}
where
\begin{equation*}
M_1(\epsilon):=O\left(\epsilon^{\frac{\eta_2(N-p)}{p(p-1)}+\alpha\left(\frac{\eta_2(p-N)}{p-1}+N\right)}\right)+O\left(\epsilon^{\frac{N-p}{p-1}+\frac{\alpha(p^2-N)}{p-1}}\right),
\end{equation*}
which, together with (\ref{e2.80}) once more and \eqref{e2.83}, implies that
\begin{eqnarray}\label{e2.84}
\left(\frac{a^p}{\|V_{\epsilon,\tau}\|_p^p}\right)^{\frac{q_1(1-\gamma_{q_1})}{p}}
&\geq& \frac{1}{\left(1+\frac{p\tau}{a^p}\int_{\mathbb{R}^N}\left(u_{a,\mu}^+\right)^{p-1}u_{\epsilon}dx
+M_1(\epsilon)\right)^{\frac{q_1(1-\gamma_{q_1})}{p}}}\notag\\
&=&1-\frac{\tau q_1(1-\gamma_{q_1})}{a^p}\int_{\mathbb{R}^N}\left(u_{a,\mu}^+\right)^{p-1}u_{\epsilon}dx
+\tilde{M}_1(\epsilon),
\end{eqnarray}
where $\tilde{M}_1(\epsilon)=M_1(\epsilon)+O\left(\epsilon^{\frac{2(N-p)}{p(p-1)}+2\alpha\left(\frac{p-N}{p-1}+N\right)}\right)$.  Since $u_{a,\mu}^+\in C^{1, \gamma}_{loc}\left(\mathbb{R}^N\right)$ for some $\gamma\in(0, 1)$, if $p\leq 2-\frac{1}{N}$, which implies that $\frac{N(p-1)}{N-1}\leq 1<p$, then by $\eta_1\in(1,p)$ and (\ref{e2.64}), we have
\begin{equation}\label{e2.85}
\begin{split}
 \int_{\mathbb{R}^N}|\nabla u_{a,\mu}^+|^{p-\eta_1}|\nabla u_{\epsilon}|^{\eta_1}dx\lesssim \int_{\mathbb{R}^N}|\nabla u_{\epsilon}|^{\eta_1}dx\lesssim
 \epsilon^{N-\frac{N\eta_1}{p}}
\end{split}
\end{equation}
while if $p>2-\frac{1}{N}$, which implies that $1<\frac{N(p-1)}{N-1}<p$, then by choosing $\eta_1=\frac{N(p-1)}{N-1}$ and (\ref{e2.64}), we have
\begin{equation}\label{e2.86}
\begin{split}
\int_{\mathbb{R}^N}|\nabla u_{a,\mu}^+|^{p-\eta_1}|\nabla u_{\epsilon}|^{\eta_1}dx\lesssim \int_{\mathbb{R}^N}|\nabla u_{\epsilon}|^{\eta_1}dx\lesssim
 \epsilon^{N-\frac{N^2(p-1)}{p(N-1)}}\left|\log \epsilon\right|.
\end{split}
\end{equation}
Noting that $u_{a,\mu}^+$ satisfies the equation
\begin{equation*}
-\Delta_pu=\lambda |u|^{p-2}u+\mu |u|^{q_1-2}u+|u|^{p^*-2}u
\end{equation*}
with $\lambda<0$ and $\lambda a^p=\mu(\gamma_{q_1}-1)\|u_{a,\mu}^+\|_{q_1}^{q_1}$.  Thus, by inserting \eqref{e2.21}, \eqref{e2.22}, \eqref{e2.30}, (\ref{e2.74}), (\ref{e2.84}), (\ref{e2.85}), (\ref{e2.86})
and Lemma~\ref{lemN3.7} into $\Psi_{\mu}(W_{\epsilon,\tau})$ and using \eqref{e2.14}, we have
\begin{eqnarray*}%\label{e2.87}
\Psi_{\mu}(W_{\epsilon,\tau})
&\leq& m^+(a,\mu)+ \left(\frac{1}{p}\tau^{p}-\frac{1}{p^*}\tau^{p^*}\right)S^{\frac{N}{p}}+O\left(\epsilon^{\frac{(1-\alpha)(N-p)}{p-1}}\right)\\
&&+O\left(\epsilon^{N-\frac{N\eta_1}{p}}\right)+O\left(\epsilon^{\frac{\eta_2(N-p)}{p(p-1)}+\alpha\left(\frac{\eta_2(p-N)}{p-1}+N\right)}\right)-C\epsilon^{\frac{N-p}{p}}
\end{eqnarray*}
for $1<p\leq 2-\frac{1}{N}$ and
\begin{eqnarray*}%\label{e2.87}
\Psi_{\mu}(W_{\epsilon,\tau})
&\leq& m^+(a,\mu)+ \left(\frac{1}{p}\tau^{p}-\frac{1}{p^*}\tau^{p^*}\right)S^{\frac{N}{p}}+O\left(\epsilon^{\frac{(1-\alpha)(N-p)}{p-1}}\right)\\
&&+O\left(\epsilon^{N-\frac{N^2(p-1)}{p(N-1)}}\left|\log \epsilon\right|\right)+O\left(\epsilon^{\frac{\eta_2(N-p)}{p(p-1)}+\alpha\left(\frac{\eta_2(p-N)}{p-1}+N\right)}\right)-C\epsilon^{\frac{N-p}{p}}
\end{eqnarray*}
for $p>2-\frac{1}{N}$.  Since $p\in (\sqrt{N},N)$, we have $\frac{N-p}{p}<N-\frac{N}{p}$, $\frac{N-p}{p}<N-\frac{N^2(p-1)}{p(N-1)}$  and $\frac{N-p}{p}<\frac{N-p}{p-1}+\alpha\left(\frac{p(p-N)}{p-1}+N\right)$.  Thus, by choosing $\alpha\in [0,\frac{1}{p})$, $\eta_1=1+\beta_1$ with $\beta_1>0$ small enough for $1<p\leq 2-\frac{1}{N}$ and $\eta_2=p-\beta_2$ with $\beta_2>0$ small enough for all $1<p\leq 2$, we obtain  that
\begin{eqnarray*}
\left\{\aligned
&\frac{N-p}{p}<\frac{(1-\alpha)(N-p)}{p-1},\\
&\frac{N-p}{p}<N-\frac{N\eta_1}{p},\quad\text{ for }1<p\leq 2-\frac{1}{N},\\
&\frac{N-p}{p}<\frac{\eta_2(N-p)}{p(p-1)}+\alpha\left(\frac{\eta_2(p-N)}{p-1}+N\right),
\endaligned\right.
\end{eqnarray*}
which implies that
\begin{equation*}
\Psi_{\mu}(W_{\epsilon,\tau})< m^+(a,\mu)+ (\frac{1}{p}\tau^{p}-\frac{1}{p^*}\tau^{p^*})S^{\frac{N}{p}}\leq m^{+}(a,\mu)+\frac{1}{N}S^{\frac{N}{p}}
\end{equation*}
for $\frac{1}{\tau_0}\leq \tau\leq \tau_0$ uniformly for $0<\epsilon<\epsilon_0$ small enough.

\vskip0.12in

\textbf{The case} $2<p<3$.

The estimate in this case is similar to that of $1<p\leq 2$, so we only sketch it.  We first notice that \eqref{e2.72} holds true for any $\eta\in [p-1,2]$ in this case.  Moreover, if $p\leq \frac{3N}{N+2}$, which implies that $p-1<\frac{N(p-1)}{N-p}\leq 2$, then by choosing $\eta_2=\frac{N(p-1)}{N-p}$,  we obtain that
\begin{equation}\label{e2.90}
\int_{\mathbb{R}^N}\left(u_{a,\mu}^+\right)^{p-\eta_2}u_{\epsilon}^{\eta_2}dx\thickapprox \epsilon^{N-\frac{N(p-1)(N-p)}{p(N-p)}}\left|\log\epsilon\right|
\end{equation}
while if $p>\frac{3N}{N+2}$, which implies that $p-1<2<\frac{N(p-1)}{N-p}$, then by choosing $\eta_2=2$,  we obtain that
\begin{equation}\label{e9.4}
\int_{\mathbb{R}^N}\left(u_{a,\mu}^+\right)^{p-\eta_2}u_{\epsilon}^{\eta_2}dx\thickapprox \epsilon^{\frac{2(N-p)}{p(p-1)}+\alpha\left(N-\frac{2(N-p)}{p-1}\right)}.
\end{equation}
Thus, by \eqref{e2.90} and \eqref{e9.4}, we improve \eqref{e2.84} by replacing $M_1(\epsilon)$ by
\begin{eqnarray*}
\hat{M}_1(\epsilon)=O\left(\epsilon^{\frac{N-p}{p-1}+\frac{\alpha(p^2-N)}{p-1}}\right)+\left\{\aligned
&O\left(\epsilon^{N-\frac{N(p-1)(N-p)}{p(N-p)}}\left|\log\epsilon\right|\right), \quad p\leq \frac{3N}{N+2},\\
&O\left(\epsilon^{\frac{2(N-p)}{p(p-1)}+\alpha\left(N-\frac{2(N-p)}{p-1}\right)}\right), \quad p>\frac{3N}{N+2}.
\endaligned
\right.
\end{eqnarray*}
If $p\leq 3-\frac{2}{N}$, which implies that $p-1<\frac{N(p-1)}{N-1}\leq 2$, then by choosing $\eta_1=\frac{N(p-1)}{N-1}$, we still have \eqref{e2.86} while if $p>3-\frac{2}{N}$, which implies that $p-1<2<\frac{N(p-1)}{N-1}$, then by choosing $\eta_1=2$ and (\ref{e2.64}), we have
\begin{equation*}
\int_{\mathbb{R}^N}|\nabla u_{a,\mu}^+|^{p-\eta_1}|\nabla u_{\epsilon}|^{\eta_1}dx\lesssim
\epsilon^{\frac{2(N-p)}{p(p-1)}+\alpha\left(N-\frac{2(N-1)}{p-1}\right)}.
\end{equation*}
Using these new estimates in $\Psi_{\mu}(W_{\epsilon,\tau})$, we also have
\begin{equation*}
\begin{split}
\Psi_{\mu}(W_{\epsilon,\tau})< m^+(a,\mu)+ \left(\frac{1}{p}\tau^{p}-\frac{1}{p^*}\tau^{p^*}\right)S^{\frac{N}{p}}\leq m^{+}(a,\mu)+\frac{1}{N}S^{\frac{N}{p}}
\end{split}
\end{equation*}
for $\frac{1}{\tau_0}\leq \tau\leq \tau_0$ uniformly for $0<\epsilon<\epsilon_0$ small enough by choosing $\alpha\in \left[0,\frac{1}{p}\right)$.

\vskip0.12in

\textbf{The case} $p\geq 3$.

In this case, there exists a constant $C>0$ such that
\begin{equation}\label{e2.23}
(1+t^2+2t\cos \beta)^{\frac{p}{2}}\leq 1+t^p+pt\cos\beta+C(t^2+t^{p-1})
\end{equation}
for all $t\geq0$ and $\beta\in\mathbb{R}$.  Thus, by replacing \eqref{e2.72} by \eqref{e2.23} in this case, we have
\begin{eqnarray*}
\|\nabla V_{\epsilon,\tau}\|_{p}^{p}&\leq& \|\nabla u_{a,\mu}^+\|_{p}^{p}+\tau^{p}\|\nabla u_{\epsilon}\|_{p}^{p}+p\tau\int_{\mathbb{R}^N}|\nabla u_{a,\mu}^+|^{p-2}\nabla u_{a,\mu}^+\cdot \nabla u_{\epsilon}dx\\
&&+C\tau^2 \int_{\mathbb{R}^N}|\nabla u_{a,\mu}^+|^{p-2}|\nabla u_{\epsilon}|^2dx+C\tau^{p-1} \int_{\mathbb{R}^N}|\nabla u_{a,\mu}^+||\nabla u_{\epsilon}|^{p-1}dx
\end{eqnarray*}
and
\begin{eqnarray*}
\|V_{\epsilon,\tau}\|_{p}^{p}&\leq& \|u_{a,\mu}^+\|_{p}^{p}+\tau^{p}\|u_{\epsilon}\|_{p}^{p}+p\tau\int_{\mathbb{R}^N}(u_{a,\mu}^+)^{p-1}u_{\epsilon}dx\\
&&+C\tau^2 \int_{\mathbb{R}^N}(u_{a,\mu}^+)^{p-2}u_{\epsilon}^2dx+C\tau^{p-1} \int_{\mathbb{R}^N}u_{a,\mu}^+u_{\epsilon}^{p-1}dx.
\end{eqnarray*}
Now, by using Lemma~\ref{lemN3.7} in estimating $\Psi_{\mu}(W_{\epsilon,\tau})$ for fixed $\eta_1$ and $\eta_2$ in this case, we arrive at
\begin{eqnarray*}
\Psi_{\mu}(W_{\epsilon,\tau})
&\leq& m^+(a,\mu)+ \left(\frac{1}{p}\tau^{p}-\frac{1}{p^*}\tau^{p^*}\right)S^{\frac{N}{p}}+O\left(\epsilon^{(1-\alpha)\frac{N-p}{p-1}}\right)\\
&&+O\left(\epsilon^{\frac{2(N-p)}{p(p-1)}+\frac{\alpha(Np+2-3N)}{p-1}}\right)+O\left(\epsilon^{\frac{N-p}{p}+\alpha}\right)-C\epsilon^{\frac{N-p}{p}}.
\end{eqnarray*}
Since $p\geq 3$, by choosing $\alpha\in \left(\frac{(N-p)(p-3)}{p(Np-3N+2)},\frac{1}{p}\right)$, we have
\begin{eqnarray*}
\left\{\aligned
&\frac{N-p}{p}<(1-\alpha)\frac{N-p}{p-1},\\
&\frac{N-p}{p}<\frac{2(N-p)}{p(p-1)}+\alpha \frac{Np+2-3N}{p-1},\\
&\frac{N-p}{p}<\frac{N-p}{p}+\alpha,
\endaligned\right.
\end{eqnarray*}
which implies that
\begin{equation*}
\Psi_{\mu}(W_{\epsilon,\tau})< m^+(a,\mu)+ \left(\frac{1}{p}\tau^{p}-\frac{1}{p^*}\tau^{p^*}\right)S^{\frac{N}{p}}\leq m^{+}(a,\mu)+\frac{1}{N}S^{\frac{N}{p}}
\end{equation*}
for $\frac{1}{\tau_0}\leq \tau\leq \tau_0$ uniformly for $0<\epsilon<\epsilon_0$ small enough.  It completes the proof.
\hfill$\Box$

%\bigskip
%\textbf{Conflict of interest}. The authors state no conflict of interest.



\end{document}